\documentclass[11pt]{article}
\usepackage{amsfonts, amsmath, amssymb, latexsym, eucal, amscd} 
\usepackage{cite}
\usepackage[all]{xy}
\newtheorem{theorem}{Theorem}[section]
\newtheorem{lemma}[theorem]{Lemma}
\newtheorem{prop}[theorem]{Proposition}
\newtheorem{corollary}[theorem]{Corollary}
\newtheorem{exAux}[theorem]{Example}
\newenvironment{example}{\begin{exAux} \rm}{\end{exAux}}
\newtheorem{Def}[theorem]{Definition}
\newenvironment{defi}{\begin{Def} \rm}{\end{Def}}
\newtheorem{Note}[theorem]{Note}
\newenvironment{note}{\begin{Note} \rm}{\end{Note}}
\newtheorem{Problem}[theorem]{Problem}

\newtheorem{Rem}[theorem]{Remark}

\newtheorem{Not}[theorem]{Notation}
\newenvironment{notation}{\begin{Not} \rm}{\end{Not}}
\newtheorem{Conj}[theorem]{Conjecture}

\newtheorem{Ass}[theorem]{Assumption}

\newenvironment{proof}{\medskip\noindent{\bf Proof.\ }}{\qed\medskip}
\newenvironment{proofof}[1]{\medskip\noindent{\bf Proof  of {#1}.\ 
}}{\qed\medskip}
\newcommand{\qed}{\hfill\mbox{$\Box$\qquad\qquad}}

\newcommand{\F}{\mathbb{F}}

\renewcommand{\S}{\mathbb{S}}
\newcommand{\cS}{\mathcal{S}}
\newcommand{\Mat}{\text{\rm Mat}}
\newcommand{\Matd}{\text{\rm Mat}_{d+1}(\F)}
\newcommand{\End}{\text{\rm End}(V)}
\newcommand{\vphi}{\varphi}
\renewcommand{\th}{\theta}

\renewcommand{\b}[1]{\langle #1 \rangle }
\newif\ifDRAFT

\begin{document}

\title{Compatibility and companions for Leonard pairs}

\author{Kazumasa Nomura and Paul Terwilliger}

\maketitle

\medskip

\begin{center}
   \bf  Abstract.
\end{center}

\begin{quote}
In this paper, we introduce the concepts of compatibility and companion
for Leonard pairs.
These concepts are roughly described as follows.
Let $\F$ denote a field, and let $V$ denote a vector space over $\F$
with finite positive dimension.
A Leonard pair on $V$ is an ordered pair of diagonalizable $\F$-linear maps
$A : V \to V$ and $A^* : V \to V$ that each act in an irreducible tridiagonal fashion
on an eigenbasis for the other one.
Leonard pairs $A,A^*$ and $B,B^*$ on $V$ are said to be compatible whenever 
$A^* = B^*$ and $[A,A^*] = [B,B^*]$, where $[r,s] = r s - s r$.
For a Leonard pair $A,A^*$ on $V$, by a companion of $A,A^*$ we mean
an $\F$-linear map $K: V \to V$ such that $K$ is a polynomial in $A^*$ and
$A-K, A^*$ is a Leonard pair on $V$.
The concepts of compatibility and companion are related as follows.
For compatible Leonard pairs $A,A^*$ and $B,B^*$ on $V$,
define $K = A-B$. Then $K$ is a companion of $A,A^*$.
For a Leonard pair $A,A^*$ on $V$ and a companion $K$ of $A,A^*$,
define $B = A-K$ and $B^* = A^*$.
Then $B,B^*$ is a Leonard pair on $V$ that is compatible with $A,A^*$.
Let $A,A^*$ denote a Leonard pair on $V$.
We  find all the Leonard pairs $B, B^*$ on $V$ that are compatible with $A,A^*$.
For each solution $B, B^*$ we describe the corresponding companion $K = A-B$.
\end{quote}

\section{Introduction}
\label{sec:intro}
\ifDRAFT {\rm sec:intro}. \fi

The notion of a Leonard pair was introduced by the second author in \cite{T:Leonard}.
We will recall the definition after a few comments.
A square matrix is said to be tridiagonal whenever each nonzero entry lies on
the diagonal, the subdiagonal, or the superdiagonal.
A tridiagonal matrix is said to be irreducible whenever each entry on the
subdiagonal is nonzero and each entry on the superdiagonal is nonzero.
Let $\F$ denote a field, and let $V$ denote a vector space over $\F$
with finite positive dimension.
A Leonard pair on $V$ is an ordered pair of $\F$-linear maps $A : V \to V$ and
$A^* : V \to V$ 
that satisfy (i) and (ii) below:
\begin{itemize}
\item[\rm (i)]
there exists a basis for $V$ with respect to which the matrix representing $A$
is irreducible tridiagonal and the matrix representing $A^*$ is diagonal;
\item[\rm (ii)]
there exists a basis for $V$ with respect to which the matrix representing $A^*$
is irreducible tridiagonal and the matrix representing $A$ is diagonal.
\end{itemize}

We have some historical remarks about Leonard pairs.
The concept of a Leonard pair originated in Algebraic Combinatorics,
in the study of $Q$-polynomial distance-regular graphs \cite{BI, BCN,DKT}.
The origin story begins with the 1973 thesis of Philippe Delsarte \cite{Del}.
In that thesis, Delsarte showed that a $Q$-polynomial distance-regular graph
yields two sequences of orthogonal polynomials that are related by what
is now called Askey-Wilson duality \cite[p.\ 261]{T:qRacah}.
Motivated by Delsarte's thesis and Eiichi Bannai's lectures at Ohio State University,
Douglas Leonard showed in 1982 that the $q$-Racah polynomials give the most general
orthogonal polynomial system that satisfies Askey--Wilson duality \cite{L}.
In their 1984 book \cite[Theorem 5.1]{BI} Bannai and Ito give a comprehensive version
Leonard's theorem, that treats all the limiting cases.
This version gives a complete classification of the orthogonal polynomial systems
that satisfy Askey--Wilson duality.
It shows that the orthogonal polynomial systems that satisfy Askey--Wilson duality all come
from the terminating branch of the 
Askey scheme (see \cite{KLS} and \cite[Section 1]{T:LSnote}.)
The Leonard theorem \cite[Theorem 5.1]{BI} is a bit complicated.
To simplify and clarify the theorem, the second author introduced the
notion of a Leonard pair and Leonard system \cite{T:Leonard}.
The Leonard systems are classified up to isomorphism in \cite[Theorem 1.9]{T:Leonard}.
This result gives a linear algebraic version of Leonard's theorem.
For more information on Leonard pairs and orthogonal polynomials, see
\cite[Appendix A]{T:Leonard} and \cite{T:qRacah, T:survey}.

We just mentioned how Leonard pairs are related to orthogonal polynomials.
Leonard pairs have applications to many other areas of mathematics and physics,
such as 
Lie theory \cite{IT, NT:Krawt, BM, Hart, Hart2, IT2b},
quantum groups \cite{Al, BT, IT2, IT3, Bockting, AC, BockT, IRT},
spin models  \cite{Cur:spinLP, NT:spin, CN:spin, CauW},
double affine Hecke algebras \cite{NT:DAHA, H:DAHA, H:DAHA2,Lee, LeeT},
partially ordered sets \cite{Liu, T:poset, Wor, MT},
and exactly solvable models in statistical mechanics 
\cite{Bas1, Bas2, Bas3, Bas4, Bas5, Bas6, Bas7}.
For more information about Leonard pairs and related topics,
see  \cite{T:Introduction, T:Kyoto, NT:Krawt, NT:TB, Tanaka}.

Next we recall some basic facts about Leonard pairs.
Let $A,A^*$ denote a Leonard pair on $V$.
By the construction, each of $A$ and $A^*$ is diagonalizable.
By \cite[Lemma 1.3]{T:Leonard} the eigenspaces of $A$ and $A^*$ all have dimension one.
Let $d+1$ denote the dimension of $V$, and
let $\{\th_i\}_{i=0}^d$ denote an ordering of the eigenvalues of $A$.
For $0 \leq i \leq d$ let $v_i$ denote an eigenvector for $A$ corresponding to $\th_i$.
The ordering $\{\th_i\}_{i=0}^d$ is said to be standard whenever the matrix
representing $A^*$ with respect to the basis $\{v_i\}_{i=0}^d$ is irreducible tridiagonal.
For a standard ordering $\{\th_i\}_{i=0}^d$ of the eigenvalues of $A$,
the ordering $\{\th_{d-i}\}_{i=0}^d$ is standard and no further ordering is
standard.
Similar comments apply to the orderings of the eigenvalues for $A^*$.
The Leonard pair $A,A^*$ is often described using some data
called a parameter array \cite[Definition 17.1]{T:qRacah}.
This is a sequence of scalars
\begin{equation}
   ( \{\th_i\}_{i=0}^d; \{\th^*_i\}_{i=0}^d; \{\vphi_i\}_{i=1}^d; \{\phi_i\}_{i=1}^d)   \label{eq:parray00}
\end{equation}
such that:
(i)
there exists a basis for $V$ with respect to which 
the matrices representing $A$ and $A^*$ are
\begin{align*}
A &:
\begin{pmatrix}
 \th_0 & & & & & \text{\bf 0}  \\
 1 & \th_1 \\
    & 1 & \th_2 \\
    &    &  \cdot & \cdot \\
    &     &          & \cdot & \cdot \\
\text{\bf 0} & & & & 1 & \th_d
\end{pmatrix},
&
A^* &:
\begin{pmatrix}
\th^*_0 & \vphi_1 & & & & {\bf 0}  \\
          & \th^*_1 & \vphi_2 \\
         &  &  \th^*_1 &  \cdot    \\
         &  &  &  \cdot & \cdot  \\
         &  &  &  &  \cdot & \vphi_d  \\
\text{\bf 0} & & & & & \th^*_d
\end{pmatrix};
\end{align*}
(ii)
there exists a basis for $V$
with respect to which the matrices representing $A$ and $A^*$ are
\begin{align*}
A &:
\begin{pmatrix}
 \th_d & & & & & \text{\bf 0}  \\
 1 & \th_{d-1} \\
    & 1 & \th_{d-2} \\
    &    &  \cdot & \cdot \\
    &     &          & \cdot & \cdot \\
\text{\bf 0} & & & & 1 & \th_0
\end{pmatrix},
&
A^* &:
\begin{pmatrix}
\th^*_0 & \phi_1 & & & & {\bf 0}  \\
          & \th^*_1 & \phi_2 \\
         &  &  \th^*_1 &  \cdot    \\
         &  &  &  \cdot & \cdot  \\
         &  &  &  &  \cdot & \phi_d  \\
\text{\bf 0} & & & & & \th^*_d
\end{pmatrix}.
\end{align*}
We are using the description in \cite[Theorem 18.1]{T:LSnote}.
For the above parameter array 
the sequence $\{\th_i\}_{i=0}^d$ (resp.\ $\{\th^*_i\}_{i=0}^d$) is an ordering
of the eigenvalues of $A$ (resp.\ $A^*$).
These orderings are standard \cite[Theorem 3.2]{T:Leonard}.
We comment on the uniqueness of the parameter array.
Consider a parameter array \eqref{eq:parray00} of $A,A^*$.
Then by \cite[Theorem 1.11]{T:Leonard}
each of the following is a parameter array of $A,A^*$:
\begin{align*}
&  ( \{\th_i\}_{i=0}^d; \{\th^*_i\}_{i=0}^d; \{\vphi_i\}_{i=1}^d; \{\phi_i\}_{i=1}^d),
&&   ( \{\th_i\}_{i=0}^d; \{\th^*_{d-i}\}_{i=0}^d; \{\phi_{d-i+1}\}_{i=1}^d; \{\vphi_{d-i+1}\}_{i=1}^d),
\\
&  ( \{\th_{d-i}\}_{i=0}^d; \{\th^*_i\}_{i=0}^d; \{\phi_i\}_{i=1}^d; \{\vphi_i\}_{i=1}^d),
&&  ( \{\th_{d-i}\}_{i=0}^d; \{\th^*_{d-i}\}_{i=0}^d; \{\vphi_{d-i+1}\}_{i=1}^d; \{\phi_{d-i+1}\}_{i=1}^d).
\end{align*}
Moreover, $A,A^*$ has no further parameter array.
By \cite[Lemma 12.4]{T:TDD} two Leonard pairs over $\F$ are isomorphic
if and only if they have a common parameter array.
Now consider a parameter array \eqref{eq:parray00} of $A,A^*$.
By \cite[Theorem 1.9]{T:Leonard} the expressions
\begin{align}
&  \frac{\th_{i-2} - \th_{i+1} } { \th_{i-1} - \th_i },
&&  \frac{\th^*_{i-2} - \th^*_{i+1} } { \th^*_{i-1} - \th^*_i }            \label{eq:indep0}
\end{align}
are equal and independent of $i$ for $2 \leq i \leq d-1$.
For the rest of this paragraph, assume that $d \geq 3$.
Let $\beta+1$ denote the common value of \eqref{eq:indep0}.
The scalar $\beta$ is called the fundamental constant of $A,A^*$.
For a parameter array of $A,A^*$ the entries
satisfy numerous relations \cite[Theorem 1.9]{T:Leonard}.
In \cite{T:parray}
the solutions are given in closed form,
in terms of seven free variables in addition to $d$ and $\beta$.
These seven variables are called basic.
The closed forms depend on the nature of $\beta$, as we now describe.
The Leonard pair $A,A^*$ is said to have type I whenever $\beta \neq \pm 2$;
type II whenever $\beta = 2$ and $\text{\rm Char}(\F) \neq 2$;
type III$^+$ whenever $\beta=-2$, $\text{\rm Char}(\F) \neq 2$, and $d$ is even;
type III$^-$ whenever $\beta=-2$, $\text{\rm Char}(\F) \neq 2$, and $d$ is odd;
type IV whenever $\beta=2$ and $\text{\rm Char}(\F)=2$.
For each type, the solutions are given in \cite[Section 5]{T:parray}.

We now describe our goals for the present paper.
We introduce the concepts of compatibility and companion for Leonard pairs.
These concepts are described as follows.
Leonard pairs $A,A^*$ and $B,B^*$ on $V$ are said to be compatible whenever 
$A^* = B^*$ and $[A,A^*] = [B,B^*]$, where $[r,s] = r s - s r$.
For a Leonard pair $A,A^*$ on $V$, by a companion of $A,A^*$ we mean
an $\F$-linear map $K: V \to V$ such that $K$ is a polynomial in $A^*$ and
$A-K, A^*$ is a Leonard pair on $V$.
The concepts of compatibility and companion are related as follows.
For compatible Leonard pairs $A,A^*$ and $B,B^*$ on $V$,
define $K = A-B$. Then $K$ is a companion of $A,A^*$.
For a Leonard pair $A,A^*$ on $V$ and a companion $K$ of $A,A^*$,
define $B = A-K$ and $B^* = A^*$.
Then $B,B^*$ is a Leonard pair on $V$ that is compatible with $A,A^*$.
Let $A,A^*$ denote a Leonard pair on $V$.
In this paper, we find all the Leonard pairs $B,B^*$ on $V$
that are compatible with $A,A^*$.
For each solution $B,B^*$ we describe the corresponding companion $K = A-B$.

We will describe our main results after a few comments.
Consider a Leonard pair  $A,A^*$ on $V$ with a parameter array \eqref{eq:parray00}.
We will show that for $d \geq 3$, the scalar
\begin{align*}
 \kappa &= (\th_{i-1}-\th_{i+1})^2 + (\beta+2)(\th_i - \th_{i-1})(\th_i - \th_{i+1}) 
\end{align*}
is independent of $i$ for $1 \leq i \leq d-1$.
We call $\kappa$ the invariant value for $A,A^*$.

As we will see, every Leonard pair can be represented by an ordered pair of
matrices, such that the first matrix is irreducible tridiagonal with all entries $1$ on the subdiagonal,
and the second matrix is diagonal.
Motivated by this fact, we consider the following setup.
Fix a diagonal matrix
\[
  A^* = \text{\rm diag}(\th^*_0, \th^*_1, \ldots, \th^*_d)
\]
with $\{\th^*_i\}_{i=0}^d$ mutually distinct scalars in $\F$.
An irreducible tridiagonal matrix is said to be normalized whenever it has all entries $1$ on the
subdiagonal.
Let the set $\Omega$ consist of  the normalized irreducible tridiagonal matrices $A$
such that $A,A^*$ is a Leonard pair on $\F^{d+1}$.
For the moment let $A \in \Omega$.
As we will see in Lemma \ref{lem:Omega00}, 
the set $\Omega$ contains all the matrices $B$ such that $B,A^*$ is a Leonard pair
on $\F^{d+1}$ that is compatible with $A,A^*$.
For the rest of this section, fix matrices $A$ and $B$ in $\Omega$, and let
\begin{align*}
&  (\{\th_i\}_{i=0}^d; \{\th^*_i\}_{i=0}^d; \{\vphi_i\}_{i=1}^d; \{\phi_i\}_{i=1}^d),
&&  (\{\th'_i\}_{i=0}^d; \{\th^*_i\}_{i=0}^d; \{\vphi'_i\}_{i=1}^d; \{\phi'_i\}_{i=1}^d) 
\end{align*}
denote a parameter array of $A,A^*$ and $B,A^*$, respectively.

We now describe our first main result.
The Leonard pairs $A,A^*$ and $B,A^*$ are compatible if and only if
\begin{align*}
\vphi_i \phi_i &= \vphi'_i \phi'_i   && (1 \leq i \leq d).
\end{align*}
Our second main result is as follows.
For the case $d=1$, $A,A^*$ and $B,A^*$ are compatible if and only if
$\vphi_1 \phi_1 = \vphi'_1 \phi'_1$.
For the case $d=2$, $A,A^*$ and $B, A^*$ are compatible
if and only if $\vphi_1 \phi_1 = \vphi '_1 \phi'_1$ and
$\vphi_2 \phi_2 = \vphi'_2 \phi'_2$.
For the case $d \geq 3$, 
$A, A^*$ and $B,A^*$ are compatible if and only if
\begin{align}
\kappa &= \kappa',    &
\vphi_1 \phi_1 &= \vphi'_1 \phi'_1,  &
\vphi_d \phi_d &= \vphi'_d \phi'_d,           \label{eq:maincond}
\end{align}
where $\kappa$ (resp.\ $\kappa'$) is the invariant value for $A,A^*$ (resp.\ $B,A^*$). 
We now describe our further main results.
Assume that $d \geq 3$.
Note that $A,A^*$ and $B,A^*$ have the same fundamental constant
and the same type.
For each type,
we describe the conditions \eqref{eq:maincond} in terms of the basic variables for $A,A^*$ and $B,A^*$.
These descriptions are given in 
Theorems 
\ref{thm:type1main},
\ref{thm:type2main},
\ref{thm:type3+main},
\ref{thm:type3-main},
\ref{thm:type4main}.
We solve the resulting equations in terms of the basic variables of $A,A^*$.
Our solutions are listed in Theorems
\ref{thm:type1sol},
\ref{thm:type1sol2},
\ref{thm:type2sol},
\ref{thm:type3+sol},
\ref{thm:type3-sol},
\ref{thm:type4sol}.
For each solution we describe the corresponding companion $K = A-B$.
These descriptions can be found in Theorems
\ref{thm:ex2},
\ref{thm:ex3},
\ref{thm:ex2b},
\ref{thm:ex3b},
\ref{thm:type2ex2},
\ref{thm:type2ex1},
\ref{thm:type3+ex2},
\ref{thm:type3-ex2},
\ref{thm:type4K}.

We mention some examples of compatible Leonard pairs.
Let $A$, $B \in \Omega$ as above.
We mentioned earlier that $A,A^*$ and $B,A^*$ are compatible if and only if
$\vphi_i \phi_i = \vphi'_i \phi'_i$ for $1 \leq i \leq d$.
This condition is satisfied if one of the following {\rm (i)--(iv)} holds:
\begin{itemize}
\item[\rm (i)]
 $\vphi'_i = \vphi_i$ and $\phi'_i = \phi_i \quad (1 \leq i \leq d)$;
\item[\rm (ii)]
 $\vphi'_i = \phi_i$ and $\phi'_i = \vphi_i \quad (1 \leq i \leq d)$;
\item[\rm (iii)]
 $\vphi'_i = - \vphi_i$ and $\phi'_i = - \phi_i \quad (1 \leq i \leq d)$;
\item[\rm (iv)]
 $\vphi'_i = - \phi_i$ and $\phi'_i = - \vphi_i  \quad  (1 \leq i \leq d)$.
\end{itemize}
As we will see in Proposition \ref{prop:A=B}, the condition (i) or (ii)
holds if and only if there exists $\zeta \in \F$ such that $B = A + \zeta I$.
Here $I$ denotes the identity matrix.
Define $A^\vee = - \mathbb{S} A \mathbb{S}^{-1}$, where  $\mathbb{S}$ denotes the diagonal matrix that
has $(i,i)$-entry $(-1)^i$ for $0 \leq i \leq d$.
Then $A^\vee \in \Omega$ (see Lemma \ref{lem:veecoincide}.)
As we will see in Proposition \ref{prop:ABvee}, the condition (iii) or (iv)
holds if and only if there exists $\zeta \in \F$ such that
$B = A^\vee + \zeta I$.
In the main body of the paper, we interpret conditions (iii), (iv) using a
symmetric binary relation called the bond relation.

The paper is organized as follows.
In Section \ref{sec:pre} we fix our notation, and recall some materials from linear algebra.
In Sections \ref{sec:trid}, \ref{sec:normalize} we obtain some results about tridiagonal matrices
that will be used later in the paper.
In Section \ref{sec:normalize} we discuss the normalization of an irreducible tridiagonal matrix.
In Section  \ref{sec:bond} we introduce the bond relation.
In Section \ref{sec:LP}
we recall some basic facts about Leonard pairs.
In Section \ref{sec:bondLP} we apply the bond relation to Leonard pairs.
In Section \ref{sec:compatible}
we introduce the concepts of compatibility and companion for Leonard pairs.
In Section \ref{sec:Omega} we introduce the set $\Omega$.
In Section \ref{sec:bondOmega} we consider the bond relation on $\Omega$.
In Section \ref{sec:compOmega} we consider the compatibility relation on $\Omega$.
In Section \ref{sec:characterize}
we prove our first main result.
In Section \ref{sec:type} we describe some basic facts about the type of a Leonard pair.
In Section \ref{sec:refine} we display a formula that will be used in the proof of
our second main result.
In Section \ref{sec:typeO} we describe the companions of a Leonard pair for $d=1,2$.
In Sections \ref{sec:parraytype1}--\ref{sec:companiontype4} we consider Leonard pairs with $d \geq 3$.
For each type I--IV
we describe the parameter array in terms of the basic variables,
and prove the formula from Section 14
(Sections 
\ref{sec:parraytype1}, 
\ref{sec:parraytype2}, 
\ref{sec:parraytype3+}, 
\ref{sec:parraytype3-}, 
\ref{sec:parraytype4});
we represent condition \eqref{eq:maincond} in terms of the basic variables and 
give the solutions
(Sections 
\ref{sec:characterizetype1},
\ref{sec:characterizetype2},
\ref{sec:characterizetype3+},
\ref{sec:characterizetype3-},
\ref{sec:characterizetype4});
and we describe the companions of the given Leonard pair
(Sections
\ref{sec:companiontype1},
\ref{sec:companiontype2},
\ref{sec:companiontype3+},
\ref{sec:companiontype3-},
\ref{sec:companiontype4}).

\section{Preliminaries}
\label{sec:pre}
\ifDRAFT {\rm sec:pre}. \fi

The following notational conventions hold throughout the paper.
Let $\F$ denote a field.
Every vector space and algebra discussed in this paper is over $\F$.
Fix an integer $d \geq 0$.
The notation $\{x_i\}_{i=0}^d$ refers to the sequence $x_0,x_1,\ldots,x_d$.
Let $\Matd$ denote the algebra consisting of the $d+1$ by $d+1$
matrices that have all entries in $\F$.
We index the rows and columns by $0,1,\ldots, d$.
The identity element of $\Mat_{d+1}(\F)$ is denoted by $I$.
Let $\F^{d+1}$ denote the vector space consisting of the column
vectors with $d+1$ rows and all entries in $\F$.
We index the rows by $0,1,\ldots,d$.
The algebra $\Matd$ acts on $\F^{d+1}$ by left multiplication.
Let $V$ denote a vector space with dimension $d+1$.
Let $\End$ denote the algebra consisting of the
$\F$-linear maps $V \to V$.
The identity element of $\End$ is denoted by $I$.
We recall how each basis $\{v_i\}_{i=0}^d$ of $V$ gives an algebra
isomorphism $\text{\rm End}(V) \to \Mat_{d+1}(\F)$.
For $A \in \text{\rm End}(V)$ and $M \in \Mat_{d+1}(\F)$,
we say that {\em $M$ represents $A$ with respect to $\{v_i\}_{i=0}^d$}
whenever $A v_j = \sum_{i=0}^d M_{i,j} v_i$ for $0 \leq j \leq d$.
The isomorphism sends $A$ to the unique matrix in $\Mat_{d+1}(\F)$ that represents
$A$ with respect to $\{v_i\}_{i=0}^d$.
Let $A \in \End$.
By an {\em eigenspace of $A$}, we mean a subspace $W \subseteq V$ 
such that $W \neq 0$ and there exists $\theta \in \F$ such that $W = \{ v \in V \, |\, A v = \theta v\}$;
in this case $\theta$ is the {\em eigenvalue} of $A$ associated with $W$.
We say that $A$ is {\em diagonalizable} whenever $V$ is spanned by the eigenspaces of $A$.
We say that $A$ is {\em multiplicity-free} whenever $A$ is diagonalizable and its eigenspaces
all have dimension one.
Assume that $A$ is multiplicity-free.
Let $\{\th_i\}_{i=0}^d$ denote an ordering of the eigenvalues of $A$.
For $0 \leq i \leq d$ let $V_i$ denote the eigenspace of $A$ associated with $\th_i$,
and define $E_i \in \End$ such that $(E_i - I) V_i =0$ and $E_i V_j=0$ for $j \neq i$ $(0 \leq j \leq d)$.
We call $E_i$ the {\em primitive idempotent of $A$ associated with $\th_i$}.
We have
(i) $E_i E_j = \delta_{i,j} E_i$ $(0 \leq i,j \leq d)$;
(ii) $I = \sum_{i=0}^d E_i$;
(iii) $A E_i = \th_i E_i = E_i A$ $(0 \leq i \leq d)$;
(iv) $A = \sum_{i=0}^d \th_i E_i$;
(v) $V_i = E_i V$ $(0 \leq i \leq d)$;
(vi) $\text{\rm rank} (E_i) = 1$ $(0 \leq i \leq d)$;
(vii) $\text{\rm tr} (E_i) = 1$ $(0 \leq i \leq d)$,
where tr means trace.
Moreover
\begin{align*}
 E_i &= \prod_{ \scriptsize \begin{matrix}  { 0 \leq j \leq d}  \\  j \neq i \end{matrix} }
            \frac{ A - \th_j I } {\th_i - \th_j }
          && ( 0 \leq i \leq d).
\end{align*}
Let $\b{A}$ denote the subalgebra of $\text{\rm End}(V)$ generated by $A$.
The algebra $\b{A} $ is commutative.
The elements $\{A^i\}_{i=0}^d$ form a basis of $\b{ A }$
and $\prod_{i=0}^d (A - \th_i I) = 0$.
Moreover $\{E_i\}_{i=0}^d$ form a basis of $\b{ A }$.

\begin{lemma}    \label{lem:As2}    \samepage
\ifDRAFT {\rm lem:As2}. \fi
Assume that $A \in \End$ is multiplicity-free with primitive idempotents $\{E_i\}_{i=0}^d$.
Then for $H \in \End$ the following {\rm (i)--(iii)} are equivalent:
\begin{itemize}
\item[\rm (i)]
$H \in \b{ A }$;
\item[\rm (ii)]
$H$ commutes with $A$;
\item[\rm (iii)]
$H$ commutes with $E_i$ for $0 \leq i \leq d$.
\end{itemize}
\end{lemma}

\begin{proof}
This is a reformulation of the fact that 
a matrix $M \in \Matd$ commutes with each diagonal matrix in $\Matd$ 
if and only if $M$ diagonal.
\end{proof}

\begin{lemma}  {\rm (Skolem-Noether, see \cite[Corollary 7.125]{Rot}.) }
\label{lem:SN}    \samepage
\ifDRAFT {\rm lem:SN}. \fi
For a map $\sigma :\End \to \End$,
the following are equivalent:
\begin{itemize}
\item[\rm (i)]
$\sigma$ is an algebra isomorphism;
\item[\rm (ii)]
there exists an invertible $S \in \End$ such that
$X^\sigma = S X S^{-1}$ for all $X \in \End$.
\end{itemize}
\end{lemma}

\section{Tridiagonal matrices and diagonal equivalence}
\label{sec:trid}
\ifDRAFT {\rm sec:trid}. \fi

Recall the algebra $\Matd$.
In this section we describe an equivalence relation on $\Matd$ called
diagonal equivalence.
We investigate this equivalence relation on the set of irreducible tridiagonal 
matrices in $\Matd$.

\begin{defi}     \label{def:diagonalequiv}    \samepage
\ifDRAFT {\rm def:diagonalequiv}. \fi
Matrices $A$, $B$ in $\Matd$ are said to be {\em diagonally equivalent}
whenever there exists an invertible diagonal matrix $S \in \Matd$
such that $B = S A S^{-1}$.
\end{defi}

Note that diagonal equivalence is an equivalence relation on $\Matd$.

\begin{lemma}    \label{lem:trid1}     \samepage
\ifDRAFT {\rm lem:trid1}. \fi
For a matrix $A \in \Matd$ and an invertible diagonal matrix 
\[
   S = \text{\rm diag}(s_0, s_1, \ldots, s_d),
\]
the matrix $S A S^{-1}$ has $(i,j)$-entry
$s_i s_j^{-1} A_{i,j}$ for $0 \leq i,j \leq d$.
\end{lemma}

\begin{proof}
By matrix multiplication.
\end{proof}

\begin{corollary}    \label{cor:trid1}    \samepage
\ifDRAFT {\rm cor:trid1}. \fi
For diagonally equivalent matrices $A$, $B$ in $\Matd$,
\begin{align}
  A_{i,i} &= B_{i,i}   &&  (0 \leq i \leq d),                          \label{eq:tridAB1}
\\
 A_{i,j} A_{j,i} &= B_{i,j} B_{j,i}   &&   (0 \leq i,j \leq d).         \label{eq:tridAB2}
\end{align}
\end{corollary}

\begin{proof}
Use Lemma \ref{lem:trid1}.
\end{proof}

We have a comment about Corollary \ref{cor:trid1}.
Suppose that $A$, $B$ in $\Matd$ satisfy \eqref{eq:tridAB1}, \eqref{eq:tridAB2}.
It is natural to conjecture that $A$, $B$ are diagonally equivalent.
This conjecture is not true in general;
a counterexample is
\begin{align*}
A &=
\begin{pmatrix}
 0 & 1 & 2  \\
 1 & 0 & 1  \\
 2 & 1 & 0
\end{pmatrix},
&
B &=
\begin{pmatrix}
 0 & 1 & 4  \\
 1 & 0 & 1  \\
 1 & 1 & 0
\end{pmatrix}.
\end{align*}
We now consider a class of matrices for which the conjecture is true.

A matrix $M \in \Mat_{d+1}(\F)$ is said to be {\em tridiagonal} whenever
the $(i,j)$-entry $M_{i,j}=0$ if $|i-j|>1$ $(0 \leq i,j \leq d)$.
Assume that $M$ is tridiagonal.
Then $M$ is said to be {\em irreducible}
whenever $M_{i,j} \neq 0$ if $|i-j|=1$ $(0 \leq i,j \leq d)$.

Next we give a variation on Lemma \ref{lem:trid1}.

\begin{lemma}    \label{lem:trid5}    \samepage
\ifDRAFT {\rm lem:trid5}. \fi
For a tridiagonal matrix $A \in \Matd$ and an
invertible diagonal matrix $S = \text{\rm diag} (s_0, s_1, \ldots, s_d)$,
the matrix $S A S^{-1}$ is tridiagonal with entries
\[
\begin{array}{ccccc}
\text{\rm $(i,i)$-entry} & & \text{\rm $(i,i-1)$-entry} & & \text{\rm $(i-1,i)$-entry}
\\ \hline
A_{i,i} & & s_i s_{i-1}^{-1} A_{i,i-1} & & s_{i-1} s_i^{-1} A_{i-1,i}       \rule{0mm}{3ex}
\end{array}
\]
\end{lemma}

\begin{proof}
Use Lemma \ref{lem:trid1}.
\end{proof}

We emphasize a few points from Lemma \ref{lem:trid5}.

\begin{lemma}    \label{lem:trid5b}    \samepage
\ifDRAFT {\rm lem:trid5b}. \fi
For diagonally equivalent matrices $A$, $B$ in $\Matd$,
$A$ is tridiagonal if and only if $B$ is tridiagonal.
In this case, $A$ is irreducible if and only if $B$ is irreducible.
\end{lemma}

\begin{proof}
Use Lemmas \ref{lem:trid1} and \ref{lem:trid5}.
\end{proof}

We now establish the converse of Corollary \ref{cor:trid1} for irreducible tridiagonal
matrices.

\begin{lemma}    \label{lem:trid2}    \samepage
\ifDRAFT {\rm lem:trid2}. \fi
For irreducible tridiagonal matrices $A$, $B$ in $\Matd$,
assume that
\begin{align}
A_{i,i} &= B_{i,i}  \quad (0 \leq i \leq d),
&
A_{i,i-1} A_{i-1,i} &= B_{i,i-1} B_{i-1,i}  \quad (1 \leq i \leq d).   \label{eq:condAB}
\end{align}
Define a diagonal matrix $S \in \Matd$ that has diagonal entries
\begin{align}
  S_{i,i} &= \frac{ B_{1,0} B_{2,1} \cdots B_{i,i-1} } { A_{1,0} A_{2,1} \cdots A_{i,i-1} }   
                 &&          (0 \leq i \leq d).                    \label{eq:defS}
\end{align}
Then $B = S A S^{-1}$.
\end{lemma}

\begin{proof}
For $0 \leq i \leq d$ abbreviate $s_i = S_{i,i}$.
By \eqref{eq:condAB} and \eqref{eq:defS},
\begin{align}
  \frac{ s_i } { s_{i-1} } 
 &= \frac{ B_{i, i-1} } { A_{i,i-1} } = \frac{ A_{i-1, i} } { B_{i-1, i} }
  &&   (1 \leq i \leq d).                                            \label{eq:si}
\end{align}
The matrices  $S A S^{-1}$ and $B$ are tridiagonal.
By Lemma \ref{lem:trid1} and \eqref{eq:condAB},
\begin{align*}
(S A S^{-1})_{i,i} &= A_{i,i} = B_{i,i}    &&  (0 \leq i \leq d).
\end{align*}
By Lemma \ref{lem:trid1} and \eqref{eq:si},
\begin{align*}
(S A S^{-1})_{i, i-1} &= s_i s_{i-1}^{-1} A_{i, i-1} = B_{i, i-1}  &&  (1 \leq i \leq d),
\\
(S A S^{-1})_{i-1, i} &= s_{i-1} s_i^{-1} A_{i-1,i} = B_{i-1, i}  &&  (1 \leq i \leq d).
\end{align*}
By these comments, $S A S^{-1} = B$.
\end{proof}

\begin{lemma}    \label{lem:tridnew1}    \samepage
\ifDRAFT {\rm lem:tridnew1}. \fi
For an irreducible tridiagonal matrix $A \in \Matd$ and an invertible
diagonal matrix $S \in \Matd$,
define $B = S A S^{-1}$.
Then the following {\rm (i)--(v)} are equivalent:
\begin{itemize}
\item[\rm (i)]
$A_{i, i-1} = B_{i, i-1}$ for $1 \leq i \leq d$;
\item[\rm (ii)]
$A_{i-1, i} = B_{i-1, i}$ for $1 \leq i \leq d$;
\item[\rm (iii)]
$A-B$ is diagonal;
\item[\rm (iv)]
$A=B$;
\item[\rm (v)]
$S_{i,i} = S_{0,0}$ for $0 \leq i \leq d$.
\end{itemize}
\end{lemma}

\begin{proof}
Note that $B$ is irreducible tridiagonal.

(i) $\Leftrightarrow$ (ii)
By \eqref{eq:tridAB2}.

(i),(ii) $\Leftrightarrow$ (iii)
Since $B$ is tridiagonal.

(i),(ii) $\Rightarrow$ (iv)
By  \eqref{eq:tridAB1} and since $B$ is tridiagonal.

(iv) $\Rightarrow$ (iii)
Clear.

(i) $\Leftrightarrow$ (v)
For $0 \leq i \leq d$ abbreviate $s_i = S_{i,i}$.
By Lemma \ref{lem:trid5}, 
$B_{i,i-1} = s_i s_{i-1}^{-1} A_{i,i-1}$ for $1 \leq i \leq d$.
So (i) holds if and only if  $s_i s_{i-1}^{-1} = 1$  $(1 \leq i \leq d)$
if and only if $s_i =  s_0$ $(0 \leq i \leq d)$.
\end{proof}

\section{A normalization}
\label{sec:normalize} 
\ifDRAFT {\rm sec:normalize}. \fi

In this section we introduce a type of  irreducible tridiagonal matrix,
said to be normalized.

\begin{defi}    \label{def:normalized}    \samepage
\ifDRAFT {\rm def::normalized}. \fi
An irreducible tridiagonal matrix $A \in \Matd$ is said to be {\em normalized}
whenever $A_{i,i-1} = 1$ for $1 \leq i \leq d$.
\end{defi}

\begin{lemma}    \label{lem:trid4}    \samepage
\ifDRAFT {\rm lem:trid4}. \fi
Every irreducible tridiagonal matrix in $\Matd$ is diagonally equivalent
to a unique normalized irreducible tridiagonal matrix in $\Matd$.
\end{lemma}

\begin{proof}
Consider an irreducible tridiagonal matrix $A \in \Matd$.
We first show the existence of a normalized irreducible tridiagonal matrix in $\Matd$
that is diagonally equivalent to $A$.
Define an irreducible  tridiagonal matrix $B \in \Matd$ such that
$B_{i,i} = A_{i,i}$  for $0 \leq i \leq d$
and 
$B_{i,i-1} = 1$, $B_{i-1,i} = A_{i, i-1} A_{i-1, i}$ for $1 \leq i \leq d$.
By Lemma \ref{lem:trid2}, there exists an invertible diagonal matrix $S \in \Matd$
such that $B = S A S^{-1}$. 
Then $B$ is a normalized irreducible tridiagonal matrix that is diagonally equivalent to $A$.
We have shown the existence.
Next we show the uniqueness.
Consider normalized irreducible tridiagonal matrices $B_1$, $B_2$ in $\Matd$
each of which is diagonally equivalent to $A$.
Then $B_1$ and $B_2$ are diagonally equivalent.
We apply Lemma \ref{lem:tridnew1} to $B_1$ and $B_2$.
We have $(B_1)_{i,i-1} = (B_2)_{i,i-1}$ for $1 \leq i \leq d$.
So Lemma \ref{lem:tridnew1}(i) holds.
By this and  Lemma \ref{lem:tridnew1}(iv) we obtain $B_1 = B_2$.
We have shown the uniqueness.
\end{proof}

\section{The bond relation}
\label{sec:bond}
\ifDRAFT {\rm sec:bond}. \fi

In this section we introduce a symmetric binary relation
on the set of irreducible tridiagonal matrices in $\Matd$.  
We call this relation the bond relation.
To motivate things, we mention a variation on Lemma \ref{lem:tridnew1}.

\begin{lemma}    \label{lem:tridnew2}    \samepage
\ifDRAFT {\rm lem:tridnew2}. \fi
For an irreducible tridiagonal matrix $A \in \Matd$ and an invertible
diagonal matrix $S \in \Matd$,
define $B =S A S^{-1}$.
Then the following {\rm (i)--(iv)} are equivalent:
\begin{itemize}
\item[\rm (i)]
$A_{i, i-1} = -B_{i, i-1}$ for $1 \leq i \leq d$;
\item[\rm (ii)]
$A_{i-1, i} =  -B_{i-1, i}$ for $1 \leq i \leq d$;
\item[\rm (iii)]
$A+B$ is diagonal;
\item[\rm (iv)]
$S_{i,i} = (-1)^i S_{0,0}$ for $0 \leq i \leq d$.
\end{itemize}
Moreover, if {\rm (i)--(iv)} hold then $A+B$ has $(i,i)$-entry $2 A_{i,i}$
for $0 \leq i \leq d$.
\end{lemma}

\begin{proof}
Note by Lemma \ref{lem:trid5b} that $B$ is irreducible tridiagonal.

(i) $\Leftrightarrow$ (ii)
By \eqref{eq:tridAB2}.

(i),(ii) $\Leftrightarrow$ (iii)
Since $B$ is tridiagonal.

(i) $\Leftrightarrow$ (iv)
For $0 \leq i \leq d$ abbreviate $s_i = S_{i,i}$.
By Lemma \ref{lem:trid5},
$B_{i,i-1} = s_i s_{i-1}^{-1} A_{i,i-1}$ for $1 \leq i \leq d$.
So (i) holds if and only if 
 $s_i s_{i-1}^{-1} =  - 1$ $(1 \leq i \leq d)$
if and only if $s_i = (-1)^i s_0$ $(0 \leq i \leq d)$.

Suppose (i)--(iv) hold.
By Corollary \ref{cor:trid1} we have $A_{i,i} = B_{i,i}$ for $0 \leq i \leq d$.
So the matrix $A+B$ has $(i,i)$-entry $2 A_{i,i}$ for $0 \leq i \leq d$.
\end{proof}

In view of Lemma \ref{lem:tridnew2},
we make a definition.

\begin{defi}    \label{def:vee}    \samepage
\ifDRAFT {\rm def:vee}. \fi
For an irreducible tridiagonal matrix $A \in \Matd$,
define $A^\vee = - \S A \S^{-1}$,
where $\S \in \Matd$ is  diagonal with $(i,i)$-entry $(-1)^i$ for $0 \leq i \leq d$.
Note that $A^\vee$ is irreducible tridiagonal and diagonally equivalent to $-A$.
\end{defi} 

\begin{note}    \label{note:vee}    \samepage
\ifDRAFT {\rm note:vee}. \fi
For an irreducible tridiagonal matrix $A \in \Matd$,
$(A^\vee)^\vee=A$.
\end{note}

\begin{note}    \label{note:vee2}    \samepage
\ifDRAFT {\rm note:vee2}. \fi
Referring to Definition \ref{def:vee}, assume that $\text{\rm Char}(\F)=2$.
Then $\S = I$ and $A^\vee = A$.
\end{note}

\begin{defi}    \label{def:bond}    \samepage
\ifDRAFT {\rm def:bond}. \fi
Irreducible tridiagonal matrices $A$, $B$ in $\Matd$ are said to be
{\em bonded} whenever $B = A^\vee$.
\end{defi}

\begin{note}    \label{note:bond00}    \samepage
\ifDRAFT {\rm note:bond00}. \fi
The bond relation is a symmetric binary relation on the set 
of all irreducible tridiagonal matrices $\Matd$.
\end{note}

\begin{lemma}    \label{lem:bondunique}    \samepage
\ifDRAFT {\rm lem:bondunique}. \fi
For an irreducible tridiagonal matrix $A \in \Matd$,
there exists a unique irreducible tridiagonal matrix in $\Matd$ that is bonded to $A$.
\end{lemma}

\begin{proof}
By the construction.
\end{proof}

\begin{lemma}    \label{lem:bond000}    \samepage
\ifDRAFT {\rm lem:bond000}. \fi
Consider irreducible tridiagonal matrices $A$, $B$  in $\Matd$.
Then $A$ and $B$ are bonded if and only if the following {\rm (i)--(iii)} hold:
\begin{itemize}
\item[\rm (i)]
$A_{i, i-1} = B_{i, i-1}$ for $1 \leq i \leq d$;
\item[\rm (ii)]
$A_{i-1, i} = B_{i-1, i}$ for $1 \leq i \leq d$;
\item[\rm (iii)]
$A_{i,i} = - B_{i,i}$ for $0 \leq i \leq d$.
\end{itemize}
\end{lemma}

\begin{proof}
Use Lemma \ref{lem:trid5}.
\end{proof}

\begin{lemma}    \label{lem:bondnormalized}    \samepage
\ifDRAFT {\rm lem:bondnormalized}. \fi
For irreducible tridiagonal matrices $A$, $B$ in $\Matd$,
assume that $A$ and $B$ are bonded.
Then $A$ is normalized if and only if $B$ is normalized.
\end{lemma}

\begin{proof}
By Lemma \ref{lem:bond000}(i).
\end{proof}

In the next two results,
we characterize the bond relation in various ways.

\begin{lemma}    \label{lem:bond000b}    \samepage
\ifDRAFT {\rm lem:bond000b}. \fi
Consider irreducible tridiagonal matrices $A$, $B$ in $\Matd$.
Then the following {\rm (i)--(iii)} are equivalent:
\begin{itemize}
\item[\rm (i)]
$A$ and $B$ are bonded;
\item[\rm (ii)]
$A-B$ is diagonal with $(i,i)$-entry $2 A_{i,i}$ for $0 \leq i \leq d$;
\item[\rm (iii)]
$B-A$ is diagonal with $(i,i)$-entry $2 B_{i,i}$ for $0 \leq i \leq d$.
\end{itemize}
\end{lemma}

\begin{proof}
Use Lemma \ref{lem:bond000}.
\end{proof}

\begin{lemma}    \label{lem:bonded0}    \samepage
\ifDRAFT {\rm lem:bonded0}. \fi
For an irreducible tridiagonal matrix $A \in \Matd$ and a diagonal matrix $K \in \Matd$,
the following are equivalent:
\begin{itemize}
\item[\rm (i)]
$A$ and $A-K$ are bonded;
\item[\rm (ii)]
$K_{i,i} = 2 A_{i,i}$ for $0 \leq i \leq d$.
\end{itemize}
\end{lemma}

\begin{proof}
Define $B = A-K$ and use Lemma \ref{lem:bond000b}.
\end{proof}

\begin{lemma}    \label{lem:bond0c}    \samepage
\ifDRAFT {\rm lem:bond0c}. \fi
Let $A \in \Matd$ be irreducible tridiagonal.
If $\text{\rm Char}(\F)=2$ then $A$ is bonded to $A$.
If $\text{\rm Char}(\F) \neq 2$ then 
the following are equivalent:
\begin{itemize}
\item[\rm (i)]
$A$ is bonded to $A$;
\item[\rm (ii)]
$A_{i,i} = 0$ for $0 \leq i \leq d$.
\end{itemize}
\end{lemma}

\begin{proof}
First assume that $\text{\rm Char}(\F)=2$.
Then $A$ is bonded to $A$ by Note \ref{note:vee2}.
Next assume that $\text{\rm Char}(\F) \neq 2$.
In Lemma \ref{lem:bonded0}, set $K=0$ to get the equivalence
of (i), (ii).
\end{proof}

\section{Leonard pairs and Leonard systems}
\label{sec:LP} 
\ifDRAFT {\rm sec:LP}. \fi

In this section, we briefly recall the notion of a Leonard pair
and a Leonard system.

\begin{defi}  {\rm (See \cite[Definition 1.1]{T:Leonard}.) }
  \label{def:LP}    \samepage
\ifDRAFT {\rm def:LP}. \fi
By a {\em Leonard pair on $V$} we mean an ordered pair $A,A^*$
of elements in $\End$ that satisfy {\rm (i)} and {\rm (ii)} below:
\begin{itemize}
\item[\rm (i)]
there exists a basis for $V$ with respect to which the matrix representing
$A$ is irreducible tridiagonal and the matrix representing $A^*$ is diagonal;
\item[\rm (ii)]
there exists a basis for $V$ with respect to which the matrix representing
$A^*$ is irreducible tridiagonal and the matrix representing $A$ is diagonal.
\end{itemize}
\end{defi}

\begin{note}
By a common notational convention, $A^*$ denotes
the conjugate-transpose of $A$.
We are not using this convention.
In a Leonard pair $A,A^*$ the elements $A$ and $A^*$ are arbitrary subject
to (i) and (ii) above.
\end{note}

\begin{note}    \label{note:d0}    \samepage
\ifDRAFT {\rm note:d0}. \fi
Assume that $d=0$.
Then any ordered pair of elements in $\text{\rm End}(V)$ form a Leonard pair on $V$.
\end{note}

For the rest of this paper, we assume $d \geq 1$.

\medskip
Consider a Leonard pair $A,A^*$ on $V$ and a Leonard pair $B,B^*$ on
a vector space $V'$.
By an {\em isomorphism of Leonard pairs from $A,A^*$ to $B,B^*$} we mean
an algebra isomorphism $\End \to \text{\rm End}(V')$ that sends
$A \mapsto B$ and $A^* \mapsto B^*$.
The Leonard pairs $A,A^*$ and $B,B^*$ are said to be {\em isomorphic}
whenever there exists an isomorphism of Leonard pairs from $A,A^*$
to $B,B^*$.

\begin{lemma}   \label{lem:SASinv}    \samepage
\ifDRAFT {\rm lem:SASinv}. \fi
For a Leonard pair $A,A^*$ on $V$ and a pair $B,B^*$ of elements
in $\End$,
the following are equivalent:
\begin{itemize}
\item[\rm (i)]
$B,B^*$ is a Leonard pair on $V$ that is isomorphic to $A,A^*$;
\item[\rm (ii)]
there exists an invertible $S \in \End$ such that
$B = S A S^{-1}$ and $B^* = S A^* S^{-1}$.
\end{itemize}
\end{lemma}

\begin{proof}
Routine verification using Lemma \ref{lem:SN}.
\end{proof}

\begin{lemma}   {\rm (See \cite[Lemma 5.1]{NT:affine}.) }
\label{lem:affineLP}    \samepage
\ifDRAFT {\rm lem:affineLP}. \fi
For a Leonard pair $A,A^*$ on $V$ and scalars 
$\xi$, $\xi^*$, $\zeta$, $\zeta^*$ in $\F$ with $\xi \xi^* \neq 0$,
the pair  $\xi A + \zeta I, \, \xi^* A^* + \zeta^* I$ is a Leonard pair on $V$.
\end{lemma}

When working with a Leonard pair, it is often convenient to consider a closely related
object called a Leonard system. 
In order to define this we first make an observation about Leonard pairs.

\begin{lemma}    {\rm (See \cite[Lemma 3.1]{T:Leonard}.) }
\label{lem:mfree}   \samepage
\ifDRAFT {\rm lem:mfree}. \fi
For a Leonard pair $A,A^*$ on $V$,
each of $A$, $A^*$ is multiplicity-free.
\end{lemma}

A Leonard system is defined as follows.

\begin{defi}  {\rm (See \cite[Definition 1.4]{T:Leonard}.) }
  \label{def:LS}    \samepage
\ifDRAFT {\rm def:LS}. \fi
By a {\em Leonard system} on $V$ we mean a sequence
\begin{equation}
    \Phi = (A; \{E_i\}_{i=0}^d; A^*; \{E^*_i\}_{i=0}^d)              \label{eq:Phi}
\end{equation}
of elements in $\End$ that satisfy the following {\rm (i)--(v)}:
\begin{itemize}
\item[\rm (i)]
each of $A,A^*$ is multiplicity-free;
\item[\rm (ii)]
$\{E_i\}_{i=0}^d$ is an ordering of the primitive idempotents of $A$;
\item[\rm (iii)]
$\{E^*_i\}_{i=0}^d$ is an ordering of the primitive idempotents of $A^*$;
\item[\rm (iv)]
$\displaystyle
  E^*_i A E^*_j =
   \begin{cases}
      0  &  \text{ if $|i-j|>1$},
   \\
   \neq 0 & \text{ if $|i-j|=1$}
  \end{cases}
   \qquad (0 \leq i,j \leq d)$;
\item[\rm (v)]
$\displaystyle
  E_i A^* E_j =
   \begin{cases}
      0  &  \text{ if $|i-j|>1$},
   \\
   \neq 0 & \text{ if $|i-j|=1$}
  \end{cases}
   \qquad (0 \leq i,j \leq d)$.
\end{itemize}
We say that $\Phi$ is {\em over $\F$}.
\end{defi}

Leonard systems are related to Leonard pairs as follows.
For the Leonard system $\Phi$ from \eqref{eq:Phi},
by \cite[Section 3]{T:qRacah} the pair $A,A^*$ is a Leonard pair on $V$.
Conversely, for a Leonard pair $A,A^*$ on $V$, each of $A,A^*$ is multiplicity-free
by Lemma \ref{lem:mfree}.
Moreover there exists an ordering $\{E_i\}_{i=0}^d$ of the primitive idempotents of $A$
and an ordering of $\{E^*_i\}_{i=0}^d$ of the primitive idempotents of $A^*$
such that $(A; \{E_i\}_{i=0}^d; A^*; \{E^*_i\}_{i=0}^d)$ is a Leonard system on $V$
(see \cite[Lemma 3.3]{T:qRacah}.)

Consider the Leonard system $\Phi$ from \eqref{eq:Phi}
and a Leonard system
\[
\Phi' = (B; \{E'_i\}_{i=0}^d; B^*; \{E^{* \prime}_i\}_{i=0}^d)
\]
on a vector space $V'$.
By an {\em isomorphism of Leonard systems from $\Phi$ to $\Phi'$} we mean an algebra
isomorphism $\End \to \text{\rm End}(V')$ that sends $A \mapsto B$,
$A^* \mapsto B^*$ and
$E_i \mapsto E'_i$, $E^*_i \mapsto E^{* \prime}_i$ for $0 \leq i \leq d$.
The Leonard systems $\Phi$ and $\Phi'$ are said to be {\em isomorphic}
whenever there exists an isomorphism of Leonard systems from $\Phi$ to $\Phi'$.

\begin{lemma}    \label{lem:LSiso}    \samepage
\ifDRAFT {\rm lem:LSiso}. \fi
Consider the Leonard system $\Phi$ from \eqref{eq:Phi}
and a sequence
\[
\Phi' = (B; \{E'_i\}_{i=0}^d; B^*; \{E^{* \prime}_i\}_{i=0}^d)
\]
of elements in $\End$.
Then the following are equivalent:
\begin{itemize}
\item[\rm (i)]
$\Phi'$ is a Leonard system on $V$ that is isomorphic to $\Phi$;
\item[\rm (ii)]
there exists an invertible $S \in \End$ such that
$S A S^{-1} = B$, $S A^* S^{-1} = B^*$,
and $S E_i S^{-1} = E'_i$, $S E^*_i S^{-1} = E^{* \prime}_i$ for $0 \leq i \leq d$.
\end{itemize}
\end{lemma}

\begin{proof}
Routine verification using Lemma \ref{lem:SN}.
\end{proof}

\begin{defi}    \label{def:standard}    \samepage
\ifDRAFT {\rm def:standard}. \fi
Consider a Leonard pair $A,A^*$ on $V$.
An ordering $\{E_i\}_{i=0}^d$ of the primitive idempotents
of $A$ is said to be {\em standard} whenever
it satisfies Definition \ref{def:LS}(v).
A standard ordering of the primitive idempotents of $A^*$ is
similarly defined.
\end{defi}

Referring to Definition \ref{def:standard}, for a standard ordering $\{E_i\}_{i=0}^d$
of the primitive idempotents of $A$, the ordering $\{E_{d-i} \}_{i=0}^d$ is standard 
and no further ordering is standard.
A similar comment applies to the primitive idempotents of $A^*$.

For the Leonard system $\Phi$ from  \eqref{eq:Phi},
each of the following is a Leonard system on $V$:
\begin{align*}
\Phi^* &= (A^*; \{E^*_i\}_{i=0}^d; A ; \{E_i\}_{i=0}^d),
\\
\Phi^\downarrow  &= (A; \{E_i\}_{i=0}^d; A^*; \{E^*_{d-i} \}_{i=0}^d),
\\
\Phi^\Downarrow &= (A; \{E_{d-i}\}_{i=0}^d; A^*; \{E^*_i\}_{i=0}^d).
\end{align*}
For $g \in \{*, \downarrow, \Downarrow\}$ and 
an object $f$ associated with $\Phi$, let $f^g$ denote the corresponding object
associated with $\Phi^g$.

The Leonard system $\Phi$ from \eqref{eq:Phi} and the Leonard pair $A,A^*$
are said to be {\em associated}.

\begin{lemma}    {\rm (See \cite[Section 3]{T:TDD}.) }
\label{lem:LSequiv}    \samepage
\ifDRAFT {\rm lem:LSequiv}. \fi
Let $A,A^*$ denote a Leonard pair on $V$.
If $\Phi$ is a Leonard system associated with $A,A^*$,
then so is $\Phi^\downarrow$, $\Phi^\Downarrow$, $\Phi^{\downarrow\Downarrow}$.
No further Leonard system is associated with $A,A^*$.
\end{lemma}

\begin{defi}    \label{def:eigenseq}    \samepage
\ifDRAFT {\rm def:eigenseq}. \fi
Consider the Leonard system $\Phi$ from  \eqref{eq:Phi}.
For $0 \leq i \leq d$ let $\th_i$ (resp.\ $\th^*_i$) denote the eigenvalue of $A$
(resp.\ $A^*$) associated with $E_i$ (resp.\ $E^*_i$).
We call $\{\th_i\}_{i=0}^d$ (resp.\ $\{\th^*_i\}_{i=0}^d$) the 
{\em eigenvalue sequence} (resp.\ {\em dual eigenvalue sequence}) of $\Phi$.
\end{defi}

\begin{lemma}   {\rm (See \cite[Theorem 3.2]{T:Leonard}.)}
\label{lem:parray}    \samepage
\ifDRAFT {\rm lem:parray}. \fi
Referring to Definition \ref{def:eigenseq},
there exists a sequence $\{\vphi_i\}_{i=1}^d$ of scalars in $\F$
and a basis of $V$ with respect to which the matrices representing $A$ and $A^*$ are
\begin{align*}
A &: 
\begin{pmatrix}
 \th_0 & & & & & \text{\bf 0}  \\
 1 & \th_1 \\
    & 1 & \th_2 \\
    &    &  \cdot & \cdot \\
    &     &          & \cdot & \cdot \\
\text{\bf 0} & & & & 1 & \th_d
\end{pmatrix},
&
A^* &:
\begin{pmatrix}
\th^*_0 & \vphi_1 & & & & {\bf 0}  \\
          & \th^*_1 & \vphi_2 \\
         &  &  \th^*_1 &  \cdot    \\
         &  &  &  \cdot & \cdot  \\
         &  &  &  &  \cdot & \vphi_d  \\
\text{\bf 0} & & & & & \th^*_d
\end{pmatrix}.
\end{align*}
The sequence $\{\vphi_i\}_{i=1}^d$ is uniquely determined by $\Phi$.
Moreover $\vphi_i \neq 0$ for $1 \leq i \leq d$.
\end{lemma}

\begin{defi}  {\rm See \cite[Definition 3.10]{T:Leonard}.) }
\label{def:splitseq}    \samepage
\ifDRAFT {\rm def:splitseq}. \fi
Referring to Lemma \ref{lem:parray}, we call $\{\vphi_i\}_{i=1}^d$
the {\em first split sequence} of $\Phi$.
Let $\{\phi_i\}_{i=1}^d$ denote the first split sequence of $\Phi^\Downarrow$.
We call $\{\phi_i\}_{i=1}^d$ the {\em second split sequence} of $\Phi$.
\end{defi}

\begin{defi}    {\rm (See \cite[Definition 22.3]{T:survey}.) }
\label{def:parray}    \samepage
\ifDRAFT {\rm def:parray}. \fi
For the Leonard system $\Phi$ from \eqref{eq:Phi},
by the {\em parameter array of $\Phi$} 
we mean the sequence
\[
  (\{\th_i\}_{i=0}^d; \{\th^*_i\}_{i=0}^d; \{\vphi_i\}_{i=1}^d; \{\phi_i\}_{i=1}^d), 
\]
where $\{\th_i\}_{i=0}^d$ (resp.\ $\{\th^*_i\}_{i=0}^d$) is the eigenvalue sequence
(resp.\ dual eigenvalue sequence) of $\Phi$,
and $\{\vphi_i\}_{i=1}^d$ (resp.\ $\{\phi_i\}_{i=1}^d$) is the first split sequence 
(resp.\ second split sequence) of $\Phi$.
\end{defi}

\begin{lemma}   {\rm (See \cite[Theorem 1.9]{T:Leonard}.) }
\label{lem:classify}    \samepage
\ifDRAFT {\rm lem:classify}. \fi
Consider a sequence
\begin{equation}
  (\{\th_i\}_{i=0}^d; \{\th^*_i\}_{i=0}^d; \{\vphi_i\}_{i=1}^d; \{\phi_i\}_{i=1}^d)         \label{eq:parray}
\end{equation}    
of scalars taken from $\F$.
Then there exists a Leonard system $\Phi$ over $\F$ with parameter array \eqref{eq:parray}
if and only if the following conditions {\rm (i)--(v)} hold:
\begin{itemize}
\item[\rm (i)]
$\th_i \neq \th_j, \quad \th^*_i \neq \th^*_j \quad$ if $i \neq j \quad (0 \leq i,j \leq d)$;
\item[\rm (ii)]
$\vphi_i \neq 0, \quad \phi_i \neq 0 \quad (1 \leq i \leq d)$;
\item[\rm (iii)]
$\vphi_i = \phi_1 \sum_{\ell=0}^{i-1} \frac{\th_\ell - \th_{d-\ell} } { \th_0 - \th_d }
             + (\th^*_i - \th^*_0)(\th_{i-1}-\th_d)  \quad (1 \leq i \leq d)$;
\item[\rm (iv)]
$\phi_i = \vphi_1 \sum_{\ell=0}^{i-1} \frac{\th_\ell - \th_{d-\ell} } { \th_0 - \th_d }
             + (\th^*_i - \th^*_0)(\th_{d-i+1}-\th_0)  \quad (1 \leq i \leq d)$;
\item[\rm (v)]
the expressions
\begin{equation}
\frac{\th_{i-2} - \th_{i+1} } { \th_{i-1} - \th_i },
\qquad\qquad
\frac{\th^*_{i-2} - \th^*_{i+1} } { \th^*_{i-1} - \th^*_i }               \label{eq:indep}
\end{equation}
are equal and independent of $i$ for $2 \leq i \leq d-1$.
\end{itemize}
Moreover, if {\rm (i)--(v)} hold, then $\Phi$ is uniquely determined 
up to isomorphism of Leonard systems.
\end{lemma}

\begin{defi}     \label{def:parrayoverF}    \samepage
\ifDRAFT {\rm def:parrayoverF}. \fi
By a {\em parameter array over $\F$} we mean a sequence
\[
 (\{\th_i\}_{i=0}^d; \{\th^*_i\}_{i=0}^d; \{\vphi_i\}_{i=1}^d; \{\phi_i\}_{i=1}^d)
\]
of scalars taken from $\F$ that satisfy conditions {\rm (i)--(v)} in Lemma \ref{lem:classify}.
\end{defi}

\begin{defi}    \label{def:beta0}    \samepage
\ifDRAFT {\rm def:beta0}. \fi
Referring to Definition \ref{def:parrayoverF}, 
assume that $d \geq 3$.
Define $\beta \in \F$ such that $\beta+1$ is equal to the common value
of the two fractions in \eqref{eq:indep}.
We call $\beta$ the {\em fundamental constant} of the parameter array  \eqref{eq:parray}.
\end{defi}

\begin{defi}     \label{def:betaLS}    \samepage
\ifDRAFT {\rm def:betaLS}. \fi
Assume that $d \geq 3$.
Then the fundamental constant of a given Leonard system is
the fundamental constant of the associated parameter array.
\end{defi}

Referring to Definition \ref{def:betaLS}, 
observe that for a Leonard pair on $V$ the associated Leonard systems  have the
same fundamental constant.

\begin{defi}    \label{def:betaLP}    \samepage
\ifDRAFT {\rm def:betaLP}. \fi
Assume that $d \geq 3$.
The fundamental constant of a given Leonard pair is
the fundamental constant of an associated Leonard system.
\end{defi}

In the next result, we emphasize some relations from Lemma \ref{lem:classify}
for later use.

\begin{lemma}     \label{lem:1d}    \samepage
\ifDRAFT {\rm lem:1d}. \fi
Referring to Definition \ref{def:parrayoverF},
\begin{align}
\vphi_1 - \phi_1 &=  (\th^*_1 -\th^*_0)(\th_0 - \th_d),              \label{eq:vphi1}
\\
\vphi_d - \phi_1 &= (\th^*_d - \th^*_0)(\th_{d-1}- \th_d),         \label{eq:vphid}
\\
\phi_d - \vphi_1 &=  (\th^*_d - \th^*_0)(\th_1 - \th_0).            \label{eq:phid}
\end{align}
\end{lemma}

\begin{proof}
In Lemma \ref{lem:classify}(iii),(iv) set $i=1$ and $i=d$.
\end{proof}

\begin{lemma}    {\rm (See \cite[Theorem 11.1]{ITT}.) }
\label{lem:gammarho}    \samepage
\ifDRAFT {\rm lem:gammarho}. \fi
Assume that $d \geq 3$.
Consider a parameter array \eqref{eq:parray} over $\F$ with fundamental constant $\beta$. 
Then there exist scalars $\gamma$, $\varrho$ such that 
\begin{align*}
\gamma &= \th_{i-1} - \beta \th_i + \th_{i+1}  &&   (1 \leq i \leq d-1),  
\\
\varrho &= \th_{i-1}^2 - \beta \th_{i-1} \th_i + \th_i^2 - \gamma (\th_{i-1} + \th_i)
      &&  (1 \leq i \leq d).
\end{align*}
\end{lemma}

For a Leonard pair $A,A^*$ on $V$, consider the scalars $\gamma$, $\beta$
for the parameter array of a Leonard system associated with $A,A^*$.
Observe that the scalars $\gamma$, $\varrho$ are determined by $A,A^*$.

\begin{defi}    \label{def:kappa}     \samepage
\ifDRAFT {\rm def:kappa}. \fi
Assume that $d \geq 3$.
Referring to Lemma \ref{lem:gammarho}, define
\begin{equation}
  \kappa = \gamma^2 + (2 - \beta) \varrho.      \label{eq:defkappa}
\end{equation}
We call $\kappa$ the {\em invariant value} for the parameter array \eqref{eq:parray}.
\end{defi}

\begin{defi}    \label{def:kappaLS}    \samepage
\ifDRAFT {\rm def:kappaLS}. \fi
Assume that $d \geq 3$.
Let $\Phi$ denote a Leonard system on $V$.
By the {\em invariant value for $\Phi$} we mean the invariant value for the
parameter array of $\Phi$.
\end{defi}

For a Leonard pair $A,A^*$ on $V$,  consider a Leonard system
$\Phi$ associated with $A,A^*$.
Observe that the invariant value for $\Phi$ is determined by $A,A^*$,
and independent of the choice of an associated Leonard system.

\begin{defi}   \label{def:kappaLP}    \samepage
\ifDRAFT {\rm def:kappaLP}. \fi
Assume that $d \geq 3$.
Let $A,A^*$ denote a Leonard pair on $V$.
By the {\em invariant value for $A,A^*$} we mean
the invariant value of a Leonard system associated with $A,A^*$.
\end{defi}

\begin{lemma}    \label{lem:kappa}    \samepage
\ifDRAFT {\rm lem:kappa}. \fi
Referring to Definition \ref{def:kappa},
\begin{align*}
  \kappa &=  (\th_{i-1}- \th_{i+1})^2 + (\beta+2)(\th_i - \th_{i-1})(\th_i - \th_{i+1})   
    &&  ( 1 \leq i \leq d-1). 
\end{align*}
\end{lemma}

\begin{proof}
In \eqref{eq:defkappa}, eliminate $\gamma$ and $\varrho$ using Lemma \ref{lem:gammarho},
and rearrange the terms.
\end{proof}

\begin{lemma}   {\rm (See \cite[Lemma 19.13]{T:LSnote}.)}
\label{lem:indep2}    \samepage
\ifDRAFT {\rm lem:indep2}. \fi
Referring to Definition \ref{def:parrayoverF},
\begin{align*}
 \frac{\th_{\ell} - \th_{d-\ell} }
        {\th_0 - \th_d }
&=
 \frac{\th^*_{\ell} - \th^*_{d-\ell} }
        {\th^*_0 - \th^*_d }
&& (0 \leq \ell \leq d).
\end{align*}
\end{lemma}

\begin{notation}    \label{notation6}    \samepage
\ifDRAFT {\rm notation6}. \fi
Let
$\Phi = (A; \{E_i\}_{i=0}^d; A^* ; \{E^*_i\}_{i=0}^d)$
denote a Leonard system on $V$ with parameter array
$(\{\th_i\}_{i=0}^d; \{\th^*_i\}_{i=0}^d; \{\vphi_i\}_{i=1}^d; \{\phi_i\}_{i=1}^d)$.
\end{notation}

\begin{lemma}    {\rm (See \cite[Theorem 1.11]{T:Leonard}.) }
\label{lem:parrayrelative}    \samepage
\ifDRAFT {\rm lem:parrayrelative}. \fi
Referring to Notation \ref{notation6},
the following {\rm (i)--(iii)} hold.
\begin{itemize}
\item[\rm (i)]
The parameter array of $\Phi^*$ is 
$(\{\th^*_i\}_{i=0}^d; \{\th_i \}_{i=0}^d; \{\vphi_i \}_{i=1}^d; \{ \phi_{d-i+1} \}_{i=1}^d)$.
\item[\rm (ii)]
The parameter array of $\Phi^\downarrow$ is 
$(\{\th_i\}_{i=0}^d; \{\th^*_{d-i}\}_{i=0}^d; \{\phi_{d-i+1} \}_{i=1}^d; \{ \vphi_{d-i+1} \}_{i=1}^d)$.
\item[\rm (iii)]
The parameter array of $\Phi^\Downarrow$ is 
$(\{\th_{d-i}\}_{i=0}^d; \{\th^*_{i}\}_{i=0}^d; \{\phi_{i} \}_{i=1}^d; \{ \vphi_{i} \}_{i=1}^d)$.
\end{itemize}
\end{lemma}

\begin{lemma}  {\rm (See \cite[Lemmas 5.1, 6.1]{NT:affine}.) }
 \label{lem:affineparam}    \samepage
\ifDRAFT {\rm lem:affineparam}. \fi
Referring to Notation \ref{notation6},
for scalars   $\xi$, $\xi^*$, $\zeta$, $\zeta^*$ in $\F$ with $\xi \xi^* \neq 0$,
the sequence
\[
 (\xi A+ \zeta I; \{E_i\}_{i=0}^d; \xi^* A^* + \zeta^* I; \{E^*_i\}_{i=0}^d)
\]
is a Leonard system on $V$ with parameter array
\[
    (\{\xi \th_i + \zeta \}_{i=0}^d; \{\xi^* \th^*_i + \zeta^*\}_{i=0}^d; 
    \{ \xi \xi^* \vphi_i\}_{i=1}^d; \{ \xi \xi^* \phi_i\}_{i=1}^d).
\]
\end{lemma}

\begin{defi}    {\rm (See \cite[Definition 7.1]{T:qRacah}.) }
\label{def:ai}    \samepage
\ifDRAFT {\rm def:ai}. \fi
Referring to Notation \ref{notation6},
define
\begin{align*}
  a_i &= \text{\rm tr}(A E^*_i)  &&   (0 \leq i \leq d).
\end{align*}
\end{defi}

\begin{lemma}    \label{lem:aithi}    \samepage
\ifDRAFT {\rm lem:aithi}. \fi
Referring to Notation \ref{notation6},
$\sum_{i=0}^d a_i = \sum_{i=0}^d \th_i$.
\end{lemma}

\begin{proof}
Using $\sum_{i=0}^d E^*_i = I$ we find that $\sum_{i=0}^d a_i = \text{\rm tr}(A)$.
The result follows.
\end{proof}

\begin{lemma}    \label{lem:aiD}     \samepage
\ifDRAFT {\rm lem:aiD}. \fi
Referring to Notation \ref{notation6},
the following hold: for $0 \leq i \leq d$:
\begin{itemize}
\item[\rm (i)]
$a^\downarrow_i = a_{d-i}$;
\item[\rm (ii)]
$a^\Downarrow_i = a_i$.
\end{itemize}
\end{lemma}

\begin{proof}
By Definition \ref{def:ai}.
\end{proof}

\begin{lemma}   {\rm (See \cite[Lemma 10.2]{T:qRacah}.) }
\label{lem:standardbasis}   \samepage
\ifDRAFT {\rm lem:standardbasis}. \fi
Referring to Notation \ref{notation6},
for $0 \neq u \in E_0 V$ the vectors $\{E^*_i v\}_{i=0}^d$ form a basis of $V$.
\end{lemma}

Our next goal is to describe the action of $A$, $A^*$ on the above basis.

\begin{lemma}   {\rm (See \cite[Lemma 9.2]{T:qRacah}.) }
\label{lem:bicipre}    \samepage
\ifDRAFT {\rm lem:bicipre}. \fi
Referring to Notation  \ref{notation6},
$\text{\rm tr}(E^*_i E_0) \neq 0$ for $0 \leq i \leq d$.
\end{lemma}

\begin{defi}     {\rm (See \cite[Lemma 11.5]{T:qRacah}.) }
\label{def:bici}     \samepage
\ifDRAFT {\rm def:bici}. \fi
Referring to Notation \ref{notation6},
define
\begin{align*}
  b_i &= \frac{ \text{\rm tr}(E^*_i A E^*_{i+1} E_0) } {\text{\rm tr}(E^*_i E_0) }  && (0 \leq i \leq d-1),
\\
 c_i &= \frac{ \text{\rm tr}(E^*_i A E^*_{i-1} E_0) } {\text{\rm tr}(E^*_i E_0) }  && (1 \leq i \leq d).
\end{align*}
\end{defi}

\begin{lemma}   {\rm (See \cite[Lemma 10.2, Definition 11.1]{T:qRacah}.) }
\label{lem:aibici0}    \samepage
\ifDRAFT {\rm lem:aibici0}. \fi
Referring to Notation \ref{notation6},
with respect to the basis in Lemma \ref{lem:standardbasis}, 
the matrices representing $A$ and $A^*$ are
\begin{align*}
A & :
 \begin{pmatrix}
  a_0 & b_0 & & & & \text{\bf 0}  \\
  c_1 & a_1 & b_1 \\
       & c_2 & \cdot & \cdot  \\
       &       &  \cdot & \cdot & \cdot \\
       &        &         & \cdot & \cdot & b_{d-1}   \\
 \text{\bf 0} & & & & c_d & a_d
 \end{pmatrix},
&
A^* & :
 \text{\rm diag}(\th^*_0, \th^*_1, \ldots, \th^*_d).
\end{align*}
\end{lemma}

\begin{lemma}    {\rm (See \cite[Lemma 11.2]{T:qRacah}.) }
\label{lem:aibici}    \samepage
\ifDRAFT {\rm lem:aibici}. \fi
Referring to Notation \ref{notation6}, the following  {\rm (i)--(iii)} hold:
\begin{itemize}
\item[\rm (i)]
$b_i \neq 0 \quad (0 \leq i \leq d-1)$;
\item[\rm (ii)]
$c_i \neq 0 \quad (1 \leq i \leq d)$;
\item[\rm (iii)]
$c_i + a_i + b_i = \th_0 \quad (0 \leq i \leq d)$, where $c_0 = 0$ and $b_d = 0$.
\end{itemize}
\end{lemma}

\begin{defi}     {\rm (See \cite[Definition 9.1]{Wor}.) }
\label{def:int}    \samepage
\ifDRAFT {\rm def:int}. \fi
Referring to Notation \ref{notation6},
we call the scalars $\{a_i\}_{i=0}^d$, $\{b_i\}_{i=0}^{d-1}$, $\{c_i\}_{i=1}^d$
the {\em intersection numbers} of $\Phi$.
\end{defi}

\begin{defi}    {\rm (See \cite[Definition 7.1]{T:qRacah}.) }
\label{def:xi}    \samepage
\ifDRAFT {\rm def:xi}. \fi
Referring to Notation \ref{notation6},
define
\begin{align*}
  x_i &= \text{\rm tr}(E^*_i A E^*_{i-1} A)   &&  (1 \leq i \leq d).
\end{align*}
\end{defi}

\begin{lemma}    \label{lem:xiD}    \samepage
\ifDRAFT {\rm lem:xiD}. \fi
Referring to Notation \ref{notation6},
the following hold for $1 \leq i \leq d$:
\begin{itemize}
\item[\rm (i)]
$x^\downarrow_i = x_{d-i+1}$;
\item[\rm (ii)]
$x^\Downarrow_i = x_i$.
\end{itemize}
\end{lemma}

\begin{proof}
By Definition \ref{def:xi}.
\end{proof}

\begin{lemma}   {\rm (See \cite[Lemma 7.2]{T:qRacah}.) }
 \label{lem:ai0}    \samepage
\ifDRAFT {\rm lem:ai0}. \fi
Referring to Notation \ref{notation6},
let $M \in \Matd$ represent $A$ with respect to a basis of $V$ that
satisfies Definition \ref{def:LP}{\rm (i)}.
Then 
\begin{align*}
M_{i,i} &= a_i \qquad (0 \leq i \leq d),
&
M_{i, i-1} M_{i-1,i} &= x_i \qquad (1 \leq i \leq d).
\end{align*}
\end{lemma}

\begin{corollary}     \label{cor:aixi0}     \samepage
\ifDRAFT {\rm cor:aixi0}. \fi
Referring to Lemma \ref{lem:ai0},
assume that $M$ is normalized.
Then
\[
M = 
 \begin{pmatrix}
  a_0 & x_1 & & & & \text{\bf 0}  \\
  1 & a_1 & x_2 \\
       & 1 & \cdot & \cdot  \\
       &       &  \cdot & \cdot & \cdot \\
       &        &         & \cdot & \cdot & x_d   \\
 \text{\bf 0} & & & & 1 & a_d
 \end{pmatrix}.
\]
\end{corollary}

\begin{proof}
We have $M_{i,i-1}=1$ for $1 \leq i \leq d$, since $M$ is normalized.
By thins and Lemma \ref{lem:ai0}, we get the result.
\end{proof}

\begin{lemma}    {\rm (See \cite[Theorems 7.3, 7.4]{T:qRacah}.) }
\label{lem:aixi0}    \samepage
\ifDRAFT {\rm lem:aixi0}. \fi
Referring to Notation \ref{notation6},
for $0 \neq v \in E^*_0 V$ the vectors $\{E^*_i A^i v\}_{i=0}^d$ form a basis of $V$
that satisfies Definition \ref{def:LP}(i).
With respect to this basis  the matrix representing $A$ is normalized.
\end{lemma}

\begin{lemma}  {\rm (See \cite[Lemma 11.2]{T:qRacah}.) }
\label{lem:xibici}    \samepage
\ifDRAFT {\rm lem:xibici}. \fi
Referring to Notation  \ref{notation6},
$ x_i = b_{i-1} c_i$ for $1 \leq i \leq d$.
\end{lemma}

\begin{lemma}   {\rm (See \cite[Theorem 17.8]{T:qRacah}.) }
\label{lem:aiparam}    \samepage
\ifDRAFT {\rm lem:aiparam}. \fi
Referring to Notation \ref{notation6},
\begin{align}
a_0 &= \th_0 + \frac{\vphi_1}{\th^*_0 - \th^*_1},                  \label{eq:a0}
\\
a_i &= \th_i + \frac{ \vphi_i } {\th^*_i - \th^*_{i-1} }
        + \frac{ \vphi_{i+1} } {\th^*_i - \th^*_{i+1} }  &&   (1 \leq i \leq d-1),   \label{eq:ai}
\\
a_d &= \th_d + \frac{\vphi_d} {\th^*_d - \th^*_{d-1} }.           \label{eq:ad}
\end{align}
\end{lemma}

\begin{lemma}    {\rm (See \cite[Lemma 10.3]{T:24points}.) }
\label{lem:a0th0}    \samepage
\ifDRAFT {\rm lem:a0th0}. \fi
Referring to Notation \ref{notation6},
\begin{align}
a_0 &= \th_d + \frac{\phi_1}{\th^*_0 - \th^*_1},                \label{eq:a0b}
\\
a_i &= \th_{d-i} + \frac{ \phi_i } {\th^*_i - \th^*_{i-1} }
      + \frac{ \phi_{i+1} } {\th^*_i - \th^*_{i+1} }     &&   (1 \leq i \leq d-1),         \label{eq:aib}
\\
a_d &= \th_0 + \frac{\phi_d} {\th^*_d - \th^*_{d-1} }.                      \label{eq:adb}
\end{align}
\end{lemma}

In the next result, we express each of $\vphi_1$, $\phi_1$, $\vphi_d$, $\phi_d$, $a_d$
in terms of $a_0$.

\begin{lemma}    \label{lem:a0ad}    \samepage
\ifDRAFT {\rm lem:a0ad}. \fi
Referring to Notation \ref{notation6},
\begin{align}
\vphi_1 &= (a_0 - \th_0)(\th^*_0 - \th^*_1),                      \label{eq:vphi1b}
\\
\phi_1 &= (a_0 - \th_d)(\th^*_0 - \th^*_1),                        \label{eq:phi1b}
\\
\vphi_d &= (a_d - \th_d)(\th^*_d - \th^*_{d-1}),       \label{eq:vphidb}
\\
\phi_d &= (a_d -\th_0)(\th^*_d - \th^*_{d-1}),        \label{eq:phidb}
\\
a_d &= 
 \frac{a_0 (\th^*_0 - \th^*_1) + \th_d (\th^*_1 - \th^*_{d-1} ) + \th_{d-1} (\th^*_d - \th^*_0) }
        { \th^*_d - \th^*_{d-1} }.                                                                   \label{eq:ada0}
\end{align}
\end{lemma}

\begin{proof}
Lines \eqref{eq:vphi1b}--\eqref{eq:phidb} come from \eqref{eq:a0}, \eqref{eq:a0b},
\eqref{eq:ad}, \eqref{eq:adb}, respectively.
To get \eqref{eq:ada0},
evaluate \eqref{eq:phi1b}, \eqref{eq:vphidb} using \eqref{eq:vphid}.
\end{proof}

Let $x$ denote an indeterminate, and let $\F[x]$ denote the algebra consisting 
of the polynomials in $x$ that have all coefficients in $\F$.
Referring to Notation \ref{notation6},
for $0 \leq i \leq d$ define polynomials in $\F[x]$:
\begin{align*}
 \tau_i &= (x-\th_0)(x-\th_1) \cdots (x-\th_{i-1}),
\\
 \eta_i &= (x-\th_d) (x - \th_{d-1}) \cdots (x- \th_{d-i+1}),
\\
 \tau^*_i &= (x- \th^*_0) (x- \th^*_1)\cdots (x - \th^*_{i-1}),
\\
 \eta^*_i &= (x - \th^*_d) (x - \th^*_{d-1}) \cdots (x - \th^*_{d-i+1}).
\end{align*}

\begin{lemma}   {\rm (See \cite[Theorem 17.9]{T:qRacah}.) }
\label{lem:xiparam}    \samepage
\ifDRAFT {\rm lem:xiparam}. \fi
Referring to Notation \ref{notation6},
\begin{align*}
x_i &= \vphi_i \phi_i 
   \frac{ \tau^*_{i-1} (\th^*_{i-1}) \eta^*_{d-i} (\th^*_i) } 
          { \tau^*_i (\th^*_i)  \eta^*_{d-i+1}( \th^*_{i-1}) }
                      &&  (1 \leq i \leq d).  
\end{align*}
\end{lemma}

\begin{lemma}  \label{lem:simLS2c}    \samepage
\ifDRAFT {\rm lem:simLS2c}.   \fi
Referring to Notation \ref{notation6},
consider a Leonard system $\Phi' = (B; \{E'_i\}_{i=0}^d; A^*; \{E^*_i\}_{i=0}^d)$
on $V$ with parameter array
$(\{\th'_i\}_{i=0}^d; \{\th^*_i\}_{i=0}^d; \{\vphi'_i\}_{i=1}^d; \{\phi'_i\}_{i=1}^d)$.
Let the scalars $\{x'_i\}_{i=1}^d$ be from Definition \ref{def:xi} for $\Phi'$.
Then the following are equivalent:
\begin{itemize}
\item[\rm (i)]
$x_i = x'_i \quad (1 \leq i \leq d)$;
\item[\rm (ii)]
$\vphi_i \phi_i = \vphi'_i \phi'_i \quad (1 \leq i \leq d)$.
\end{itemize}
\end{lemma}

\begin{proof}
Use Lemma \ref{lem:xiparam}.
\end{proof}

For parameter arrays
\begin{align}
& (\{\th_i\}_{i=0}^d; \{\th^*_i\}_{i=0}^d; \{\vphi_i\}_{i=1}^d; \{\phi_i\}_{i=1}^d).
&& (\{\th'_i\}_{i=0}^d; \{\th^*_i\}_{i=0}^d; \{\vphi'_i\}_{i=1}^d; \{\phi'_i\}_{i=1}^d)   \label{eq:twoparrays}
\end{align}
over $\F$,
we will consider two special cases of the condition (ii) in Lemma \ref{lem:simLS2c}.
One special case is described in Lemma \ref{lem:A=Bpre}.
The other special case is described in Lemma \ref{lem:ABveepre}.

\begin{lemma}    \label{lem:A=Bpre}     \samepage
\ifDRAFT {\rm lem:A=Bpre}. \fi
For parameter arrays \eqref{eq:twoparrays} over $\F$,
the following hold.
\begin{itemize}
\item[\rm (i)]
Assume that
$\vphi'_i = \vphi_i$ and $\phi'_i = \phi_i$  for $1 \leq i \leq d$.
Then there exists $\zeta \in \F$ such that
$\th'_i = \th_i + \zeta$ for $0 \leq i \leq d$.
\item[\rm (ii)]
Assume that
$\vphi'_i = \phi_i$ and $\phi'_i = \vphi_i$ for $1 \leq i \leq d$.
Then there exists $\zeta \in \F$ such that
$\th'_i = \th_{d-i} + \zeta$ for $0 \leq i \leq d$.
\end{itemize}
\end{lemma}

\begin{proof}
By Lemma \ref{lem:indep2} we find that for $1 \leq i \leq d$,
\[
   \sum_{\ell=0}^{i-1} \frac{\th_\ell - \th_{d-\ell} } { \th_0 - \th_d }
 =  \sum_{\ell=0}^{i-1} \frac{\th^*_\ell - \th^*_{d-\ell} } { \th^*_0 - \th^*_d }.
\]
Denote this common value by $\vartheta_i$.
By Lemma \ref{lem:classify},
\begin{align}
  \vphi_i &= \phi_1 \vartheta_i + (\th^*_i - \th^*_0) (\th_{i-1} - \th_d),        \label{eq:vphiaux}
\\
  \phi_i &= \vphi_1 \vartheta_i + (\th^*_i - \th^*_0)(\th_{d-i+1} - \th_0).    \label{eq:phiaux}
\end{align}
Using \eqref{eq:vphiaux} we obtain
\begin{align}
 \vphi'_i &= \phi'_1  \vartheta_i + (\th^*_i - \th^*_0) (\th'_{i-1} - \th'_d).  \label{eq:vphidaux}
\end{align}

(i)
By $\vphi'_i = \vphi_i$ $(1 \leq i \leq d)$, $\phi'_1 = \phi_1$, and \eqref{eq:vphiaux}, \eqref{eq:vphidaux},
\begin{align*}
   \th'_{i-1} - \th'_d &= \th_{i-1} - \th_d        &&  (1 \leq i \leq d). 
\end{align*}
This implies  $\th'_i = \th_i + \zeta$ for $0 \leq i \leq d$, 
where $\zeta = \th'_d - \th_d$.

(ii)
By $\vphi'_i = \phi_i$ $(1 \leq i \leq d)$,  $\phi'_1 = \vphi_1$, and \eqref{eq:phiaux}, \eqref{eq:vphidaux},
\begin{align*}
  \th'_{i-1} - \th'_d &= \th_{d-i+1} - \th_0   &&   (1 \leq i \leq d).
\end{align*}
This implies $\th'_i = \th_{d-i} + \zeta$ for $0 \leq i \leq d$,
where $\zeta = \th'_d - \th_0$.
\end{proof}

We recall the bipartite property for Leonard pairs and Leonard systems.
Recall the scalars $\{a_i\}_{i=0}^d$  from Definition \ref{def:ai}.

\begin{defi}       {\rm (See \cite[Section 1]{NT:balanced}.) }
\label{def:bipartite}    \samepage
\ifDRAFT {\rm def:bipartite}. \fi
A Leonard system $\Phi$ on $V$ is said to be {\em bipartite}
whenever $a_i = 0$ for $0 \leq i  \leq d$.
\end{defi}

\begin{lemma}
\label{lem:airelative}    \samepage
\ifDRAFT {\rm lem:airelative}. \fi
Let $\Phi$ denote a bipartite Leonard system on $V$.
Then each of $\Phi^\downarrow$, $\Phi^\Downarrow$,
$\Phi^{\downarrow \Downarrow}$ is bipartite.
\end{lemma}

\begin{proof}
By Lemma \ref{lem:aiD}.
\end{proof}

In view of Lemmas \ref{lem:LSequiv} and \ref{lem:airelative},
we make a definition.

\begin{defi}        {\rm (See \cite[Section 1]{NT:balanced}.) }
\label{def:bipartiteLP}    \samepage
\ifDRAFT {\rm def:bipartiteLP}. \fi
A Leonard pair on $V$ is said to be {\em bipartite} whenever
an associated Leonard system is bipartite.
\end{defi}

We mention a lemma for later use.

\begin{lemma}     \label{lem:d2param}    \samepage
\ifDRAFT {\rm lem:d2param}. \fi
Referring to Notation \ref{notation6},
assume that $d=2$.
Then
\begin{align}
a_1 &= \th_0 + \th_2 - a_0 
  + \frac{ (a_0 - \th_1)(\th^*_0 - \th^*_1)} { \th^*_1 - \th^*_2 },     \label{eq:a1b}
\\
a_2 &= \frac{ \th_1 (\th^*_0 - \th^*_2) - a_0 (\th^*_0 - \th^*_1) }
                 { \th^*_1 - \th^*_2 }.                                             \label{eq:a2b}
\end{align}
\end{lemma}

\begin{proof}
Set $d=2$ in \eqref{eq:ada0} to get \eqref{eq:a2b}.
To get \eqref{eq:a1b}, use Lemma \ref{lem:aithi} and \eqref{eq:a2b}.
\end{proof}

\section{The bond relation for Leonard pairs and Leonard systems}
\label{sec:bondLP}
\ifDRAFT {\rm sec:bondLP}. \fi

In Section \ref{sec:bond} we considered the bond relation for irreducible tridiagonal
matrices.
In this section we consider a version of the bond relation that applies to Leonard pairs and
Leonard systems.
We first consider the bond relation for Leonard systems.

\begin{defi}    \label{def:veeLS}    \samepage
\ifDRAFT {\rm def:veeLS}. \fi
For a Leonard system  $\Phi = (A; \{E_i\}_{i=0}^d; A^*; \{E^*_i\}_{i=0}^d)$ on $V$,
define $\cS = \sum_{i=0}^d (-1)^i E^*_i$.
\end{defi}

\begin{note}    \label{note:veeLS}    \samepage
\ifDRAFT {\rm note:veeLS}. \fi
Referring to Definition \ref{def:veeLS},
$\cS^2 = I$.
\end{note}

\begin{lemma}    \label{lem:veeLS0}    \samepage
\ifDRAFT {\rm lem:veeLS0}. \fi
Referring to Definition \ref{def:veeLS},
the sequence
\[
   \Phi^\vee = (- \cS A \cS^{-1}; \{\cS E_i \cS^{-1}\}_{i=0}^d; A^*; \{E^*_i\}_{i=0}^d)
\]
is a Leonard system on $V$.
Moreover the Leonard system $\Phi^\vee$ is isomorphic to the Leonard system
\begin{equation*}
   (-A; \{E_i\}_{i=0}^d; A^*; \{E^*_i\}_{i=0}^d).
\end{equation*}
\end{lemma}

\begin{proof}
We have $\cS A^* \cS^{-1} = A^*$ and $\cS E^*_i \cS^{-1} = E^*_i$ for $0 \leq i \leq d$.
By this and Lemma \ref{lem:LSiso} we get the result.
\end{proof}

\begin{note}    \label{note:veeLS2}    \samepage
\ifDRAFT {\rm note:veeLS2}. \fi
Referring to Lemma \ref{lem:veeLS0}, we have
$(\Phi^\vee)^\vee = \Phi$.
\end{note}

\begin{note}   \label{note:veeLS2b}    \samepage
\ifDRAFT {\rm note:veeLS2b}. \fi
Referring to Lemma \ref{lem:veeLS0},
assume that $\text{\rm Char}(\F)=2$.
Then $\cS = I$ and $\Phi^\vee = \Phi$.
\end{note}

\begin{defi}    \label{def:bondLS}    \samepage
\ifDRAFT {\rm def:bondLS}. \fi
Leonard systems $\Phi$ and $\Phi'$ on $V$ are said to be
{\em bonded} whenever $\Phi' = \Phi^\vee$.
\end{defi}

\begin{note}    \label{note:veeLS3}    \samepage
\ifDRAFT {\rm note:veeLS3}. \fi
The bond relation is a symmetric binary relation on the set of
all Leonard systems on $V$.
\end{note}

\begin{note}    \label{note:veeLS4}   \samepage
\ifDRAFT {\rm note:veeLS4}. \fi
For a Leonard system $\Phi$ on $V$,
there exists a unique Leonard system on $V$ that is bonded to $\Phi$.
\end{note}

\begin{lemma}    \label{lem:bondLS4}    \samepage
\ifDRAFT {\rm lem:bondLS4}. \fi
Let $\Phi$ denote a Leonard system on $V$.
If $\text{\rm Char}(\F) = 2$ then $\Phi$ is bonded to itself.
If $\text{\rm Char}(\F) \neq 2$ the the following are equivalent:
\begin{itemize}
\item[\rm (i)]
$\Phi$ is bonded to itself;
\item[\rm (ii)]
$\Phi$ is bipartite.
\end{itemize}
\end{lemma}

\begin{proof}
By Lemma \ref{lem:bond0c}.
\end{proof}

\begin{lemma}    \label{lem:parrayPhivee}    \samepage
\ifDRAFT {\rm lem:parrayPhivee}. \fi
For the Leonard system $\Phi$ in Definition \ref{def:veeLS} and the
Leonard system $\Phi^\vee$ in Lemma \ref{lem:veeLS0},
their parameter arrays are related as follows:
\[
 \begin{array}{cc}
  \text{\rm Leonard system}  &  \text{\rm parameter array}  
\\ \hline
 \Phi & (\{\th_i\}_{i=0}^d; \{\th^*_i\}_{i=0}^d; \{\vphi_i\}_{i=1}^d; \{\phi_i\}_{i=1}^d)   \rule{0mm}{3ex}
\\
 \Phi^\vee &  
       (\{- \th_i\}_{i=0}^d; \{\th^*_i\}_{i=0}^d; \{- \vphi_i\}_{i=1}^d; \{- \phi_i\}_{i=1}^d)  \rule{0mm}{3ex}
 \end{array}
\]
\end{lemma}

\begin{proof}
By Lemmas \ref{lem:affineparam} and \ref{lem:veeLS0}.
\end{proof}

We now introduce the bond relation for Leonard pairs.

\begin{lemma}    \label{lem:veeLS1}    \samepage
\ifDRAFT {\rm lem:veeLS1}. \fi
Referring to Definition \ref{def:veeLS},
$\cS^\downarrow = (-1)^d\, \cS$ and $\,\cS^\Downarrow = \cS$.
\end{lemma}

\begin{proof}
By Definition \ref{def:veeLS}.
\end{proof}

\begin{lemma}    \label{lem:veeLS2}   \samepage
\ifDRAFT {\rm lem:veeLS2}. \fi
Let $A,A^*$ denote a Leonard pair on $V$.
Let $\Phi$ denote an associated Leonard system and let $\cS$ be
from Definition \ref{def:veeLS}.
Then the map $\End \to \End$, $X \mapsto \cS X \cS^{-1}$
is independent of the choice of $\Phi$.
\end{lemma}

\begin{proof}
By Lemmas \ref{lem:LSequiv} and \ref{lem:veeLS1}.
\end{proof}

In view of Lemma \ref{lem:veeLS2} we make a definition.

\begin{defi}    \label{def:veeLP}    \samepage
\ifDRAFT {\rm def:veeLP}. \fi
Let $A,A^*$ denote a Leonard pair on $V$,
and let $\Phi$ denote an associated Leonard system.
Define $A^\vee = - \cS A \cS^{-1}$, where $\cS$
is from Definition \ref{def:veeLS}.
Note that $A^\vee$ is independent of the choice of $\Phi$.
\end{defi}

\begin{note}    \label{note:veeLP5}    \samepage
\ifDRAFT {\rm note:veeLP5}. \fi
Referring to Definition \ref{def:veeLP},
we have $(A^\vee)^\vee = A$.
\end{note}

\begin{note}    \label{note:veeLP6}    \samepage
\ifDRAFT {\rm note:veeLP6}. \fi
Referring to Definition \ref{def:veeLP},
assume that $\text{\rm Char}(\F)=2$.
Then $A^\vee = A$.
\end{note}

\begin{lemma}    \label{lem:veeLP}    \samepage
\ifDRAFT {\rm lem:veeLP}. \fi
Let $A,A^*$ denote a Leonard pair on $V$.
Then the pair $A^\vee, A^*$ is a Leonard pair on $V$
that is isomorphic to the Leonard pair $-A, A^*$.
\end{lemma}

\begin{proof}
Use Lemma \ref{lem:veeLS0}.
\end{proof}

\begin{note}    \label{note:veeLPSL}    \samepage
\ifDRAFT {\rm note:veeLPSL}. \fi
Referring to Lemma \ref{lem:veeLS0} and Definition \ref{def:veeLP},
the Leonard system $\Phi^\vee$ and the Leonrd pair $A^\vee, A^*$ are associated.
\end{note}

\begin{defi}    \label{def:bondLP}    \samepage
\ifDRAFT {\rm def:bondLP}. \fi
Leonard pairs $A,A^*$ and $B,B^*$ on $V$ are said to be {\em bonded}
whenever $A^* = B^*$ and $B = A^\vee$.
\end{defi}

\begin{note}  \label{note:bondLP}     \samepage
\ifDRAFT {\rm note:bondLP}. \fi
The bond relation is a symmetric binary relation on the set of all Leonard
pairs on $V$.
\end{note}

\begin{note}    \label{note:bondLPunique}    \samepage
\ifDRAFT {\rm bondLPunique}. \fi
For a Leonard pair $A,A^*$ on $V$,
there exists a unique Leonard pair on $V$ that is bonded to $A,A^*$.
\end{note}

\begin{lemma}    \label{lem:bondLP3}    \samepage
\ifDRAFT {\rm lem:bondLP3}. \fi
Let $A,A^*$ denote a Leonard pair on $V$.
If $\text{\rm Char}(\F) = 2$ then $A,A^*$ is bonded to itself.
If $\text{\rm Char}(\F) \neq 2$ the the following are equivalent:
\begin{itemize}
\item[\rm (i)]
$A,A^*$ is bonded to itself;
\item[\rm (ii)]
the Leonard pair $A,A^*$ is bipartite.
\end{itemize}
\end{lemma}

\begin{proof}
By Lemma \ref{lem:bond0c}.
\end{proof}

Consider parameter arrays \eqref{eq:twoparrays} over $\F$.
In Lemma \ref{lem:A=Bpre} we described a special case of the condition
(ii) in Lemma \ref{lem:simLS2c}.
Here is another special case.

\begin{lemma}    \label{lem:ABveepre}     \samepage
\ifDRAFT {\rm lem:ABveepre}. \fi
For parameter arrays \eqref{eq:twoparrays} over $\F$,
the following hold.
\begin{itemize}
\item[\rm (i)]
Assume that
$\vphi'_i = - \vphi_i$ and $\phi'_i = - \phi_i$  for $1 \leq i \leq d$.
Then there exists $\zeta \in \F$ such that
$\th'_i = \zeta - \th_i$ for $0 \leq i \leq d$.
\item[\rm (ii)]
Assume that
$\vphi'_i = - \phi_i$ and $\phi'_i = -  \vphi_i$ for $1 \leq i \leq d$.
Then there exists $\zeta \in \F$ such that
$\th'_i = \zeta - \th_{d-i}$ for $0 \leq i \leq d$.
\end{itemize}
\end{lemma}

\begin{proof}
Let $\Phi$ (resp.\ $\Phi'$) denote a Leonard system on $V$ that has parameter array
on the left (resp.\ right) in \eqref{eq:twoparrays}.
Now apply Lemma \ref{lem:A=Bpre} to $\Phi^\vee$ and $\Phi'$ using Lemma \ref{lem:parrayPhivee}.
\end{proof}

We describe how
the bond relation for Leonard pairs and Leonard systems is related to
the bond relation for irreducible tridiagonal matrices.
Let 
$\Phi = (A; \{E_i\}_{i=0}^d; A^*;$ $\{E^*_i\}_{i=0}^d)$
denote a Leonard system on $V$.
Fix a basis $\{v_i\}_{i=0}^d$ of $V$ that satisfies Definition \ref{def:LP}(i).
For $X \in \text{\rm End}(V)$ let $X^\flat \in \Matd$ represent $X$ with respect to $\{v_i\}_{i=0}^d$.
The map $\flat : \text{\rm End}(V) \to \Matd$, $X \mapsto X^\flat$ is an algebra isomorphism.
Recall the matrix $\S \in \Matd$ from Definition \ref{def:vee} and the element $\cS \in \text{\rm End}(V)$
 from Definition \ref{def:veeLS}.
Then the isomorphism $\flat$ sends $\cS \mapsto \S$.
Moreover, $(A^\flat)^\vee = (A^\vee)^\flat$,
where $(A^\flat)^\vee$ is computed using Definition \ref{def:vee}
and $A^\vee$ is from Definition \ref{def:veeLP}.

\section{Compatibility and companions for Leonard pairs}
\label{sec:compatible}
\ifDRAFT {\rm sec:compatible}. \fi

In this section we introduce the notion of compatibility and companion for Leonard pairs.

\begin{defi}     \label{def:compatible}   \samepage
\ifDRAFT {\rm def:compatible}. \fi
Leonard pairs $A,A^*$ and $B,B^*$ on $V$ are said to be
{\em compatible} whenever $A^* = B^*$ and $[A, A^*] = [B,B^*]$.
\end{defi}

\begin{defi}    \label{def:companion}    \samepage
\ifDRAFT {\rm def:companion}. \fi
For a Leonard pair $A,A^*$ on $V$,
by a {\em companion of $A,A^*$} we mean an element $K \in \b{ A^* }$
such that $A-K, A^*$ is a Leonard pair on $V$.
\end{defi}

The Definitions \ref{def:compatible} and \ref{def:companion} are related as follows.

\begin{lemma}    \label{lem:compatible0}    \samepage
\ifDRAFT {\rm lem:compatible0}. \fi
For a Leonard pair $A,A^*$ on $V$,
the following hold.
\begin{itemize}
\item[\rm (i)]
For a companion $K$ of $A,A^*$, define $B = A-K$.
Then $B,A^*$ is a Leonard pair on $V$ that is compatible with $A,A^*$.
\item[\rm (ii)]
For a Leonard pair $B, A^*$ on $V$ that is compatible with $A,A^*$,
define $K = A-B$.
Then $K$ is a companion of $A,A^*$.
\end{itemize}
\end{lemma}

\begin{proof}
Use Lemma \ref{lem:As2}.
\end{proof}

\begin{example}    \label{ex:comp}    \samepage
\ifDRAFT {\rm ex:comp}. \fi
Let $A,A^*$ denote a Leonard pair on $V$.
Then $A,A^*$ is compatible with itself.
Moreover $K=0$ is a companion of $A,A^*$.
\end{example}

Let $A,A^*$ denote a Leonard pair on $V$.
In this paper we find every Leonard pair $B,A^*$ on $V$ that is compatible with $A,A^*$.
By Lemma \ref{lem:compatible0}, this is equivalent to finding all the companions of $A,A^*$.

\begin{lemma}    \label{lem:compatible0b}    \samepage
\ifDRAFT {\rm lem:compatible0b}. \fi
For Leonard pairs $A,A^*$ and $B,A^*$ on $V$,
the following {\rm (i)--(iii)} are equivalent:
\begin{itemize}
\item[\rm (i)]
$A,A^*$ and $B,A^*$ are compatible;
\item[\rm (ii)]
$A-B \in \b{ A^*}$;
\item[\rm (iii)]
$A-B$ commutes with $A^*$.
\end{itemize}
\end{lemma}

\begin{proof}
Use Lemma \ref{lem:As2}.
\end{proof}

\begin{lemma}    \label{lem:ABK}    \samepage
\ifDRAFT {\rm lem:ABK}. \fi
For Leonard pairs $A,A^*$ and $B, A^*$ on $V$,
define $K = A - B$.
Then the following {\rm (i)--(iii)} are equivalent:
\begin{itemize}
\item[\rm (i)]
$K$ commutes with $A^*$;
\item[\rm (ii)]
$K$ is a companion of $A,A^*$;
\item[\rm (iii)]
$-K$ is a companion of $B,A^*$.
\end{itemize}
\end{lemma}

\begin{proof}
(i) $\Rightarrow$ (ii)
We have $K \in \b{A^*}$ by Lemma \ref{lem:As2}.
The pair $A-K, A^*$ is a Leonard pair since $B=A-K$.
By these comments and Definition \ref{def:companion}, $K$ is a
companion of $A,A^*$.

(ii) $\Rightarrow$ (i)
Since $K \in \b{A^*}$ by Definition \ref{def:companion}.

(i) $\Leftrightarrow$ (iii)
Similar to the proof of (i) $\Leftrightarrow$ (ii).
\end{proof}

\begin{lemma}    \label{lem:compatible1}    \samepage
\ifDRAFT {\rm lem:compatible1}. \fi
The compatible relation is an equivalence relation on the
set of all Leonard pairs on $V$.
\end{lemma}

\begin{proof}
By Definition \ref{def:compatible}.
\end{proof}

\begin{lemma}    \label{lem:compatible3}    \samepage
\ifDRAFT {\rm lem:compatible3}. \fi
For compatible Leonard pairs $A,A^*$ and $B,A^*$ on $V$,
and for scalars $\xi$, $\zeta$, $\xi^*$, $\zeta^*$ in $\F$ with $\xi \xi^* \neq 0$,
the Leonard pairs
 $\xi A+ \zeta I, \xi^* A^* + \zeta^* I$ and
$\xi B + \zeta I, \xi^* A^* + \zeta^* I$ are compatible.
\end{lemma}

\begin{proof}
By Definition \ref{def:compatible}.
\end{proof}

\begin{lemma}    \label{lem:trivial}    \samepage
\ifDRAFT {\rm lem:trivial}. \fi
Let $A,A^*$ denote a Leonard pair on $V$,
and let $K$ denote a companion of $A,A^*$.
Then $K + \zeta I$ is companion of $A,A^*$ for $\zeta \in \F$.
\end{lemma}

\begin{proof}
By Lemma \ref{lem:affineLP} and Definition \ref{def:companion}.
\end{proof}

\section{The set $\Omega$}
\label{sec:Omega}
\ifDRAFT {\rm sec:Omega}. \fi

In our main results, we find it illuminating to represent a Leonard pair
as an ordered pair of matrices,
the first one normalized irreducible tridiagonal and the second one diagonal.
Such a representation is guaranteed by Lemma \ref{lem:aixi0}.
To describe our main results using such a representation,
we introduce a certain set of matrices denoted $\Omega$.
In this section we define $\Omega$ and give some basic facts about it.

\begin{notation}    \label{notation0}
\ifDRAFT {\rm notation0}. \fi
Assume that $V = \F^{d+1}$.
Let $\{\th^*_i\}_{i=0}^d$ denote mutually distinct scalars in $\F$,
and define $A^* \in \Matd$ by
\[
   A^* = \text{\rm diag}(\th^*_0, \th^*_1, \ldots, \th^*_d).
\]
Note that $A^*$ is multiplicity-free, and
$\b{ A^* }$ consists of the diagonal matrices in $\Matd$.
For $0 \leq i \leq d$ let $E^*_i$ denote the diagonal matrix in $\Matd$ that
has $(i,i)$-entry $1$ and all other entries $0$.
Note that $E^*_i$ is the primitive idempotent of $A^*$ associated with $\th^*_i$.
\end{notation}

For the rest of this section, Notation \ref{notation0} is in effect.

\begin{defi}   \label{def:Omega}    \samepage
\ifDRAFT {\rm def:Omega}. \fi
Let the set $\Omega$ consist of the matrices $A \in \Matd$ such that:
\begin{itemize}
\item[\rm (i)]
$A$ is normalized irreducible tridiagonal;
\item[\rm (ii)]
$A,A^*$ is a Leonard pair on $V$.
\end{itemize}
\end{defi}

\begin{lemma}     \label{lem:Omega00}    \samepage
\ifDRAFT {\rm lem:Omega00}. \fi
Let $A \in \Omega$.
Let $B \in \Matd$ such that $B,A^*$ is a Leonard pair on $V$
that is compatible with $A,A^*$.
Then $B \in \Omega$.
\end{lemma}

\begin{proof}
By Lemma \ref{lem:compatible0b},
$A-B \in \b{A^*}$ and so $A-B$ is diagonal.
Thus $A_{i,j} = B_{i,j}$ if $i \neq j$  $(0 \leq i,j \leq d)$.
By this and since $A$ is normalized irreducible tridiagonal,
we find that $B$ is normalized irreducible tridiagonal.
Thus $B \in \Omega$.
\end{proof}

\begin{lemma}    \label{lem:affineOmega}    \samepage
\ifDRAFT {\rm lem:affineOmega}. \fi
For $A \in \Omega$ and $\zeta \in \F$,
the matrix $A + \zeta I$ is contained in $\Omega$.
\end{lemma}

\begin{proof}
The matrix $A + \zeta I$ is normalized irreducible tridiagonal.
By Lemma \ref{lem:affineLP}
the pair $A + \zeta I, A^*$ is a Leonard pair.
The result follows.
\end{proof}

\begin{lemma}    \label{lem:ABOmega}     \samepage
\ifDRAFT {\rm lem:ABOmega}. \fi
For $A$, $B \in \Omega$ the following are equivalent:
\begin{itemize}
\item[\rm (i)]
the Leonard pairs $A,A^*$ and $B,A^*$ are isomorphic;
\item[\rm (ii)]
$A=B$.
\end{itemize}
\end{lemma}

\begin{proof}
(i) $\Rightarrow$ (ii)
By Lemma \ref{lem:SASinv} there exists an invertible $S \in \Matd$
such that $S A S^{-1} = B$ and $S A^* S^{-1} = A^*$.
The matrix $S$ is diagonal by Lemma \ref{lem:As2},
so $A$ and $B$ are diagonally equivalent.
Consequently $A=B$ in view of Lemma \ref{lem:trid4}.

(ii) $\Rightarrow$ (i)
Clear.
\end{proof}

\section{The bond relation on $\Omega$}
\label{sec:bondOmega}
\ifDRAFT {\rm sec:bondOmega}. \fi

In the last paragraph of Section \ref{sec:bondLP},
we explained how the bond relation for Leonard pairs and systems is related
to the bond relation for irreducible tridiagonal matrices.
In this section we discuss these bond relations in the context of the set $\Omega$.

Throughout this section Notation \ref{notation0} is in effect.

\begin{lemma}    \label{lem:veecoincide}    \samepage
\ifDRAFT {\rm lem:veecoincide}. \fi
For $A \in \Omega$ the following are the same:
\begin{itemize}
\item[\rm (i)]
the matrix $A^\vee$ from Definition \ref{def:vee};
\item[\rm (ii)]
the matrix $A^\vee$ from Definition \ref{def:veeLP}.
\end{itemize}
Moreover, $A^\vee \in \Omega$.
\end{lemma}

\begin{proof}
The matrices (i), (ii) are the same by the last paragraph of Section \ref{sec:bondLP}.
We now show that $A^\vee \in \Omega$.
By Definition \ref{def:vee}, $A^\vee$ is irreducible tridiagonal.
By Lemma \ref{lem:bondnormalized}, $A^\vee$ is normalized.
By Lemma \ref{lem:veeLP} the pair $A^\vee, A^*$ is a Leonard pair.
By these comments and Definition \ref{def:Omega}, $A^\vee \in \Omega$.
\end{proof}

\begin{lemma}    \label{lem:twobond}    \samepage
\ifDRAFT {\rm lem:twobond}. \fi
For $A$, $B \in \Omega$ the following are equivalent:
\begin{itemize}
\item[\rm (i)]
$A$ and $B$ are bonded in the sense of Definition \ref{def:bond};
\item[\rm (ii)]
the Leonard pairs $A,A^*$ and $B,A^*$ are bonded in the sense of Definition \ref{def:bondLP}.
\end{itemize}
\end{lemma}

\begin{proof}
By Lemma \ref{lem:veecoincide}.
\end{proof}

We mention a result for later use.

\begin{lemma}    \label{lem:bondcompOmega}    \samepage
\ifDRAFT {\rm lem:bondcompOmega}. \fi
For $A$, $B \in \Omega$ let $K=A-B$.
Then the following are equivalent:
\begin{itemize}
\item[\rm (i)]
$A$ and $B$ are bonded;
\item[\rm (ii)]
$K$ is diagonal with diagonal entries $K_{i,i} = 2 A_{i,i}$ for $0 \leq i \leq d$.
\end{itemize}
\end{lemma}

\begin{proof}
By Lemma \ref{lem:bond000}.
\end{proof}

\section{The compatibility relation on $\Omega$}
\label{sec:compOmega}
\ifDRAFT {\rm sec:compOmega}. \fi

In Section \ref{sec:compatible} we introduced the compatibility relation
for Leonard pairs.
In this section we discuss this relation in the context of the set $\Omega$.

Throughout this section,
Notation \ref{notation0} is in effect.

\begin{defi}     \label{def:comp}    \samepage
\ifDRAFT {\rm def:comp}. \fi
Matrices $A$ and $B$ in $\Omega$ are said to be {\em compatible}
whenever the Leonard pairs $A,A^*$ and $B,A^*$ are compatible
in the sense of Definition \ref{def:compatible}.
\end{defi}

\begin{defi}     \label{def:companionOmega}     \samepage
\ifDRAFT {\rm def:companionOmega}. \fi
For $A \in \Omega$,
by a {\em companion of A} we mean a companion of the Leonard pair
$A,A^*$.
\end{defi}

\begin{lemma}    \label{lem:compOmega}     \samepage
\ifDRAFT {\rm lem:compOmega}. \fi
For $A \in \Omega$ the following hold.
\begin{itemize}
\item[\rm (i)]
For a companion $K$ of $A$, define $B = A-K$.
Then $B$ is contained in $\Omega$ and compatible with $A$.
\item[\rm (ii)]
For $B \in \Omega$ that is compatible with $A$, define $K = A-B$.
Then $K$ is a companion of $A$.
\end{itemize}
\end{lemma}

\begin{proof}
(i)
By Lemma \ref{lem:compatible0}(i)
the pair $B,A^*$ is a Leonard pair on $V$ that is compatible with $A,A^*$.
By this and Lemma \ref{lem:Omega00}, $B \in \Omega$.
By these comments and Definition \ref{def:comp}, $B$ is compatible with $A$.

(ii)
By Lemma \ref{lem:compatible0}(ii) and Definition \ref{def:companionOmega}.
\end{proof}

\begin{example}    \label{ex:K0}    \samepage
\ifDRAFT {\rm ex:K0}. \fi
Every matrix $A \in \Omega$ is compatible with $A$,
and $K=0$ is a companion of $A$.
\end{example}

\begin{lemma}    \label{lem:comp}    \samepage
\ifDRAFT {\rm lem:comp}. \fi
For $A$, $B \in \Omega$ the following {\rm (i)--(iii)} are equivalent:
\begin{itemize}
\item[\rm (i)]
$A$ and $B$ are compatible;
\item[\rm (ii)]
$A-B$ is diagonal;
\item[\rm (iii)]
$A_{i-1,i} = B_{i-1,i}$ for $1 \leq i \leq d$.
\end{itemize}
\end{lemma}

\begin{proof}
(i) $\Leftrightarrow$ (ii)
By Lemma \ref{lem:compatible0b}.

(ii) $\Leftrightarrow$ (iii)
Use the fact that each of $A$, $B$ is normalized irreducible tridiagonal.
\end{proof}

\begin{lemma}    \label{lem:compOmega2}    \samepage
\ifDRAFT {\rm lem:compOmega2}. \fi
For $A$, $B \in \Omega$ let $K=A-B$.
Then the following {\rm (i)--(iii)} are equivalent:
\begin{itemize}
\item[\rm (i)]
$K$ is diagonal;
\item[\rm (ii)]
$K$ is a companion of $A$;
\item[\rm (iii)]
$-K$ is a companion of $B$.
\end{itemize}
\end{lemma}

\begin{proof}
By Lemma \ref{lem:ABK}.
\end{proof}

\begin{lemma}   \label{lem:compequiv}     \samepage
\ifDRAFT {\rm lem:compequiv}. \fi
The compatibility relation on $\Omega$ is an equivalence relation.
\end{lemma}

\begin{proof}
By Lemma \ref{lem:compatible1}.
\end{proof}

\begin{lemma}     \label{lem:comp3}    \samepage
\ifDRAFT {\rm lem:comp3}. \fi
For compatible matrices $A$, $B \in \Omega$
and $\zeta \in \F$,
the matrices $A + \zeta I$ and $B + \zeta I$ are contained in $\Omega$.
Moreover, these matrices are compatible.
\end{lemma}

\begin{proof}
By Lemma \ref{lem:Omega00},
the matrices $A + \zeta I$ and $B + \zeta I$ are contained in $\Omega$.
These matrices are compatible by Lemma \ref{lem:compatible3}.
\end{proof}

\begin{lemma}    \label{lem:Kzeta}    \samepage
\ifDRAFT {\rm lem:Kzeta}. \fi
For $A \in \Omega$ let $K$ denote a companion of $A$.
Then $K + \zeta I$ is a companion of $A$ for $\zeta \in \F$.
\end{lemma}

\begin{proof}
By Lemma \ref{lem:trivial}. 
\end{proof}

\section{A characterization of the compatibility relation and the bond relation in terms of the parameter array}
\label{sec:characterize}
\ifDRAFT {\rm sec:characterize}. \fi

In this section we characterize the compatibility relation and the bond relation
in terms of the parameter array.
Throughout this section, the following notation is in effect.

\begin{notation}     \label{notation:all}     \samepage
\ifDRAFT {\rm notation:all}. \fi
Assume that $V = \F^{d+1}$.
Let $\{\th^*_i\}_{i=0}^d$ denote mutually distinct scalars in $\F$,
and define $A^* \in \Matd$ by
\[
    A^* = \text{\rm diag} ( \th^*_0, \th^*_1, \ldots, \th^*_d).
\]
For $0 \leq i \leq d$ let $E^*_i$ denote the matrix in $\Matd$ that
has $(i,i)$-entry $1$ and all other entries $0$.
Let the set $\Omega$ be from Definition \ref{def:Omega}.
Let $A$, $B$ denote matrices in $\Omega$.
Let $\{E_i\}_{i=0}^d$ (resp.\ $\{E'_i\}_{i=0}^d$) denote a standard ordering of
the primitive idempotents of $A$ (resp.\ $B$) for the Leonard pair $A,A^*$
(resp.\ $B,A^*$).
Define
\begin{align*}
 \Phi &= (A; \{E_i\}_{i=0}^d; A^*; \{E^*_i\}_{i=0}^d),
&
 \Phi' &= (B; \{E'_i\}_{i=0}^d; A^*; \{E^*_i\}_{i=0}^d),
\end{align*}
and observe that each of $\Phi$, $\Phi'$ is a Leonard system on $V$.
Let
\begin{align*}
& (\{\th_i\}_{i=0}^d; \{\th^*_i\}_{i=0}^d; \{\vphi_i\}_{i=1}^d; \{\phi_i\}_{i=1}^d),
&&
 (\{\th'_i\}_{i=0}^d; \{\th^*_i\}_{i=0}^d; \{\vphi'_i\}_{i=1}^d; \{\phi'_i\}_{i=1}^d)
\end{align*}
denote the parameter array of $\Phi$ and $\Phi'$, respectively.
Let the scalars $\{a_i\}_{i=0}^d$, $\{x_i\}_{i=1}^d$ 
(resp.\ $\{a'_i\}_{i=0}^d$, $\{x'_i\}_{i=1}^d$) be from Definitions \ref{def:ai}, \ref{def:xi}
for $\Phi$ (resp.\ $\Phi'$).
For the case of $d \geq 3$,
let $\kappa$ (resp.\ $\kappa'$) denote the invariant value for $A,A^*$
(resp.\ $B,A^*$).
\end{notation}

We now present our first main result,
in which we characterize the compatibility relation 
in terms of the parameter array.

\begin{theorem}    \label{thm:main}    \samepage
\ifDRAFT {\rm thm:main}. \fi
The following {\rm (i)--(iii)} are equivalent:
\begin{itemize}
\item[\rm (i)]
$A$ and $B$ are compatible;
\item[\rm (ii)]
$x_i = x'_i \quad (1 \leq i \leq d)$.
\item[\rm (iii)]
$\vphi_i \phi_i = \vphi'_i \phi'_i \quad (1 \leq i \leq d)$.
\end{itemize}
\end{theorem}

\begin{proof}
(i) $\Rightarrow$ (ii)
By Lemma \ref{lem:comp},
$A_{i-1, i} = B_{i-1, i}$ for $1 \leq i \leq d$.
By this and Corollary \ref{cor:aixi0} we get (ii).

(ii) $\Rightarrow$ (i)
By the construction, $B_{i,i-1} = A_{i,i-1} = 1$ for $1 \leq i \leq d$.
By Corollary \ref{cor:aixi0} and (ii), we have  $B_{i-1,i} = A_{i-1,i}$ for $1 \leq i \leq d$.
Now use Lemma \ref{lem:comp}.

(ii) $\Leftrightarrow$ (iii)
By Lemma \ref{lem:simLS2c}.
\end{proof}

Next we consider some special cases of Theorem \ref{thm:main}.

\begin{prop}    \label{prop:A=B}    \samepage
\ifDRAFT {\rm prop:A=B}. \fi
The following are equivalent:
\begin{itemize}
\item[\rm (i)]
there exists $\zeta \in \F$ such that $B = A + \zeta I$;
\item[\rm (ii)]
one of the following \eqref{eq:A=B1pre}, \eqref{eq:A=B2pre} holds:
\begin{align}
 \vphi'_i &= \vphi_i,  &  \phi'_i &= \phi_i &&  (1 \leq i \leq d),    \label{eq:A=B1pre}
\\
 \vphi'_i &= \phi_i,  &  \phi'_i &= \vphi_i &&  (1 \leq i \leq d).    \label{eq:A=B2pre}
\end{align}
\end{itemize}
\end{prop}

\begin{proof}
(i) $\Rightarrow$ (ii)
Consider the Leonard system
\begin{equation}
 (A + \zeta I; \{E_i\}_{i=0}^d; A^*; \{E^*_i\}_{i=0}^d).    \label{eq:A=Baux}
\end{equation}
By Lemma \ref{lem:affineparam} the Leonard system \eqref{eq:A=Baux}
has parameter array
\[
(\{\th_i + \zeta\}_{i=0}^d; \{\th^*_i\}_{i=0}^d; \{\vphi_i\}_{i=1}^d; \{\phi_i\}_{i=1}^d).
\]
The Leonard system \eqref{eq:A=Baux} is associated with $B,A^*$ and is therefore
equal to one of $\Phi'$, $\Phi^{\prime \Downarrow}$ by Lemma \ref{lem:LSequiv}.
First assume that the Leonard system \eqref{eq:A=Baux} is equal to $\Phi'$.
Then \eqref{eq:A=B1pre} holds.
Next assume that the Leonard system \eqref{eq:A=Baux} is equal to $\Phi^{\prime \Downarrow}$.
By Lemma \ref{lem:parrayrelative}, $\Phi'^\Downarrow$ has parameter array
\[
(\{\th'_{d-i}\}_{i=0}^d; \{\th^*_i\}_{i=0}^d; \{\phi'_i\}_{i=1}^d; \{\vphi'_i\}_{i=1}^d).
\]
By these comments \eqref{eq:A=B2pre} holds.

(ii) $\Rightarrow$ (i)
First assume that \eqref{eq:A=B1pre} holds.
By Lemma \ref{lem:A=Bpre}(i) there exists $\zeta \in\F$ such that
$\th'_i = \th_i + \zeta$ for $0 \leq i \leq d$.
So the parameter array of $\Phi'$ is
\[
(\{\th_i + \zeta\}_{i=0}^d; \{\th^*_i\}_{i=0}^d; \{\vphi_i\}_{i=1}^d; \{\phi_i\}_{i=1}^d).
\]
By this and Lemma \ref{lem:affineparam} the Leonard system 
$ (A + \zeta I; \{E_i\}_{i=0}^d; A^*; \{E^*_i\}_{i=0}^d)$
has the same parameter array as $\Phi'$ and is therefore isomorphic to $\Phi'$ by
Lemma \ref{lem:classify}.
Thus the Leonard pair $B,A^*$ is isomorphic to $A+\zeta I, A^*$.
By this and Lemma \ref{lem:ABOmega}, $B = A + \zeta I$.
Next assume that \eqref{eq:A=B2pre} holds.
By Lemma \ref{lem:A=Bpre}(ii) there exists $\zeta \in \F$ such that
$\th'_i = \th_{d-i} + \zeta$ for $0 \leq i \leq d$.
So the parameter array of $\Phi'$ is
\[
(\{\th_{d-i} + \zeta\}_{i=0}^d; \{\th^*_i\}_{i=0}^d; \{\phi_i\}_{i=1}^d; \{\vphi_i\}_{i=1}^d).
\]
By Lemma \ref{lem:parrayrelative}(iii)
the parameter array of $(A; \{E_{d-i}\}_{i=0}^d; A^*; \{E^*_i\}_{i=0}^d)$
is
\[
(\{\th_{d-i}\}_{i=0}^d; \{\th^*_i\}_{i=0}^d; \{\phi_i\}_{i=1}^d; \{\vphi_i\}_{i=1}^d).
\]
By these comments and Lemma \ref{lem:affineparam}
the Leonard system $(A+\zeta I; \{E_{d-i}\}_{i=0}^d; A^*; \{E^*_i\}_{i=0}^d)$
has the same parameter array as $\Phi'$ and 
is therefore isomorphic to $\Phi'$ by Lemma \ref{lem:classify}.
Thus the Leonard pair $B,A^*$ is isomorphic to $A+\zeta I, A^*$.
By this and Lemma \ref{lem:ABOmega}, $B = A + \zeta I$.
\end{proof}

\begin{prop}    \label{prop:ABvee}    \samepage
\ifDRAFT {\rm prop:ABvee}. \fi
The following are equivalent:
\begin{itemize}
\item[\rm (i)]
there exists $\zeta \in \F$ such that $B = A^\vee + \zeta I$;
\item[\rm (ii)]
one of the following \eqref{eq:ABvee1pre}, \eqref{eq:ABvee2pre} holds:
\begin{align}
 \vphi'_i &= - \vphi_i,  &  \phi'_i &= - \phi_i &&  (1 \leq i \leq d),    \label{eq:ABvee1pre}
\\
 \vphi'_i &= - \phi_i,  &  \phi'_i &= - \vphi_i &&  (1 \leq i \leq d).    \label{eq:ABvee2pre}
\end{align}
\end{itemize}
\end{prop}

\begin{proof}
The result is obtained by applying Proposition \ref{prop:A=B} to $\Phi^\vee$ and $\Phi'$
using Lemma \ref{lem:parrayPhivee}.
The result can also be obtained using Lemma \ref{lem:ABveepre}.
\end{proof}

\section{The type of a Leonard pair and Leonard system}
\label{sec:type}
\ifDRAFT {\rm sec:type}. \fi

For the rest of this paper, we assume that $\F$ is algebraically closed.
In this section we recall from \cite{NT:balanced} the type of a Leonard pair and Leonard system.
Consider a parameter array over $\F$:
\begin{equation}
  (\{\th_i\}_{i=0}^d; \{\th^*_i\}_{i=0}^d; \{\vphi_i\}_{i=1}^d; \{\phi_i\}_{i=1}^d).   \label{eq:parray0}
\end{equation}
For the case $d \geq 3$, let $\beta$ denote the fundamental constant
of \eqref{eq:parray0}.

\begin{defi}     {\rm (See \cite[Section 4]{NT:balanced}.) }
\label{def:type}    \samepage
\ifDRAFT {\rm def:type}. \fi
To the parameter array \eqref{eq:parray0} we assign the {\em type} as follows.
\[
\begin{array}{c|c}
\text{\rm type} & \text{\rm description}
\\ \hline
\text{\rm O} &  1 \leq d \leq 2       \rule{0mm}{2.8ex}
\\
\text{\rm I} & d \geq 3,  \;\;  \beta \neq 2, \;\; \beta \neq -2          \rule{0mm}{2.5ex}
\\
\text{\rm II} & d \geq 3, \;\;  \beta = 2, \;\; \text{\rm Char}(\F) \neq 2           \rule{0mm}{2.5ex}
\\
\text{\rm III}^+ & d \geq 3, \;\;     \beta = -2, \;\; \text{\rm Char}(\F) \neq 2, 
                             \;\; \text{\rm $d$ even}                             \rule{0mm}{2.5ex}
\\
\text{\rm III}^- &  d \geq 3, \;\;    \beta = -2, \;\; \text{\rm Char}(\F) \neq 2, 
\;\; \text{\rm $d$ odd}                                                     \rule{0mm}{2.5ex}
\\
\text{\rm IV} & d \geq 3, \;\;   \beta=2, \;\; \text{\rm Char}(\F) = 2               \rule{0mm}{2.5ex}
\end{array}
\]
\end{defi}

\begin{defi}    \label{def:typeLS}    \samepage
\ifDRAFT {\rm def:typeLS}. \fi
The type of a given Leonard system is the type of the
associated parameter array.
The type of a given Leonard pair is the type of an associated Leonard system.
\end{defi}

\begin{lemma}    {\rm (See \cite[Lemma 9.3]{T:Leonard}.) }
\label{lem:q}     \samepage
\ifDRAFT {\rm lem:q}. \fi
The following {\rm (i)--(v)} hold.
\begin{itemize}
\item[\rm (i)]
For type I,
$q^{2i} \neq 1$ for $1 \leq i \leq d$, where $\beta = q^2 + q^{-2}$.
\item[\rm (ii)]
For type II,
$\text{\rm Char}(\F)$ is equal to $0$ or greater than $d$.
\item[\rm (iii)]
For type III$^+$, 
$\text{\rm Char}(F)$ is equal to $0$ or greater than $d/2$.
\item[\rm (iv)]
For type III$^-$,
$\text{\rm Char}(F)$ is equal to $0$ or greater than $(d-1)/2$.
\item[\rm (v)]
For type IV,
$d=3$.
\end{itemize}
\end{lemma}

\section{A refinement of Theorem \ref{thm:main}}
\label{sec:refine}
\ifDRAFT {\rm sec:refine}. \fi

In this section we present our second main result, which is Theorem \ref{thm:kappamain}.
This result is a refinement of Theorem \ref{thm:main}.
The following proposition will be used to prove Theorem \ref{thm:kappamain}.

\begin{prop}    \label{prop:RST} 
\ifDRAFT {\rm prop:RST}. \fi
Assume that $d \geq 3$.
Let 
\begin{equation}
(\{\th_i\}_{i=0}^d; \{\th^*_i\}_{i=0}^d; \{\vphi_i\}_{i=1}^d; \{\phi_i\}_{i=1}^d)   \label{eq:parray11}
\end{equation}
denote a parameter array over $\F$ with fundamental constant $\beta$.
Then for $1 \leq i \leq d$ we have
\begin{align}
\frac{ \vphi_i \phi_i } { (\th^*_{i-1} - \th^*_d) (\th^*_i - \th^*_0 ) }
 &= \frac { R_i  \vphi_1 \phi_1 } { (\th^*_i - \th^*_0) ( \th^*_0 - \th^*_d) }
+ \frac{ S_i \vphi_d \phi_d } { (\th^*_d - \th^*_{i-1} ) (\th^*_0 - \th^*_d) }
+ T_i \kappa,                                     \label{eq:RST}
\end{align}
where $\kappa$ is the invariant value for \eqref{eq:parray11},
and $R_i$, $S_i$, $T_i$ are given below.
\begin{itemize}
\item[] 
Type I: 
\begin{align}
R_i &= \frac{ (q^i - q^{-i})^2 (q^{d-i}- q^{i-d} ) (q^{d-i+1} - q^{i-d-1}) }
                { (q - q^{-1})^2 (q^d - q^{-d}) ( q^{d-1} - q^{1-d}) },             \label{eq:Rtype1}
\\
S_i &= \frac{ (q^{d-i+1} - q^{i-d-1})^2 (q^i - q^{-i}) (q^{i-1} - q^{1-i}) }
                { (q - q^{-1})^2 (q^d - q^{-d}) ( q^{d-1} - q^{1-d}) },             \label{eq:Stype1}
\\
T_i &= \frac{ (q^i - q^{-i} ) (q^{i-1} - q^{1-i} )(q^{d-i} - q^{i-d} ) (q^{d-i+1} - q^{i-d-1}) }
                { (q-q^{-1})^2 (q^2 - q^{-2})^2 },                                      \label{eq:Ttype1}
\end{align}
where $\beta = q^2 + q^{-2}$.
\item[]
Type II:
\begin{align}
R_i &= \frac{ i^2 (d-i) (d-i+1) } { d (d-1) },                        \label{eq:Rtype2}
\\
S_i &= \frac{ i (i-1) (d-i+1)^2 } { d (d-1) },                        \label{eq:Stype2}
\\
T_i &= \frac{ i (i-1)(d-i)(d-i+1) } { 4 }.                              \label{eq:Ttype2}
\end{align}
\item[]
Type III$^+$:
\begin{align}
R_i &= 
\begin{cases}
 0  & \text{\rm if $i$ is even},
 \\
 (d-i+1)/d &  \text{\rm if $i$ is odd},
\end{cases}                                          \label{eq:Rtype3+}
\\
S_i &=
\begin{cases} 
 i/d & \text{\rm if $i$ is even},
 \\
 0   & \text{\rm if $i$ is odd},
\end{cases}                                          \label{eq:Stype3+}
\\
T_i &= 
\begin{cases}
 - i (d-i) / 4 & \text{\rm if $i$ is even},
\\
 - (i-1)(d-i+1)/4 & \text{\rm if $i$ is odd}.
\end{cases}                                               \label{eq:Ttype3+}
\end{align}
\item[]
Type III$^-$:
\begin{align}
R_i &= 
\begin{cases}
 0  & \text{\rm if $i$ is even},
 \\
 (d-i)/(d-1) &  \text{\rm if $i$ is odd},
\end{cases}                                                 \label{eq:Rtype3-}
\\
S_i &=
\begin{cases} 
 0  & \text{\rm if $i$ is even},
 \\
 (i-1)/(d-1)  & \text{\rm if $i$ is odd},
\end{cases}                                                  \label{eq:Stype3-}
\\
T_i &= 
\begin{cases}
  i (d-i+1) / 4 & \text{\rm if $i$ is even},
\\
 (i-1)(d-i)/4 & \text{\rm if $i$ is odd}.
\end{cases}                                                 \label{eq:Ttype3-}
\end{align}
\item[]
Type IV:
\begin{align}
R_1 &= 1, & R_2 &= 0,  &   R_3 &= 0,               \label{eq:Rtype4}
\\
S_1 &= 0, & S_2 &= 0, & S_3 &= 1,                  \label{eq:Stype4}
\\
T_1 &= 0, & T_2 &= 1, & T_3 &= 0.                  \label{eq:Ttype4}
\end{align}   
\end{itemize}
\end{prop}

\begin{note}    \label{note:denom}    \samepage
\ifDRAFT {\rm note:denom}. \fi
In \eqref{eq:Rtype1}--\eqref{eq:Ttype3-}
the denominators of $R_i$, $S_i$, $T_i$ are nonzero by Lemma \ref{lem:q}.
\end{note}

The proof of Proposition \ref{prop:RST} will be given in
Sections \ref{sec:parraytype1},
\ref{sec:parraytype2},
\ref{sec:parraytype3+},
\ref{sec:parraytype3-},
\ref{sec:parraytype4}.

\begin{theorem}    \label{thm:kappamain}    \samepage
\ifDRAFT {\rm thm:kappamain}. \fi
Referring to Notation \ref{notation:all},
the following {\rm (i)--(iii)} hold.
\begin{itemize}
\item[\rm (i)]
Assume that $d=1$.
Then $A$ and $B$ are compatible if and only if 
\[
\vphi_1 \phi_1 = \vphi'_1 \phi'_1.
\]
\item[\rm (ii)]
Assume that $d=2$.
Then $A$ and $B$ are compatible if and only if
\begin{align*}
\vphi_1 \phi_1 &= \vphi'_1 \phi'_1, &
\vphi_2 \phi_2 &= \vphi'_2 \phi'_2.
\end{align*}
\item[\rm (iii)]
Assume that $d \geq 3$.
Then $A$ and $B$ are compatible if and only if
\begin{align}
\kappa &= \kappa',  &
\vphi_1 \phi_1 &= \vphi'_1 \phi'_1, &
\vphi_d \phi_d &= \vphi'_d \phi'_d.                 \label{eq:kappakappad}
\end{align}
\end{itemize}
\end{theorem}

\begin{proof}
(i), (ii)
By Theorem \ref{thm:main}.

(iii)
First assume that $A$ and $B$ are compatible.
We show that \eqref{eq:kappakappad} holds.
By Theorem \ref{thm:main},
$\vphi_1 \phi_1 = \vphi'_1 \phi'_1$ and $\vphi_d \phi_d = \vphi'_d \phi'_d$.
To show $\kappa = \kappa'$, we invoke Proposition \ref{prop:RST}.
We consider the equation \eqref{eq:RST} for $\Phi$ and $\Phi'$.
Let the scalars $\{R_i\}_{i=1}^d$, $\{S_i\}_{i=1}^d$, $\{T_i\}_{i=1}^d$
(resp.\  $\{R'_i\}_{i=1}^d$, $\{S'_i\}_{i=1}^d$, $\{T'_i\}_{i=1}^d$) be from Proposition \ref{prop:RST}
for $\Phi$ (resp. $\Phi'$).
By \eqref{eq:Rtype1}--\eqref{eq:Ttype4}
we have
$R_i = R'_i$, $S_i = S'_i$, $T_i = T'_i$ for $1 \leq i \leq d$.
Using this and Theorem \ref{thm:main}(iii),
we compare the equation \eqref{eq:RST} for $\Phi$ with the equation \eqref{eq:RST} for $\Phi'$.
This gives $T_i \kappa = T_i \kappa'$ for $1 \leq i \leq d$.
Observe that in Proposition \ref{prop:RST}, for each type the scalars $\{T_i\}_{i=1}^d$ are not all zero 
by Lemma \ref{lem:q}.
By these comments we get $\kappa = \kappa'$.
We have shown that \eqref{eq:kappakappad} holds.
Next assume that \eqref{eq:kappakappad} holds.
By Proposition \ref{prop:RST} we obtain 
$\vphi_i \phi_i = \vphi'_i \phi'_i$ for $1 \leq i \leq d$.
By this and Theorem \ref{thm:main}, $A$ and $B$ are compatible.
\end{proof}

\begin{corollary}    \label{cor:kappamain}    \samepage
\ifDRAFT {\rm cor:kappamain}. \fi
Referring to Notation \ref{notation:all},
the following {\rm (i)--(iii)} hold.
\begin{itemize}
\item[\rm (i)]
Assume that $d=1$.
Then $A$ and $B$ are compatible if and only if 
\[
(a_0  - \th_0)(a_0 - \th_d) = (a'_0 - \th'_0)(a'_0 -\th'_d).
\]
\item[\rm (ii)]
Assume that $d=2$.
Then $A$ and $B$ are compatible if and only if
\begin{align*}
(a_0  - \th_0)(a_0 - \th_2) &= (a'_0 - \th'_0)(a'_0 -\th'_2),
\\
(a_2  - \th_0)(a_2 - \th_2) &= (a'_2 - \th'_0)(a'_2 -\th'_2).
\end{align*}
\item[\rm (iii)]
Assume that $d \geq 3$.
Then $A$ and $B$ are compatible if and only if
\begin{align*}
\kappa &= \kappa',
\\
(a_0  - \th_0)(a_0 - \th_d) &= (a'_0 - \th'_0)(a'_0 -\th'_d),
\\
(a_d  - \th_0)(a_d - \th_d) &= (a'_d - \th'_0)(a'_d -\th'_d). 
\end{align*}
\end{itemize}
\end{corollary}

\begin{proof}
By  \eqref{eq:vphi1b} and \eqref{eq:phi1b},
$\vphi_1 \phi_1 = \vphi'_1 \phi'_1$ if and only if 
\[
(a_0 - \th_0) (a_0 - \th_d) = (a'_0 - \th'_0)(a'_0 - \th'_d).
\]
By \eqref{eq:vphidb} and \eqref{eq:phidb},
$\vphi_d \phi_d = \vphi'_d \phi'_d$ if and only if
\[
(a_d - \th_0) (a_d - \th_d) = (a'_d - \th'_0)(a'_d - \th'_d).
\]
By these comments and Theorem \ref{thm:kappamain} we get the result.
\end{proof}

We finish this section with a comment.
Referring to Notation \ref{notation:all},
assume that $A$, $B$ are compatible and let $K=A-B$.
Recall from Lemma \ref{lem:compOmega} that
$K$ is a companion of $A$, and therefore diagonal by Lemma \ref{lem:compOmega2}.
By Lemma \ref{lem:ai0} we see that
\begin{align}
 K_{i,i} &= a_i - a'_i   &&   (0 \leq i \leq d).             \label{eq:Kaa}
\end{align}

\section{The companions for type O}
\label{sec:typeO}
\ifDRAFT {\rm sec:typeO}. \fi

In this section we describe the companions for type O.
We first consider the case $d=1$.

\begin{prop}     \label{prop:d1K}     \samepage
\ifDRAFT {\rm prop:d1K}. \fi
Referring to Notation \ref{notation:all}, assume that $d=1$
and $A$, $B$ are compatible.
Consider the companion $K=A-B$ from Definition \ref{def:companionOmega}.
Then
\begin{align*}
K_{0,0} &= a_0 - a'_0,
\\
K_{1,1} &= \th_0 + \th_1 - a_0 - \th'_0 - \th'_1 + a'_0. 
\end{align*}
\end{prop}

\begin{proof}
By \eqref{eq:Kaa} and Lemma \ref{lem:aithi}.
\end{proof}

Next we consider the case $d=2$.

\begin{prop}     \label{prop:d2K}     \samepage
\ifDRAFT {\rm prop:d2K}. \fi
Referring to Notation \ref{notation:all}, assume that $d=2$
and $A$, $B$ are compatible.
Consider the companion $K = A-B$ from Definition \ref{def:companionOmega}.
Then
\begin{align*}
K_{0,0} &= a_0 - a'_0,
\\
K_{1,1} &= \th_0 + \th_2 - a_0 - \th'_0 - \th'_2 + a'_0
               + \frac{ (a_0 - a'_0 - \th_1 +\th'_1)(\th^*_0 - \th^*_1) } { \th^*_1 - \th^*_2},
\\
K_{2,2} &= \frac{ (\th_1- \th'_1)(\th^*_0 - \th^*_2) - (a_0 - a'_0)(\th^*_0 - \th^*_1) }
                     { \th^*_1 - \th^*_2 }. 
\end{align*}
\end{prop}

\begin{proof}
By \eqref{eq:Kaa} and Lemma \ref{lem:d2param}.
\end{proof}

\section{The parameter arrays of type I}
\label{sec:parraytype1}
\ifDRAFT {\rm sec:parraytype1}. \fi

In this section,
we describe the parameter arrays  of type I.
We then prove Proposition \ref{prop:RST} for type I.
Throughout this section, assume that $d \geq 3$ and fix a nonzero $q \in \F$ such that $q^4 \neq 1$.

\begin{lemma}   \label{lem:type1param}   \samepage
\ifDRAFT {\rm lem:type1param}. \fi
For a sequence 
\begin{equation}
  (\delta, \mu, h, \delta^*, \mu^*, h^*, \tau)       \label{eq:type1pseq}
\end{equation}
of scalars in $\F$, define
\begin{align}
 \th_i &= \delta + \mu q^{2i-d} + h q^{d-2i}   && (0 \leq i \leq d),                    \label{eq:type1th}
\\
 \th^*_i &= \delta^* + \mu^* q^{2i-d} + h^* q^{d-2i}   && (0 \leq i \leq d),        \label{eq:type1ths}
\\
 \vphi_i &= (q^i-q^{-i})(q^{d-i+1}-q^{i-d-1})(\tau - \mu \mu^* q^{2i -d-1} - h h^* q^{d-2i+1})  
          &&        (1 \leq i \leq d),                                                          \label{eq:type1vphi}
\\
 \phi_i &=  (q^i-q^{-i})(q^{d-i+1}-q^{i-d-1})(\tau - h \mu^* q^{2i -d-1} - \mu h^* q^{d-2i+1})  
          &&     (1 \leq i \leq d).                                                            \label{eq:type1phi}
\end{align}
Then the sequence
\begin{equation}
   (\{\th_i\}_{i=0}^d; \{\th^*_i\}_{i=0}^d; \{\vphi_i\}_{i=1}^d; \{\phi_i\}_{i=1}^d)        \label{eq:type1parray}
\end{equation}
is a parameter array over $\F$ that has type I and fundamental constant $\beta = q^2 + q^{-2}$,
provided that the inequalities in Lemma \ref{lem:classify}{\rm (i),(ii)} hold.
Conversely, assume that the sequence \eqref{eq:type1parray} is a parameter array over $\F$ that has type I
and fundamental constant $\beta = q^2 + q^{-2}$.
Then there exists a unique sequence \eqref{eq:type1pseq} of scalars in $\F$ that satisfies
\eqref{eq:type1th}--\eqref{eq:type1phi}.
\end{lemma}

\begin{proof}
Assume that the inequalities in Lemma \ref{lem:classify}(i),(ii) hold.
Using \eqref{eq:type1th}--\eqref{eq:type1phi} we routinely verify
the conditions Lemma \ref{lem:classify}(iii)--(v).
Thus the sequence \eqref{eq:type1parray} is a parameter array over $\F$.
Evaluating the expression on the left in \eqref{eq:indep} using \eqref{eq:type1th} we find that
the parameter array \eqref{eq:type1parray} has fundamental constant $\beta = q^2 + q^{-2}$. 
By $q^4 \neq 1$ we have  $\beta \neq \pm 2$.
So the parameter array \eqref{eq:type1parray} has type I.
The last assertion comes from \cite[Theorem 6.1]{NT:balanced}.
\end{proof}

\begin{defi}           \label{def:type1pseq}     \samepage
\ifDRAFT {\rm def:type1pseq}. \fi
Referring to Lemma \ref{lem:type1param}, assume that the sequence \eqref{eq:type1parray}
is a parameter array over $\F$.
We call the scalars  $\delta, \mu, h, \delta^*, \mu^*, h^*, \tau$ 
the {\em basic variables of \eqref{eq:type1parray} with respect to $q$}.
We call the sequence $(\delta, \mu, h, \delta^*, \mu^*, h^*, \tau)$
the {\em basic sequence of \eqref{eq:type1parray} with respect to $q$}.
\end{defi}

\begin{lemma}   \label{lem:type1condpre}    \samepage
\ifDRAFT {\rm lem:type1param}. \fi
Referring to Lemma \ref{lem:type1param},
the following hold for $0 \leq i,j \leq d$:
\begin{align*}
\th_i - \th_j &= (q^{i-j} -q^{j-i}) (\mu q^{i+j-d} - h q^{d-i-j}),
\\
\th^*_i - \th^*_j &=  (q^{i-j} -q^{j-i}) (\mu^* q^{i+j-d} - h^* q^{d-i-j}).
\end{align*}
\end{lemma}

\begin{proof}
Routine verification using \eqref{eq:type1th} and \eqref{eq:type1ths}.
\end{proof}

\begin{lemma}    \label{lem:type1cond}   \samepage
\ifDRAFT {\rm lem::type1cond}. \fi
The inequalities in Lemma \ref{lem:classify}{\rm (i),(ii)} hold
if and only if
\begin{align}
 & q^{2i} \neq 1 &&  (1 \leq i \leq d),                       \label{eq:type1paramcond1}
\\
 & \mu \neq h q^{2i}  &&  (1-d \leq i \leq d-1),          \label{eq:type1paramcond2}
\\
 & \mu^* \neq h^* q^{2i}  && (1-d \leq i \leq d-1),      \label{eq:type1paramcond3}
\\
 & \tau \neq \mu \mu^* q^{2i-d-1} + h h^* q^{d-2i+1}  &&  (1 \leq i \leq d),  \label{eq:type1paramcond4}
\\
 & \tau \neq h \mu^* q^{2i-d-1} + \mu h^* q^{d-2i+1}  &&  (1 \leq i \leq d).   \label{eq:type1paramcond5}
\end{align}
\end{lemma}

\begin{proof}
Routine verification using \eqref{eq:type1vphi}, \eqref{eq:type1phi},
and Lemma \ref{lem:type1condpre}.
\end{proof}

For the rest of this section, let $\Phi$ denote a Leonard system over $\F$
of type I with fundamental constant $\beta = q^2 + q^{-2}$ and parameter array
\begin{equation}
   (\{\th_i\}_{i=0}^d; \{\th^*_i\}_{i=0}^d; \{\vphi_i\}_{i=1}^d; \{\phi_i\}_{i=1}^d).        \label{eq:type1parray2}
\end{equation}

\begin{defi}     \label{def:pseqLS}    \samepage
\ifDRAFT {\rm def:pseqLS}. \fi
By the {\em basic variables} (resp.\ {\em basic sequence}) {\em of $\Phi$ with respect to $q$} 
we mean the basic variables (resp.\ basic sequence)  with respect to $q$ 
for the parameter array \eqref{eq:type1parray2}.
\end{defi}

For the rest of this section, let
$(\delta, \mu, h, \delta^*, \mu^*, h^*, \tau)$
denote the basic sequence of $\Phi$ with respect to $q$.

\begin{lemma}    \label{lem:type1paramDown}    \samepage
\ifDRAFT {\rm lem:type1paramDown}. \fi
In the table below, for each Leonard system in the first column we give
the basic sequence with respect to $q$:
\[
\begin{array}{cc}
\text{\rm Leonard system}  &  \text{\rm basic sequence}
\\ \hline
\Phi^\downarrow & (\delta, \mu, h, \delta^*, h^*, \mu^*, \tau)  \rule{0mm}{3ex}
\\
\Phi^\Downarrow & (\delta, h, \mu, \delta^*, \mu^*, h^*, \tau) \rule{0mm}{2.7ex}
\\
\Phi^\vee & (- \delta, - \mu, -h, \delta^*, \mu^*, h^*, - \tau) \rule{0mm}{2.7ex}
\end{array}
\]
\end{lemma}

\begin{proof}
Concerning $\Phi^\downarrow$ and $\Phi^\Downarrow$, use 
Lemma \ref{lem:parrayrelative}.
Concerning $\Phi^\vee$, use Lemma \ref{lem:parrayPhivee}.
\end{proof}

\begin{lemma}    \label{lem:type1affine}    \samepage
\ifDRAFT {\rm lem:type1affine}. \fi
For scalars $\xi$, $\zeta$, $\xi^*$, $\zeta^*$ in $\F$ with $\xi \xi^* \neq 0$,
consider the Leonard system
\[
   (\xi A + \zeta I; \, \{E_i\}_{i=0}^d; \;  \xi^* A^* + \zeta^* I; \, \, \{E^*_i\}_{i=0}^d).
\]
For this Leonard system the basic sequence with respect to $q$ is equal to
\[
   ( \xi \delta + \zeta; \;  \xi \mu, \; \xi h,  \; 
     \xi^* \delta^* + \zeta^*;  \; \xi^* \mu^*, \; \xi^* h^*, \; \xi \xi^* \tau).
\]
\end{lemma}

\begin{proof}
Use Lemma \ref{lem:affineparam}.
\end{proof}

\begin{corollary}    \label{cor:type1affine}    \samepage
\ifDRAFT {\rm cor:type1affine}. \fi
The Leonard system 
\[
 (A- \delta I; \, \{E_i\}_{i=0}^d; \, A^* - \delta^* I; \,  \{E^*_i\}_{i=0}^d)
\]
has basic sequence
$(0, \mu, h, 0, \mu^*, h^*, \tau)$ with respect to $q$.
\end{corollary}

\begin{defi}    \label{def:type1reduced}    \samepage
\ifDRAFT {\rm def:type1reduced}. \fi
The Leonard system $\Phi$ is said to be {\em reduced} whenever
$\delta=0$ and $\delta^* = 0$.
\end{defi}

\begin{lemma}    \label{lem:type1kappa}    \samepage
\ifDRAFT {\rm lem:type1kappa}. \fi
The invariant value  $\kappa$ for $\Phi$ satisfies
\begin{equation}
   \kappa = \mu h ( q-q^{-1})^2 (q^2- q^{-2})^2.     \label{eq:type1kappa}
\end{equation}
\end{lemma}

\begin{proof}
By Lemma \ref{lem:kappa} and \eqref{eq:type1th}.
\end{proof}

\begin{proofof}{Proposition \ref{prop:RST}, type I}
One routinely verifies \eqref{eq:RST} using  \eqref{eq:Rtype1}--\eqref{eq:Ttype1},
\eqref{eq:type1kappa},
and Lemma \ref{lem:type1param}.
\end{proofof}

Note that Theorem \ref{thm:kappamain} holds for type I.

\section{A characterization of compatibility in terms of the basic sequence, type I}
\label{sec:characterizetype1}
\ifDRAFT {\rm sec:characterizetype1}. \fi

In this section we characterize the compatibility relation for Leonard pairs
of type I in terms of the basic sequence.
Throughout this section, Notation \ref{notation:all} is in effect.
Assume that $\Phi$ has type I and fundamental constant $\beta = q^2 + q^{-2}$.
Note that $\Phi'$ has type I and fundamental constant $\beta $.
Let
\begin{align*}
& (\delta, \mu, h, \delta^*, \mu^*, h^*, \tau),  
&& 
 (\delta', \mu', h', \delta^*, \mu^*, h^*, \tau')
\end{align*}
denote the basic sequence of $\Phi$ and $\Phi'$ with respect to $q$, respectively.

\begin{theorem}    \label{thm:type1main}    \samepage
\ifDRAFT {\rm thm:type1main}. \fi
The matrices $A$, $B$ are compatible if and only if
the following \eqref{eq:type1cond1}--\eqref{eq:type1cond3} hold:
\begin{align}
  \mu h &= \mu' h',                     \label{eq:type1cond1}
\\
  \tau (\mu + h) &= \tau' (\mu' + h'),     \label{eq:type1cond2}
\\
  \tau^2 +  (\mu + h)^2 \mu^* h^*
    &= \tau^{\prime 2} + (\mu' + h')^2  \mu^* h^*.    \label{eq:type1cond3}
\end{align}
\end{theorem}

\begin{proof}
We will invoke Theorem \ref{thm:kappamain}.
To do this we investigate the conditions in \eqref{eq:kappakappad}.
By Lemma \ref{lem:type1kappa} and \eqref{eq:type1paramcond1},
$\kappa = \kappa'$ if and only if \eqref{eq:type1cond1} holds.
Using \eqref{eq:type1vphi}, \eqref{eq:type1phi} we find that under the assumption \eqref{eq:type1cond1}
the expression
$\vphi_1 \phi_1 - \vphi'_1 \phi'_1 - \vphi_d \phi_d + \vphi'_d \phi'_d$ is equal to
\[
   (q-q^{-1})^2 (q^d - q^{-d})^2 (q^{d-1} - q^{1-d}) (\mu^* - h^*)
\]
times
\[
   \tau (\mu + h) - \tau' (\mu' + h').
\]
Using \eqref{eq:type1vphi}, \eqref{eq:type1phi} we find that under the assumptions
\eqref{eq:type1cond1} and \eqref{eq:type1cond2},
the expression $\vphi_1 \phi_1 - \vphi'_1 \phi'_1$ is equal to
\[
  (q-q^{-1})^2 (q^d - q^{-d})^2
\]
times
\[
   \tau^2 +  (\mu + h)^2 \mu^* h^* - \tau^{\prime 2} - (\mu' + h')^2 \mu^* h^* .
\]
By these comments and \eqref{eq:type1paramcond1}, \eqref{eq:type1paramcond3},
we find that under the assumption \eqref{eq:type1cond1},
both $\vphi_1 \phi_1 = \vphi'_1 \phi'_1$ and $\vphi_d \phi_d = \vphi'_d \phi'_d$
hold if and only if both \eqref{eq:type1cond2} and \eqref{eq:type1cond3} hold.
Now the result follows from Theorem \ref{thm:kappamain}.
\end{proof}

Our next goal is to solve the equations \eqref{eq:type1cond1}--\eqref{eq:type1cond3} 
for $\mu'$, $h'$, $\tau'$.
It is convenient to handle separately the cases $\text{\rm Char}(\F) \neq 2$ and $\text{\rm Char}(\F) = 2$.
We first consider the case $\text{\rm Char}(\F) \neq 2$.

\begin{theorem}     \label{thm:type1sol}    \samepage
\ifDRAFT {\rm thm:type1sol}.  \fi
Assume that $\text{\rm Char}(\F) \neq 2$.
Then the equations \eqref{eq:type1cond1}--\eqref{eq:type1cond3} hold
if and only if at least one of the following \eqref{eq:type1sol1}--\eqref{eq:type1sol5} holds:
\begin{align}
& \qquad\qquad  \tau' = \tau,   \qquad\qquad \text{ $(\mu', h')$ is a permutation of $(\mu,h)$};  \label{eq:type1sol1}
\\
& \qquad\qquad \tau' = -\tau,  \qquad\quad \; \text{ $(\mu', h')$ is a permutation of $(-\mu, -h)$};    \label{eq:type1sol2}
\\
& \mu^* h^* \neq 0,  \quad
\mu' h' = \mu h,    \quad
\mu' + h' = \tau (\mu^* h^*)^{-1/2},   \quad \;\,
\tau' = (\mu + h)  (\mu^* h^*)^{1/2};                      \label{eq:type1sol3}
\\
& \mu^* h^* \neq 0,  \quad
\mu' h' = \mu h,    \quad
\mu' + h' = - \tau (\mu^* h^*)^{-1/2},   \quad 
\tau' = - (\mu + h)  (\mu^* h^*)^{1/2};                      \label{eq:type1sol4}
\\
& \qquad\qquad
\mu^* h^* = 0,  \qquad\qquad
\tau' = \tau = 0,  \qquad\qquad
\mu' h' = \mu h.                                                     \label{eq:type1sol5}
\end{align}
\end{theorem}

\begin{proof}
One routinely checks that each of \eqref{eq:type1sol1}--\eqref{eq:type1sol5} gives
a solution to \eqref{eq:type1cond1}--\eqref{eq:type1cond3}.
Now assume that \eqref{eq:type1cond1}--\eqref{eq:type1cond3} hold.
We show that at least one of \eqref{eq:type1sol1}--\eqref{eq:type1sol5} holds.

We claim that if $\tau^2 =\tau^{\prime 2}$ and $\tau \neq 0$ then
either \eqref{eq:type1sol1} or \eqref{eq:type1sol2} holds.
If $\tau' = \tau \neq 0$ then $\mu + h = \mu' + h'$ by \eqref{eq:type1cond2},
so \eqref{eq:type1sol1} holds in view of \eqref{eq:type1cond1}.
If $\tau' = - \tau \neq 0$ then $\mu + h = - \mu' - h'$ by \eqref{eq:type1cond2},
so \eqref{eq:type1sol2} holds in view of \eqref{eq:type1cond1}.
We have shown the claim.

For the moment, assume that $\mu^* h^* = 0$.
By \eqref{eq:type1cond3}, $\tau^{\prime 2} = \tau^2$.
If $\tau=0$ then \eqref{eq:type1sol5} holds.
If $\tau \neq 0$ then
either \eqref{eq:type1sol1} or \eqref{eq:type1sol2} holds by the claim.
For the rest of this proof, assume that $\mu^* h^* \neq 0$.
In \eqref{eq:type1cond3}, multiply each side by $\tau^{\prime 2}$,
and simplify the result using \eqref{eq:type1cond2} to get
\[
(\tau^2 - \tau^{\prime 2}) 
 \big( \tau^{\prime 2} - (\mu + h)^2 \mu^* h^*  \big) = 0.
\]
Thus at least one of the following \eqref{eq:type1condaux1}, \eqref{eq:type1condaux2} holds:
\begin{align}
\tau^{\prime 2} &= \tau^2,                \label{eq:type1condaux1}
\\
 \tau^{\prime 2} &=  (\mu + h)^2 \mu^* h^*.          \label{eq:type1condaux2}
\end{align}

First consider the case \eqref{eq:type1condaux1}.
We may assume that $\tau=0$; otherwise either \eqref{eq:type1sol1} or \eqref{eq:type1sol2} holds
by the claim.
By \eqref{eq:type1cond3} with $\tau = \tau'=0$ we get
$(\mu+h)^2 = (\mu' + h')^2$.
Thus either $\mu' + h' = \mu + h$ or $\mu' + h' = - \mu - h$.
If $\mu' + h' = \mu + h$ then \eqref{eq:type1sol1} holds in view of \eqref{eq:type1cond1}.
If $\mu' + h' = - \mu - h$ then \eqref{eq:type1sol2} holds in view of \eqref{eq:type1cond1}.

Next consider the case \eqref{eq:type1condaux2}.
By \eqref{eq:type1cond3} and \eqref{eq:type1condaux2},
\begin{equation}
   \tau^2 =  (\mu' + h')^2 \mu^* h^*.                                 \label{eq:type1condaux3}
\end{equation}
For the moment, assume that $\tau=0$.
Then by \eqref{eq:type1condaux3}, $\mu' + h' = 0$, and so one of
\eqref{eq:type1sol3}, \eqref{eq:type1sol4} holds in view of \eqref{eq:type1condaux2}.
For the rest of this proof, assume that $\tau \neq 0$.
By \eqref{eq:type1condaux3}, one of the following \eqref{eq:type1condaux4},
\eqref{eq:type1condaux5} holds:
\begin{align}
  \tau &=(\mu' + h')  (\mu^* h^*)^{1/2} ,                \label{eq:type1condaux4}
\\
   \tau &=  - (\mu' + h')  (\mu^* h^*)^{1/2}.             \label{eq:type1condaux5}
\end{align}
If \eqref{eq:type1condaux4} holds,
then by \eqref{eq:type1cond2},
\[
  (\mu' + h') (\mu + h)  (\mu^* h^*)^{1/2} = \tau' (\mu' + h').
\]
In this equation, we have $\mu' + h' \neq 0$ by \eqref{eq:type1condaux3} and $\tau \neq 0$.
So
\[
  (\mu + h) (\mu^* h^*)^{1/2} = \tau'.
\]
Thus \eqref{eq:type1sol3} holds.
If \eqref{eq:type1condaux5} holds,
then by \eqref{eq:type1cond2},
\[
   - (\mu' + h') (\mu + h) (\mu^* h^*)^{1/2} = \tau' (\mu' + h').
\]
In this equation, we have $\mu' + h' \neq 0$ by \eqref{eq:type1condaux3} and $\tau \neq 0$.
So
\[
  -  (\mu + h)  (\mu^* h^*)^{1/2} = \tau'.
\]
Thus \eqref{eq:type1sol4} holds.
\end{proof}

We have some comments about \eqref{eq:type1sol1}, \eqref{eq:type1sol2}.
By Proposition \ref{prop:A=B}, \eqref{eq:type1sol1} holds
if and only if there exists $\zeta \in \F$ such that $B = A + \zeta I$.
By Proposition \ref{prop:ABvee}, \eqref{eq:type1sol2} holds
if and only if there exists $\zeta \in \F$ such that $B = A^\vee + \zeta I$.
The solutions \eqref{eq:type1sol1}--\eqref{eq:type1sol5} are not mutually exclusive.

Next we consider the case $\text{\rm Char}(\F)=2$.

\begin{theorem}     \label{thm:type1sol2}    \samepage
\ifDRAFT {\rm thm:type1sol2}.  \fi
Assume that $\text{\rm Char}(\F) = 2$.
Then the equations \eqref{eq:type1cond1}--\eqref{eq:type1cond3} hold
if and only if at least one of the following \eqref{eq:type1sol1b}--\eqref{eq:type1sol5b} holds:
\begin{align}
& \qquad\qquad  \tau' = \tau,   \qquad\qquad \text{ $(\mu', h')$ is a permutation of $(\mu,h)$};  \label{eq:type1sol1b}
\\
& \mu^* h^* \neq 0,  \quad
\mu' h' = \mu h,    \quad
\mu' + h' = \tau (\mu^* h^*)^{-1/2},   \quad \;\,
\tau' = (\mu + h)  (\mu^* h^*)^{1/2};                      \label{eq:type1sol3b}
\\
& \qquad\qquad
\mu^* h^* = 0,  \qquad\qquad
\tau' = \tau = 0,  \qquad\qquad
\mu' h' = \mu h.                                                     \label{eq:type1sol5b}
\end{align}
\end{theorem}

\begin{proof}
One routinely checks that each of \eqref{eq:type1sol1b}--\eqref{eq:type1sol5b} gives a
a solution to \eqref{eq:type1cond1}--\eqref{eq:type1cond3}.
Now assume that \eqref{eq:type1cond1}--\eqref{eq:type1cond3} holds.
We show that at least one of \eqref{eq:type1sol1b}--\eqref{eq:type1sol5b} holds.

For the moment,  assume that $\mu^* h^* = 0$.
By \eqref{eq:type1cond3} we have $\tau^{\prime 2} = \tau^2$, so $\tau' = \tau$
since $\text{\rm Char}(\F) = 2$.
We may assume that $\tau \neq 0$; otherwise \eqref{eq:type1sol5b} holds.
By \eqref{eq:type1cond2}, $\mu + h = \mu' + h'$.
By this and \eqref{eq:type1cond1} we get \eqref{eq:type1sol1b}.
For the rest of this proof, assume that $\mu^* h^* \neq 0$.
By \eqref{eq:type1cond3},
\[
  \tau^2 - \tau^{\prime 2}
  =\left( (\mu'+h')^2 - (\mu + h)^2 \right)  \mu^* h^* .
\]
By this and since $\text{\rm Char}(\F)=2$,
\[
 (\tau - \tau')^2 = (\mu' + h' - \mu - h)^2  \mu^* h^* ,
\]
and so
\[
  \tau - \tau' = (\mu' + h' - \mu - h)  (\mu^* h^*)^{1/2}.
\]
In this equation, multiply each side by $\tau'$, and simplify the result using \eqref{eq:type1cond2}
to get
\[
   (\tau - \tau') \left( \tau' - (\mu + h) (\mu^* h^*)^{1/2}  \right) = 0.
\]
Thus at least one of the following \eqref{eq:type1solaux1b}, \eqref{eq:type1solaux2b} holds:
\begin{align}
  \tau' &= \tau,                    \label{eq:type1solaux1b}
\\
  \tau' &= (\mu + h)  (\mu^* h^*)^{1/2}.           \label{eq:type1solaux2b}
\end{align}
First assume that \eqref{eq:type1solaux1b} holds.
If $\tau' = \tau \neq 0$, then by \eqref{eq:type1cond1} and \eqref{eq:type1cond2} 
we get \eqref{eq:type1sol1b}.
If $\tau' = \tau = 0$, then by \eqref{eq:type1cond3} we get
$(\mu+h)^2 = (\mu' + h')^2$.
By this and since $\text{\rm Char}(\F)=2$ we get $\mu + h = \mu' + h'$.
By this and \eqref{eq:type1cond1} we get \eqref{eq:type1sol1b}.
Next assume that \eqref{eq:type1solaux2b} holds.
In \eqref{eq:type1cond3}, eliminate $\tau'$ using  \eqref{eq:type1solaux2b} to get
\[
    \tau^2 = (\mu' + h')^2 \mu^* h^*.
\]
By this and since $\text{\rm Char}(\F)=2$,
\[
    \tau = (\mu' + h') (\mu^* h^*)^{1/2}.
\]
By these comments and \eqref{eq:type1cond1} we get \eqref{eq:type1sol3b}.
\end{proof}

We have a comment about \eqref{eq:type1sol1b}.
By Proposition \ref{prop:A=B}, \eqref{eq:type1sol1b} holds
if and only if there exists $\zeta \in \F$ such that $B = A + \zeta I$.
The solutions \eqref{eq:type1sol1b}--\eqref{eq:type1sol5b} are not mutually exclusive.

\section{Describing the companions for a Leonard pair of type I}
\label{sec:companiontype1}
\ifDRAFT {\rm sec:companiontype1}. \fi

In this section we describe the companions for a Leonard pair of type I.
Throughout this section, Notation \ref{notation:all} is in effect.
Assume that $\Phi$ has type I and fundamental constant $\beta = q^2 + q^{-2}$.
Note that $\Phi'$ has type I and fundamental constant $\beta$.
Let
\begin{align}
& (\delta, \mu, h, \delta^*, \mu^*, h^*, \tau),  
&& 
 (\delta', \mu', h', \delta^*, \mu^*, h^*, \tau')                       \label{eq:type1pseqs}
\end{align}
denote the basic sequence of $\Phi$ and $\Phi'$ with respect to $q$, respectively.
Assume that $A$, $B$ are compatible, and consider the companion $K=A-B$.
We will give the entries of $K$.
To avoid complicated formulas, 
we assume that each of $\Phi$, $\Phi'$ is reduced,
so that $\delta=0$, $\delta'=0$, $\delta^* = 0$.

We first assume $\text{\rm Char}(\F) \neq 2$.
Under this assumption we consider the cases \eqref{eq:type1sol1}--\eqref{eq:type1sol5}
in Theorem \ref{thm:type1sol}.
For the moment assume that  \eqref{eq:type1sol1} holds.
By the comments below Theorem \ref{thm:type1sol},
there exists $\zeta \in \F$ such that $B = A + \zeta I$.
By this and $\delta=\delta'=0$ we get $B=A$.
So $K=0$.
Next assume that \eqref{eq:type1sol2} holds.
By the comments below Theorem \ref{thm:type1sol},
there exists $\zeta \in \F$ such that $B = A^\vee + \zeta I$.
By this and $\delta=\delta'=0$ we get $B=A^\vee$.
By this and Lemma \ref{lem:bondcompOmega}, 
$K_{i,i} = 2 a_i$ for $0 \leq i \leq d$.
Next we give the $K$ that corresponds to solutions \eqref{eq:type1sol3}, \eqref{eq:type1sol4}.
In this case $\mu^* h^* \neq 0$; we may assume $\mu^* h^* = 1$ in view of 
Lemma \ref{lem:type1affine}.
To avoid trivialities we interpret $1^{1/2}=1$.

\begin{theorem}      \label{thm:ex2}     \samepage
\ifDRAFT {\rm thm:ex2}. \fi
Assume that $\text{\rm Char}(\F) \neq 2$.
Then the following hold.
\begin{itemize}
\item[\rm (i)]
Assume that \eqref{eq:type1sol3} holds with $\mu^* h^* = 1$.
Then
\begin{align*}
K_{0,0} &= \frac{q^d (\mu^* - q^{-d-1})(\mu + h-\tau) } {\mu^* - q^{d-1} },
\\
K_{i,i} &= 
 \frac{q^{d-2i} (\mu^* - q^{-d-1})(\mu^* - q^{d+1})(\mu + h - \tau) }
      { (\mu^* - q^{d-2i-1})(\mu^* - q^{d-2i+1}) }    &&  (1 \leq i \leq d-1),
\\
K_{d,d} &= \frac{q^{-d} (\mu^* - q^{d+1})(\mu + h - \tau) } {\mu^* - q^{1-d} }.
\end{align*}
\item[\rm (ii)]
Assume that \eqref{eq:type1sol4} holds with $\mu^* h^* = 1$.
Then
\begin{align*}
K_{0,0} &= \frac{q^d (\mu^* + q^{-d-1})(\mu + h + \tau) } {\mu^* + q^{d-1} },
\\
K_{i,i} &= 
 \frac{q^{d-2i} (\mu^* + q^{-d-1})(\mu^* + q^{d+1})(\mu + h + \tau) }
      { (\mu^* + q^{d-2i-1})(\mu^* + q^{d-2i+1}) }    &&  (1 \leq i \leq d-1),
\\
K_{d,d} &= \frac{q^{-d} (\mu^* + q^{d+1})(\mu + h + \tau) } {\mu^* + q^{1-d} }.
\end{align*}
\end{itemize}
\end{theorem}

\begin{proof}
Use \eqref{eq:Kaa} and Lemmas \ref{lem:aiparam}, \ref{lem:type1param}
with $\delta = \delta' = 0$.
\end{proof}

Next we give the $K$ that corresponds to solution \eqref{eq:type1sol5}.

\begin{theorem}      \label{thm:ex3}     \samepage
\ifDRAFT {\rm thm:ex3}. \fi
Assume that $\text{\rm Char}(\F) \neq 2$. Then
\begin{itemize}
\item[\rm (i)]
Assume that \eqref{eq:type1sol5} holds with $\mu^*=0$.
Then
\begin{align*}
 K_{i,i} &= q^{2i-d} (\mu + h - \mu' - h')    &&   (0 \leq i \leq d).
\end{align*}
\item[\rm (i)]
Assume that \eqref{eq:type1sol5} holds with $h^*=0$.
Then
\begin{align*}
 K_{i,i} &= q^{d-2i} (\mu + h - \mu' - h')    &&   (0 \leq i \leq d).
\end{align*}
\end{itemize}
\end{theorem}

\begin{proof}
Use \eqref{eq:Kaa} and Lemmas \ref{lem:aiparam}, \ref{lem:type1param}
with $\delta = \delta' = 0$.
\end{proof}

Next we assume $\text{\rm Char}(\F) = 2$.
Under this assumption we consider the cases \eqref{eq:type1sol1b}--\eqref{eq:type1sol5b}
in Theorem \ref{thm:type1sol2}.
For the moment assume that  \eqref{eq:type1sol1b} holds.
By the comments below Theorem \ref{thm:type1sol2},
there exists $\zeta \in \F$ such that $B = A + \zeta I$.
By this and $\delta = \delta' = 0$ we get $B=A$. So $K=0$.
We now give the $K$ that corresponds to solution \eqref{eq:type1sol3b}.
In view of Lemma \ref{lem:type1affine},
we may assume that $\mu^* h^* = 1$.

\begin{theorem}      \label{thm:ex2b}     \samepage
\ifDRAFT {\rm thm:ex2b}. \fi
Assume that $\text{\rm Char}(\F) = 2$
and \eqref{eq:type1sol3b} holds with $\mu^* h^* = 1$.
Then
\begin{align*}
K_{0,0} &= \frac{q^d (\mu^* + q^{-d-1})(\mu + h + \tau) } {\mu^* + q^{d-1} },
\\
K_{i,i} &= 
 \frac{q^{d-2i} (\mu^* + q^{-d-1})(\mu^* + q^{d+1})(\mu + h + \tau) }
      { (\mu^* + q^{d-2i-1})(\mu^* + q^{d-2i+1}) }    &&  (1 \leq i \leq d-1),
\\
K_{d,d} &= \frac{q^{-d} (\mu^* + q^{d+1})(\mu + h + \tau) } {\mu^* + q^{1-d} }. 
\end{align*}
\end{theorem}

\begin{proof}
Use \eqref{eq:Kaa} and Lemmas \ref{lem:aiparam}, \ref{lem:type1param}
with $\delta = \delta' = 0$.
\end{proof}

Next we give the $K$ that corresponds to solution \eqref{eq:type1sol5b}.

\begin{theorem}      \label{thm:ex3b}     \samepage
\ifDRAFT {\rm thm:ex3b}. \fi
Assume that $\text{\rm Char}(\F) = 2$. Then
\begin{itemize}
\item[\rm (i)]
Assume that \eqref{eq:type1sol5b} holds with $\mu^*=0$.
Then
\begin{align*}
 K_{i,i} &= q^{2i-d} (\mu + h - \mu' - h')    &&   (0 \leq i \leq d).
\end{align*}
\item[\rm (i)]
Assume that \eqref{eq:type1sol5b} holds with $h^*=0$.
Then
\begin{align*}
 K_{i,i} &= q^{d-2i} (\mu + h - \mu' - h')    &&   (0 \leq i \leq d).
\end{align*}
\end{itemize}
\end{theorem}

\begin{proof}
Use \eqref{eq:Kaa} and Lemmas \ref{lem:aiparam}, \ref{lem:type1param}
with $\delta = \delta' = 0$.
\end{proof}

\section{The parameter arrays of type II}
\label{sec:parraytype2}
\ifDRAFT {\rm sec:parraytype2}. \fi

In this section we describe the parameter arrays of type II.
We then prove Proposition \ref{prop:RST} for type II.
Throughout this section, assume that $d \geq 3$.

\begin{lemma}    \label{lem:type2param}    \samepage
\ifDRAFT {\rm lem:type2param}. \fi
Assume that $\text{\rm Char}(\F) \neq 2$.
For a sequence
\begin{equation}
   (\delta, \mu, h, \delta^*, \mu^*, h^*, \tau)                \label{eq:type2pseq}
\end{equation}
of scalars in $\F$,  define
\begin{align}
\th_i &= \delta + \mu (i-d/2)+ h i (d-i),                \label{eq:type2th}
\\
\th^*_i &= \delta^* + \mu^* (i-d/2) + h^* i (d-i)     \label{eq:type2ths}
\end{align}
for $0 \leq i \leq d$ and
\begin{align}
\vphi_i &= i (d-i+1) \big(\tau - \mu \mu^*/2 + 
  (h \mu^* + \mu h^*)(i-(d+1)/2) + h h^* (i-1)(d-i) \big),             \label{eq:type2vphi}
\\
\phi_i &= i (d-i+1) \big(\tau + \mu \mu^*/2+
         (h \mu^* - \mu h^*)(i-(d+1)/2)+h h^*(i-1)(d-i) \big)         \label{eq:type2phi}
\end{align}
for $1 \leq i \leq d$.
Then the sequence
\begin{equation}
   (\{\th_i\}_{i=0}^d; \{\th^*_i\}_{i=0}^d; \{\vphi_i\}_{i=1}^d; \{\phi_i\}_{i=1}^d)        \label{eq:type2parray}
\end{equation}
is a parameter array over $\F$ that has type II,
provided that the inequalities in Lemma \ref{lem:classify}{\rm (i),(ii)} hold.
Conversely, assume that the sequence \eqref{eq:type2parray} is a parameter 
array over $\F$ that has type II.
Then there exists a unique sequence \eqref{eq:type2pseq} of scalars in $\F$ that satisfies
\eqref{eq:type2th}--\eqref{eq:type2phi}.
\end{lemma}

\begin{proof}
Assume that the inequalities in Lemma \ref{lem:classify}(i),(ii) hold.
Using \eqref{eq:type2th}--\eqref{eq:type2phi} we routinely verify
the conditions Lemma \ref{lem:classify}(iii)--(v).
Thus the sequence \eqref{eq:type2parray} is a parameter array over $\F$.
Evaluating the expression on the left in \eqref{eq:indep} using \eqref{eq:type2th} we find that
the parameter array \eqref{eq:type2parray} has fundamental constant $\beta = 2$. 
So the parameter array \eqref{eq:type2parray} has type II.
The last assertion comes from \cite[Theorem 7.1]{NT:balanced}.
\end{proof}

\begin{defi}    \label{def:type2pseq}    \samepage
\ifDRAFT {\rm def:type2pseq}. \fi
Referring to Lemma \ref{lem:type2param},
assume that the sequence \eqref{eq:type2parray} is a parameter array over $\F$.
We call the scalars $\delta, \mu, h, \delta^*, \mu^*, h^*, \tau$ 
the {\em basic variables of \eqref{eq:type2parray}}.
We call the sequence $(\delta, \mu, h, \delta^*, \mu^*, h^*, \tau)$ 
the {\em basic sequence of \eqref{eq:type2parray}}.
\end{defi}

\begin{lemma}    \label{lem:type2condpre}    \samepage
\ifDRAFT {\rm lem:type2condpre}. \fi
Referring to Lemma \ref{lem:type2param},
the following hold for $0 \leq i,j \leq d$:
\begin{align*}
\th_i - \th_j &= (i-j) \big( \mu + h (d-i-j) \big),
\\
\th^*_i - \th^*_j &= (i-j) \big( \mu^* + h^* (d-i-j) \big).
\end{align*}
\end{lemma}

\begin{proof}
Routine verification using \eqref{eq:type2th} and \eqref{eq:type2ths}.
\end{proof}

\begin{lemma}    \label{lem:type2cond}   \samepage
\ifDRAFT {\rm lem::type2cond}. \fi
Referring to Lemma \ref{lem:type2param}, 
the inequalities in Lemma \ref{lem:classify}{\rm (i),(ii)} hold
if and only if
\begin{align}
 & \text{$\text{\rm Char}(\F)$ is equal to $0$ or greater than $d$},          \label{eq:type2paramcond1}
\\
 & \mu \neq h i  \qquad\qquad \;\;\,  (1-d \leq i \leq d-1),          \label{eq:type2paramcond2}
\\
 & \mu^* \neq h^* i \qquad\qquad (1-d \leq i \leq d-1),     \label{eq:type2paramcond3}
\\
 & \tau \neq \mu \mu^*/2 - 
       (h \mu^* + \mu h^*)(i-(d+1)/2) -  h h^* (i-1)(d-i)   & (1 \leq i \leq d),  \label{eq:type2paramcond4}
\\
& \tau \neq  - \mu \mu^*/2 - 
       (h \mu^* - \mu h^*)(i-(d+1)/2) -  h h^* (i-1)(d-i)  & (1 \leq i \leq d). \label{eq:type2paramcond5}
\end{align}
\end{lemma}

\begin{proof}
Routine verification using \eqref{eq:type2vphi}, \eqref{eq:type2phi},
and Lemma \ref{lem:type2condpre}.
\end{proof}

For the rest of this section, let $\Phi$ denote a Leonard system over $\F$
that has type II and parameter array
\begin{equation}
   (\{\th_i\}_{i=0}^d; \{\th^*_i\}_{i=0}^d; \{\vphi_i\}_{i=1}^d; \{\phi_i\}_{i=1}^d).        \label{eq:type2parray2}
\end{equation}

\begin{defi}     \label{def:type2pseqLS}    \samepage
\ifDRAFT {\rm def:type2pseqLS}. \fi
By the {\em basic variables} (resp.\ {\em basic sequence}) {\em of $\Phi$} we mean the 
basic variables (resp.\ basic sequence) of the parameter array \eqref{eq:type2parray2}.
\end{defi}

For the rest of this section, let
$(\delta, \mu, h, \delta^*, \mu^*, h^*, \tau)$
denote the basic sequence of $\Phi$.

\begin{lemma}    \label{lem:type2paramDown}    \samepage
\ifDRAFT {\rm lem:type2paramDown}. \fi
In the table below, for each Leonard system in the first column we give
the basic sequence:
\[
\begin{array}{cc}
\text{\rm Leonard system}  &  \text{\rm basic sequence}
\\ \hline
\Phi^\downarrow & (\delta, \mu, h, \delta^*, - \mu^*, h^*, \tau)  \rule{0mm}{3ex}
\\
\Phi^\Downarrow & (\delta, - \mu, h, \delta^*, \mu^*, h^*, \tau)   \rule{0mm}{2.7ex}
\\
\Phi^\vee & (- \delta, - \mu, -h, \delta^*, \mu^*, h^*, - \tau)  \rule{0mm}{2.7ex}
\end{array}
\]
\end{lemma}

\begin{proof}
Concerning $\Phi^\downarrow$ and $\Phi^\Downarrow$, use 
Lemma \ref{lem:parrayrelative}.
Concerning $\Phi^\vee$, use Lemma \ref{lem:parrayPhivee}.
\end{proof}

\begin{lemma}    \label{lem:type2affine}    \samepage
\ifDRAFT {\rm lem:type2affine}. \fi
For scalars $\xi$, $\zeta$, $\xi^*$, $\zeta^*$ in $\F$ with $\xi \xi^* \neq 0$,
consider the Leonard system
\[
   (\xi A + \zeta I; \, \{E_i\}_{i=0}^d; \;  \xi^* A^* + \zeta^* I; \, \, \{E^*_i\}_{i=0}^d).
\]
For this Leonard system the basic sequence is equal to
\[
   ( \xi \delta + \zeta; \;  \xi \mu, \; \xi h,  \; 
     \xi^* \delta^* + \zeta^*;  \; \xi^* \mu^*, \; \xi^* h^*, \; \xi \xi^* \tau).
\]
\end{lemma}

\begin{proof}
Use Lemma \ref{lem:affineparam}.
\end{proof}

\begin{corollary}    \label{cor:type2affine}    \samepage
\ifDRAFT {\rm cor:type2affine}. \fi
The Leonard system 
\[
 (A- \delta I; \, \{E_i\}_{i=0}^d; \, A^* - \delta^* I; \,  \{E^*_i\}_{i=0}^d)
\]
has basic sequence
$(0, \mu, h, 0, \mu^*, h^*, \tau)$.
\end{corollary}

\begin{defi}    \label{def:type2reduced}    \samepage
\ifDRAFT {\rm def:type2reduced}. \fi
We say that $\Phi$ is {\em reduced} whenever
$\delta=0$ and $\delta^* = 0$.
\end{defi}

\begin{lemma}    \label{lem:type2kappa}    \samepage
\ifDRAFT {\rm lem:type2kappa}. \fi
The invariant value $\kappa$ for $\Phi$ satisfies
$\kappa = 4 h^2$.
\end{lemma}

\begin{proof}
By Lemma \ref{lem:kappa} and \eqref{eq:type2th}.
\end{proof}

\begin{proofof}{Proposition \ref{prop:RST}, type II}
One routinely verifies \eqref{eq:RST} using \eqref{eq:Rtype1}--\eqref{eq:Ttype1}
and Lemmas \ref{lem:type2param}, \ref{lem:type2kappa}.
\end{proofof}

Note that Theorem \ref{thm:kappamain} holds for type II.

\section{A characterization of compatibility in terms of the basic sequence, type II}
\label{sec:characterizetype2}
\ifDRAFT {\rm sec:characterizetype2}. \fi

In this section we characterize the compatibility relation for Leonard pairs
of type II in terms of the basic sequence.
Throughout this section, Notation \ref{notation:all} is in effect.
Assume that $\Phi$ has type II.
Note that $\Phi'$ has type II.
Let
\begin{align*}
& (\delta, \mu, h, \delta^*, \mu^*, h^*, \tau),  
&& 
 (\delta', \mu', h', \delta^*, \mu^*, h^*, \tau')
\end{align*}
denote the basic sequence of $\Phi$ and $\Phi'$, respectively.

\begin{theorem}    \label{thm:type2main}    \samepage
\ifDRAFT {\rm thm:type2main}. \fi
The matrices $A$, $B$ are compatible if and only if
the following \eqref{eq:type2cond1}--\eqref{eq:type2cond3} hold:
\begin{align}
h^2 &= h^{\prime 2},                                        \label{eq:type2cond1} 
\\                     
2 h \tau + \mu^2 h^* &= 2 h' \tau' + \mu^{\prime 2} h^*,     \label{eq:type2cond2}
\\
4 \tau^2 - \mu^2 \big( \mu^{* 2} + (d-1)^2 h^{* 2} \big)
  &= 4 \tau^{\prime 2} - \mu^{\prime 2}   \big( \mu^{* 2} + (d-1)^2 h^{* 2} \big).  \label{eq:type2cond3}
\end{align}
\end{theorem}

\begin{proof}
We will invoke Theorem \ref{thm:kappamain}.
To do this we investigate the conditions in \eqref{eq:kappakappad}.
By Lemma \ref{lem:type2kappa} and \eqref{eq:type2paramcond1},
$\kappa = \kappa'$ if and only if \eqref{eq:type2cond1} holds.
Using \eqref{eq:type2vphi}, \eqref{eq:type2phi} we find that under 
the assumption \eqref{eq:type2cond1}
the expression
$\vphi_1 \phi_1 - \vphi'_1 \phi'_1 - \vphi_d \phi_d + \vphi'_d \phi'_d$ is equal to
$d^2 (1-d) \mu^*$
times
\[
   2 h \tau + \mu^2 h^* - 2 h' \tau' - \mu^{\prime 2} h^*.
\]
Using \eqref{eq:type2vphi}, \eqref{eq:type2phi} we find that under the assumptions
\eqref{eq:type2cond1} and \eqref{eq:type2cond2}
the expression $\vphi_1 \phi_1 - \vphi'_1 \phi'_1$ is equal to
$d^2/4$
times
\[
  4 \tau^2 -\mu^2 \left( \mu^{* 2} + (d-1)^2 h^{* 2} \right)
 - 4 \tau^{\prime 2} + \mu^{\prime 2}  \left( \mu^{* 2} + (d-1)^2 h^{* 2} \right).
\]
By these comments and \eqref{eq:type2paramcond1}, \eqref{eq:type2paramcond3},
we find that under the assumption \eqref{eq:type2cond1}
both $\vphi_1 \phi_1 = \vphi'_1 \phi'_1$ and $\vphi_d \phi_d = \vphi'_d \phi'_d$
hold if and only if both \eqref{eq:type2cond2} and \eqref{eq:type2cond3} hold.
Now the result follows from Theorem \ref{thm:kappamain}.
\end{proof}

Our next goal is to solve the equations \eqref{eq:type2cond1}--\eqref{eq:type2cond3} 
for $\mu'$, $h'$, $\tau'$.

\begin{theorem}   \label{thm:type2sol}      \samepage
\ifDRAFT {\rm thm:type2sol}. \fi
The equations \eqref{eq:type2cond1}--\eqref{eq:type2cond3} hold if and only if at least one of the
following \eqref{eq:type2sol1}--\eqref{eq:type2sol5} holds:
\begin{align}
 & \; h' = h, \qquad\qquad \tau' = \tau, \qquad\qquad 
  \mu^{\prime 2} = \mu^2;                   \label{eq:type2sol1}
\\
&   \rule{0mm}{3ex}
\; h' = - h, \qquad\quad\; \tau' = - \tau, \qquad\quad \, 
    \mu^{\prime 2} = \mu^2;                      \label{eq:type2sol2}
\\
&   \rule{0mm}{7ex}
\begin{array}{lc}
\displaystyle
h^* \neq 0,   \qquad
h' = h \neq 0, \qquad
\tau' = - \tau - \frac{h} { 2 h^* } \left( \mu^{* 2} + (d-1)^2  h^{* 2} \right),
\\    \displaystyle
 \mu^{\prime 2} = \mu^2 + 
  \frac{h}{h^{* 2} }
   \left( 4 h^* \tau  + h \big( \mu^{* 2} + (d-1)^2 h^{* 2} \big)  \right);
\end{array}                                               \label{eq:type2sol3}
\\ \rule{0mm}{7ex}
&
\begin{array}{lc}
\displaystyle
h^* \neq 0,  \qquad 
h' = - h \neq 0, \qquad
 \tau' =  \tau + \frac{h} { 2 h^* } \left( \mu^{* 2} + (d-1)^2  h^{* 2} \right),
\\  \displaystyle
 \mu^{\prime 2} = \mu^2 + 
  \frac{h}{h^{* 2} }
   \left( 4  h^* \tau  + h (\mu^{* 2} + (d-1)^2 h^{* 2})  \right);
\end{array}                                                      \label{eq:type2sol4}
\\
&   \rule{0mm}{4ex}
\; h^* = 0, \qquad h' = h = 0,   \qquad
   4 (\tau^2 - \tau^{\prime 2}) = (\mu^2 - \mu^{\prime 2}) \mu^{* 2}.    \label{eq:type2sol5}
\end{align}
\end{theorem}

\begin{proof}
One routinely checks that each of \eqref{eq:type2sol1}--\eqref{eq:type2sol5} gives
a solution to \eqref{eq:type2cond1}--\eqref{eq:type2cond3}.
Now assume that \eqref{eq:type2cond1}--\eqref{eq:type2cond3} hold.
We show that at least one of \eqref{eq:type2sol1}--\eqref{eq:type2sol5} holds.

First consider the case $h^*=0$.
Note by \eqref{eq:type2paramcond3} that $\mu^* \neq 0$.
By \eqref{eq:type2cond2} and \eqref{eq:type2cond3},
\begin{align}
 h \tau &= h' \tau',                                                                      \label{eq:type2cond2b}
\\
 4 (\tau^2 - \tau^{\prime 2}) &=  (\mu^2 - \mu^{\prime2 }) \mu^{* 2}.     \label{eq:type2cond3b}
\end{align}
We may assume that $h \neq 0$; otherwise $h=h'=0$ by \eqref{eq:type2cond1},
and so \eqref{eq:type2sol5} holds by \eqref{eq:type2cond3b}.
By \eqref{eq:type2cond1} we have either $h' = h$ or $h' = -h$.
If $h' = h$ then by \eqref{eq:type2cond2b} we get $\tau' = \tau$,
and so \eqref{eq:type2sol1} holds by \eqref{eq:type2cond3b}.
If $h' = -h$ then by \eqref{eq:type2cond2b} we get $\tau' = - \tau$,
and so \eqref{eq:type2sol2} holds by \eqref{eq:type2cond3b}.

Next consider the case $h^* \neq 0$.
First assume that $h=0$.
By \eqref{eq:type2cond1}, $h' = 0$.
By \eqref{eq:type2cond2}, $\mu^2 = \mu^{\prime 2}$.
By this and \eqref{eq:type2cond3}, $\tau^2 = \tau^{\prime 2}$.
So either $\tau' = \tau$ or $\tau' = - \tau$.
If $\tau' = \tau$ then \eqref{eq:type2sol1} holds.
If $\tau' = - \tau$ then \eqref{eq:type2sol2} holds.
Next assume that $h \neq 0$.
By \eqref{eq:type2cond1} we have either $h' = h$ or $h' = -h$.
First assume that $h' = h$.
By \eqref{eq:type2cond2},
\begin{equation}
 \mu^{\prime 2} = \mu^2 + \frac{ 2 h (\tau - \tau') } { h^* }.    \label{eq:type2solaux1}
\end{equation}
In \eqref{eq:type2cond3}, eliminate $\mu'$ using \eqref{eq:type2solaux1} to get
\begin{equation}
 (\tau - \tau') 
 \left( 2 ( \tau+ \tau') + \frac{h}{h^*} \big( \mu^{*2}+(d-1)^2 h^{* 2} \big) \right) = 0.  \label{eq:type2solaux3}
\end{equation}
We may assume that $\tau \neq \tau'$;
otherwise $\mu^{\prime 2} = \mu^2$ by \eqref{eq:type2solaux1},
and so \eqref{eq:type2sol1} holds.
By \eqref{eq:type2solaux3},
\[
   \tau' = - \tau - \frac{h}{2 h^*} \big( \mu^{*2} + (d-1)^2 h^{* 2} \big).
\]
By this and \eqref{eq:type2solaux1},
\[
  \mu^{\prime 2} = \mu^2 +
   \frac{h} {h^{*2} } \left( 4 \tau h^* + h \big( \mu^{*2}+(d-1)^2 h^{* 2} \big) \right).
\]
Thus \eqref{eq:type2sol3} holds.
Next assume that $h' = - h$.
By \eqref{eq:type2cond2},
\begin{equation}
 \mu^{\prime 2} = \mu^2 + \frac{ 2 h (\tau + \tau') } { h^* }.    \label{eq:type2solaux1b}
\end{equation}
In \eqref{eq:type2cond3}, eliminate $\mu'$ using \eqref{eq:type2solaux1b} to get
\begin{equation}
 (\tau + \tau') 
 \left( 2 ( \tau - \tau') + \frac{h}{h^*} \big( \mu^{*2}+(d-1)^2 h^{* 2} \big) \right) = 0.  \label{eq:type2solaux3b}
\end{equation}
We may assume that $\tau + \tau' \neq 0$;
otherwise $\mu^{\prime 2} = \mu^2$ by \eqref{eq:type2solaux1b},
and so \eqref{eq:type2sol2} holds.
By \eqref{eq:type2solaux3b},
\[
   \tau' = \tau + \frac{h}{2 h^*} \big( \mu^{*2} + (d-1)^2 h^{* 2} \big).
\]
By this and \eqref{eq:type2solaux1b},
\[
  \mu^{\prime 2} = \mu^2 +
   \frac{h} {h^{*2} } \left( 4 \tau h^* + h \big( \mu^{*2}+(d-1)^2 h^{* 2} \big) \right).
\]
Thus \eqref{eq:type2sol4} holds.
\end{proof}

We have some comments about \eqref{eq:type2sol1}, \eqref{eq:type2sol2}.
By Proposition \ref{prop:A=B}, \eqref{eq:type2sol1} holds
if and only if there exists $\zeta \in \F$ such that $B = A + \zeta I$.
By Proposition \ref{prop:ABvee}, \eqref{eq:type2sol2} holds
if and only if there exists $\zeta \in \F$ such that $B = A^\vee + \zeta I$.
The solutions \eqref{eq:type2sol1}--\eqref{eq:type2sol5} are not mutually exclusive.

\section{Describing the companions for a Leonard pair of type II}
\label{sec:companiontype2}
\ifDRAFT {\rm sec:companiontype2}. \fi

In this section we describe the companions for a Leonard pair of type II.
Throughout this section, Notation \ref{notation:all} is in effect.
Assume that $\Phi$ has type II.
Note that $\Phi'$ has type II.
Let
\begin{align}
& (\delta, \mu, h, \delta^*, \mu^*, h^*, \tau),  
&& 
 (\delta', \mu', h', \delta^*, \mu^*, h^*, \tau')                       \label{eq:type2pseqs}
\end{align}
denote the basic sequence of $\Phi$ and $\Phi'$, respectively.
Assume that $A$, $B$ are compatible, and consider the companion $K=A-B$.
We will give the entries of $K$.
To avoid complicated formulas, 
we assume that each of $\Phi$, $\Phi'$ is reduced,
so that $\delta=0$, $\delta'=0$, $\delta^* = 0$.

For the moment assume that  \eqref{eq:type2sol1} holds.
By the comments below Theorem \ref{thm:type2sol},
there exists $\zeta \in \F$ such that $B = A + \zeta I$.
By this and $\delta=\delta'=0$ we get $B=A$.
So $K=0$.
Next assume that \eqref{eq:type2sol2} holds.
By the comments below Theorem \ref{thm:type2sol},
there exists $\zeta \in \F$ such that $B = A^\vee + \zeta I$.
By this and $\delta=\delta'=0$ we get $B=A^\vee$.
By this and Lemma \ref{lem:bondcompOmega}, 
$K_{i,i} = 2 a_i$ for $0 \leq i \leq d$.
Next we give the $K$ that corresponds to solutions 
\eqref{eq:type2sol3}, \eqref{eq:type2sol4}.

\begin{theorem}      \label{thm:type2ex2}     \samepage
\ifDRAFT {\rm thm:type2ex2}. \fi
The following hold.
\begin{itemize}
\item[\rm (i)]
Assume that \eqref{eq:type2sol3} holds.
Then
\begin{align*}
K_{0,0}  &= - 
    \frac{d \left( 4 h^* \tau + h \big( \mu^{*2} + (d-1)^2 h^{* 2} \big) \right) }
           {2 h^* \big( \mu^* + (d-1) h^* \big)},  
\\
K_{i,i} &= - 
   \frac{ \left( (d-2i) \mu^* + \big( d(d+1)-2i (d-i)  \big) h^* \right)
            \left( 4 h^* \tau + h \big( \mu^{* 2} + (d-1)^2 h^{* 2} \big) \right) }
         {2 h^* \big( \mu^* + (d-2i-1) h^* \big) \big( \mu^* + (d-2i+1) h^* \big)}  &
\\ 
& \hspace{26em}  (1 \leq i \leq d-1),
\\
K_{d,d} &= - 
   \frac{d \left( 4 h^* \tau + h \big( \mu^{* 2} + (d-1)^2 h^{* 2} \big) \right) }
          {2 h^* \big(  - \mu^* + (d-1) h^* \big) }.    
\end{align*}
\item[\rm (ii)]
Assume that \eqref{eq:type2sol4} holds.
Then
\begin{align*}
K_{0,0} &= \frac{ d h \big( \mu^* +(d-1)h^* \big) } { 2 h^* },  
\\
K_{i,i} &= \frac{ h \left( (d-2i) \mu^* + \big( d(d-1)-2i(d-i) \big) h^* \right) } { 2 h^*} &&(1 \leq i \leq d-1),
\\
K_{d,d} &= \frac{ d h \big( - \mu^* + (d-1) h^* \big) } { 2 h^* }.
\end{align*}
\end{itemize}
\end{theorem}

\begin{proof}
Use \eqref{eq:Kaa} and Lemmas \ref{lem:aiparam}, \ref{lem:type2param}
with $\delta = \delta' = 0$.
\end{proof}

Next we give the $K$ that corresponds to solution \eqref{eq:type2sol5}.

\begin{theorem}    \label{thm:type2ex1}    \samepage
\ifDRAFT {\rm thm:type2ex1}. \fi
Assume that \eqref{eq:type2sol5} holds.
Then 
\begin{align*}
    K_{i,i}  &= - \frac{(d-2i)(\tau - \tau') } {\mu^* }
   &&  (0 \leq i \leq d).
\end{align*}
\end{theorem}

\begin{proof}
Use \eqref{eq:Kaa} and Lemmas \ref{lem:aiparam}, \ref{lem:type2param}
with $\delta = \delta' = 0$.
\end{proof}

\section{The parameter arrays of type III$^+$}
\label{sec:parraytype3+}
\ifDRAFT {\rm sec:parraytype3+}. \fi

In this section we describe the parameter arrays of type III$^+$.
We then prove Proposition \ref{prop:RST} for type III$^+$.
Throughout this section, assume that $d \geq 3$.

\begin{lemma}    \label{lem:type3+param}    \samepage
\ifDRAFT {\rm lem:type3+param}. \fi
Assume that $d$ is even and $\text{\rm Char}(\F) \neq 2$.
For a sequence
\begin{equation}
   (\delta, s, h, \delta^*,  s^*, h^*, \tau)                \label{eq:type3+pseq}
\end{equation}
of scalars in $\F$,  define
\begin{align}
\theta_i &=
  \begin{cases}
     \delta+s+h(i-d/2)  & \text{\rm if $i$ is even}, \\
     \delta-s -h(i-d/2) & \text{\rm if $i$ is odd}
  \end{cases}        
  &&  (0 \leq i \leq d),                                         \label{eq:type3+th}
\\
\theta^*_i &=
   \begin{cases}
     \delta^* +s^* +h^*(i-d/2)  &   \text{\rm if $i$ is even}, \\
     \delta^* - s^* -h^*(i-d/2) &  \text{\rm if $i$ is odd}
   \end{cases}  
   &&  (0 \leq i \leq d),                                                 \label{eq:type3+ths}
\end{align}
and for $1 \leq i \leq d$,
\begin{align}
\varphi_i &=
   \begin{cases}
      i(\tau-sh^*-s^*h-hh^*(i-(d+1)/2)) &   \text{\rm if $i$ is even}, \\
      (d-i+1)(\tau+sh^*+s^*h+hh^*(i-(d+1)/2))  & \text{\rm if $i$ is odd},    
   \end{cases}                                                                 \label{eq:type3+vphi}
\\
\phi_i &=
   \begin{cases}
      i(\tau-sh^*+s^*h+hh^*(i-(d+1)/2)) &   \text{\rm if $i$ is even}, \\
      (d-i+1)(\tau+sh^*-s^*h - hh^*(i-(d+1)/2))  & \text{\rm if $i$ is odd}.
   \end{cases}                                                                \label{eq:type3+phi}
\end{align}
Then the sequence
\begin{equation}
   (\{\th_i\}_{i=0}^d; \{\th^*_i\}_{i=0}^d; \{\vphi_i\}_{i=1}^d; \{\phi_i\}_{i=1}^d)        \label{eq:type3+parray}
\end{equation}
is a parameter array over $\F$ that has type III$^+$,
provided that the inequalities in Lemma \ref{lem:classify}{\rm (i),(ii)} hold.
Conversely, assume that the sequence \eqref{eq:type3+parray} is a parameter 
array over $\F$ that has type III$^+$.
Then there exists a unique sequence \eqref{eq:type3+pseq} of scalars in $\F$ that satisfies
\eqref{eq:type3+th}--\eqref{eq:type3+phi}.
\end{lemma}

\begin{proof}
Assume that the inequalities in Lemma \ref{lem:classify}(i),(ii) hold.
Using \eqref{eq:type3+th}--\eqref{eq:type3+phi} we routinely verify
the conditions Lemma \ref{lem:classify}(iii)--(v).
Thus the sequence \eqref{eq:type3+parray} is a parameter array over $\F$.
Evaluating the expression on the left in \eqref{eq:indep} using \eqref{eq:type3+th} we find that
the parameter array \eqref{eq:type3+parray} has fundamental constant $\beta = - 2$. 
So the parameter array \eqref{eq:type3+parray} has type III$^+$.
The last assertion comes from \cite[Theorem 8.1]{NT:balanced}.
\end{proof}

\begin{defi}    \label{def:type3+pseq}    \samepage
\ifDRAFT {\rm def:type3+pseq}. \fi
Referring to Lemma \ref{lem:type3+param},
assume that the sequence \eqref{eq:type3+parray} is a parameter array over $\F$.
We call the scalars  $\delta, s, h, \delta^*,  s^*, h^*, \tau$ 
the {\em basic variables of \eqref{eq:type3+parray}}.
We call the sequence $(\delta, s, h, \delta^*,  s^*, h^*, \tau)$
the {\em basic sequence of \eqref{eq:type3+parray}}.
\end{defi}

\begin{lemma}    \label{lem:type3+condpre}    \samepage
\ifDRAFT {\rm lem:type3+condpre}. \fi
Referring to Lemma \ref{lem:type3+param},
the following hold for $0 \leq i,j \leq d$:
\begin{align*}
\th_i - \th_j &=
 \begin{cases}
   h (i-j)  &  \text{\rm if $i$ is even, $j$ is even},
\\
  2 s + h (i+j-d) &  \text{\rm if $i$ is even, $j$ is odd},
\\
  h (j-i) & \text{\rm if $i$ is odd, $j$ is odd},
 \end{cases}
\\
\th^*_i - \th^*_j &=
 \begin{cases}
   h^* (i-j)  &  \text{\rm if $i$ is even, $j$ is even},
\\
  2 s^* + h^* (i+j-d) &  \text{\rm if $i$ is even, $j$ is odd},
\\
  h^* (j-i) & \text{\rm if $i$ is odd, $j$ is odd}.
 \end{cases}
\end{align*}
\end{lemma}

\begin{proof}
Routine verification using \eqref{eq:type3+th} and \eqref{eq:type3+ths}.
\end{proof}

\begin{lemma}    \label{lem:type3+cond}   \samepage
\ifDRAFT {\rm lem::type3+cond}. \fi
Referring to Lemma \ref{lem:type3+param}, 
the inequalities in Lemma \ref{lem:classify}{\rm (i),(ii)} hold
if and only if
\begin{align}
 & \text{$\text{\rm Char}(\F)$ is equal to $0$ or greater than $d/2$},          \label{eq:type3+paramcond1}
\\
 & h \neq 0,   \qquad h^* \neq 0,                                      \label{eq:type3+paramcond1b}
\\
 & 2 s \neq i h  \qquad\;\;\, \text{\rm if $i$ is odd} \qquad  (1-d \leq i \leq d-1),          \label{eq:type3+paramcond2}
\\
 & 2 s^* \neq i h^*    \qquad \text{\rm if $i$ is odd} \qquad  (1-d \leq i \leq d-1),     \label{eq:type3+paramcond3}
\\
 & \tau \neq s h^* + s^* h + h h^* \big( i - (d+1)/2 \big)  
              \;\; \,   \qquad \text{\rm if $i$ is even}  \qquad (1 \leq i \leq d),       \label{eq:type3+paramcond4a}
\\
 & \tau \neq - s h^* - s^* h - h h^* \big( i - (d+1)/2 \big)  
                  \qquad \text{\rm if $i$ is odd}  \qquad \, (1 \leq i \leq d),       \label{eq:type3+paramcond5a}
\\
 & \tau \neq s h^* - s^* h - h h^* \big( i - (d+1)/2 \big)  
             \;\; \,    \qquad \text{\rm if $i$ is even}  \qquad (1 \leq i \leq d),       \label{eq:type3+paramcond4b}
\\
 & \tau \neq - s h^* + s^* h + h h^* \big( i - (d+1)/2 \big)  
                  \qquad \text{\rm if $i$ is odd}  \qquad \, (1 \leq i \leq d).       \label{eq:type3+paramcond5b}
\end{align}
\end{lemma}

\begin{proof}
Routine verification using \eqref{eq:type3+vphi}, \eqref{eq:type3+phi},
and Lemma \ref{lem:type3+condpre}.
\end{proof}

For the rest of this section, let $\Phi$ denote a Leonard system over $\F$
that has type III$^+$ and parameter array
\begin{equation}
   (\{\th_i\}_{i=0}^d; \{\th^*_i\}_{i=0}^d; \{\vphi_i\}_{i=1}^d; \{\phi_i\}_{i=1}^d).        \label{eq:type3+parray2}
\end{equation}

\begin{defi}     \label{def:type3+pseqLS}    \samepage
\ifDRAFT {\rm def:type3+pseqLS}. \fi
By the {\em basic variables} (resp.\ {\em basic sequence}) {\em of $\Phi$} we mean 
the basic variables (resp.\ basic sequence) of the parameter array \eqref{eq:type3+parray2}.
\end{defi}

For the rest of this section, let
$(\delta, s, h, \delta^*, s^*, h^*, \tau)$
denote the basic sequence of $\Phi$.

\begin{lemma}    \label{lem:type3+paramDown}    \samepage
\ifDRAFT {\rm lem:type3+paramDown}. \fi
In the table below, for each Leonard system in the first column we give
the basic sequence:
\[
\begin{array}{cc}
\text{\rm Leonard system}  &  \text{\rm basic sequence}
\\ \hline
\Phi^\downarrow & (\delta, s, h,  \delta^*, s^*,  - h^*, \tau)  \rule{0mm}{3ex}
\\
\Phi^\Downarrow & (\delta, s,  -h, \delta^*, s^*, h^*,  \tau)   \rule{0mm}{2.7ex}
\\
\Phi^\vee &  (- \delta, -s, -h,  \delta^*, s^*, h^*,  - \tau)   \rule{0mm}{2.7ex}
\end{array}
\]
\end{lemma}

\begin{proof}
Concerning $\Phi^\downarrow$ and $\Phi^\Downarrow$, use 
Lemma \ref{lem:parrayrelative}.
Concerning $\Phi^\vee$, use Lemma \ref{lem:parrayPhivee}.
\end{proof}

\begin{lemma}    \label{lem:type3+affine}    \samepage
\ifDRAFT {\rm lem:type3+affine}. \fi
For scalars $\xi$, $\zeta$, $\xi^*$, $\zeta^*$ in $\F$ with $\xi \xi^* \neq 0$,
consider the Leonard system
\[
   (\xi A + \zeta I; \, \{E_i\}_{i=0}^d; \;  \xi^* A^* + \zeta^* I; \, \, \{E^*_i\}_{i=0}^d).
\]
For this Leonard system the basic sequence is equal to
\[
   ( \xi \delta + \zeta; \;  \xi s, \; \xi h,  \; 
     \xi^* \delta^* + \zeta^*;  \; \xi^* s^*, \; \xi^* h^*, \; \xi \xi^* \tau).
\]
\end{lemma}

\begin{proof}
Use Lemma \ref{lem:affineparam}.
\end{proof}

\begin{corollary}    \label{cor:type3+affine}    \samepage
\ifDRAFT {\rm cor:type3+affine}. \fi
The Leonard system 
\[
 (A- \delta I; \, \{E_i\}_{i=0}^d; \, A^* - \delta^* I; \,  \{E^*_i\}_{i=0}^d)
\]
has basic sequence
$(0, s, h, 0, s^*, h^*, \tau)$.
\end{corollary}

\begin{defi}    \label{def:type3+reduced}    \samepage
\ifDRAFT {\rm def:type3+reduced}. \fi
We say that $\Phi$ is {\em reduced} whenever
$\delta=0$ and $\delta^* = 0$.
\end{defi}

\begin{lemma}    \label{lem:type3+kappa}    \samepage
\ifDRAFT {\rm lem:type3+kappa}. \fi
The variable $\kappa$ for $\Phi$ satisfies
$\kappa = 4 h^2$.
\end{lemma}

\begin{proof}
By Lemma \ref{lem:kappa} and \eqref{eq:type3+th}.
\end{proof}

\begin{proofof}{Proposition \ref{prop:RST}, type III$^+$}
One routinely verifies \eqref{eq:RST} using \eqref{eq:Rtype1}--\eqref{eq:Ttype1}
and Lemmas \ref{lem:type3+param}, \ref{lem:type3+kappa}.
\end{proofof}

Note that Theorem \ref{thm:kappamain} holds for type III$^+$.

\section{A characterization of compatibility in terms of the basic sequence, type III$^+$}
\label{sec:characterizetype3+}
\ifDRAFT {\rm sec:characterizetype3+}. \fi

In this section we characterize the compatibility relation for Leonard pairs
of type III$^+$ in terms of the basic sequence.
Throughout this section, Notation \ref{notation:all} is in effect.
Assume that $\Phi$ has type III$^+$.
Note that $\Phi'$ has type III$^+$.
Let
\begin{align*}
& (\delta, s, h, \delta^*, s^*, h^*, \tau),  
&& 
 (\delta', s', h', \delta^*, s^*, h^*, \tau')
\end{align*}
denote the basic sequence of $\Phi$ and $\Phi'$, respectively.

\begin{theorem}    \label{thm:type3+main}    \samepage
\ifDRAFT {\rm thm:type3+main}. \fi
The matrices $A$, $B$ are compatible if and only if
the following \eqref{eq:type3+cond1}--\eqref{eq:type3+cond3} hold:
\begin{align}
h^2 &= h^{\prime 2},                                 \label{eq:type3+cond1}
\\
(\tau+ s h^*)^2 &= (\tau' + s' h^*)^2,          \label{eq:type3+cond2}
\\
(\tau-  s h^*)^2 & = (\tau' - s' h^*)^2.      \label{eq:type3+cond3}
\end{align}
\end{theorem}

\begin{proof}
We will invoke Theorem \ref{thm:kappamain}.
To do this we investigate the conditions in \eqref{eq:kappakappad}.
By Lemma \ref{lem:type3+kappa} and \eqref{eq:type3+paramcond1},
$\kappa = \kappa'$ if and only if \eqref{eq:type3+cond1} holds.
Using \eqref{eq:type3+vphi}, \eqref{eq:type3+phi} we find that under 
the assumption \eqref{eq:type3+cond1}
the expression
$\vphi_1 \phi_1 - \vphi'_1 \phi'_1$ is equal to
$d^2$ times
\[
  (\tau + s h^*)^2 - (\tau' + s' h^*)^2,
\]
and the expression 
$\vphi_d \phi_d - \vphi'_d \phi'_d$ is equal to $d^2$ times
\[
  (\tau - s h^*)^2 - (\tau' - s' h^*)^2.
\]
By these comments and \eqref{eq:type3+paramcond1}, 
we find that under the assumption \eqref{eq:type3+cond1},
$\vphi_1 \phi_1 = \vphi'_1 \phi'_1$ holds if and only if \eqref{eq:type3+cond2} holds,
and
$\vphi_d \phi_d = \vphi'_d \phi'_d$ if and only if \eqref{eq:type3+cond3} holds.
Now the result follows from Theorem \ref{thm:kappamain}.
\end{proof}

Our next goal is to solve the equations \eqref{eq:type3+cond1}--\eqref{eq:type3+cond3}
for $s'$, $h'$, $\tau'$.

\begin{theorem}   \label{thm:type3+sol}      \samepage
\ifDRAFT {\rm thm:type3+sol}. \fi
The equations \eqref{eq:type3+cond1}--\eqref{eq:type3+cond3} hold if and only if at least one of the
following \eqref{eq:type3+sol1}--\eqref{eq:type3+sol4} holds:
\begin{align}
& s' = s,  \qquad\quad\qquad\; \, \tau' = \tau, \qquad\qquad \;\; h^{\prime 2} = h^2;        \label{eq:type3+sol1}
\\
&s' = -s,  \qquad\qquad\; \;\,  \tau' = - \tau,  \qquad\qquad   h^{\prime 2} = h^2;  \label{eq:type3+sol2}
\\
& s' = \tau/h^*,  \qquad\qquad  \tau' = s h^*,  \quad\qquad \; \; \,  h^{\prime 2} = h^2;             \label{eq:type3+sol3}
\\
& s' = - \tau/h^*,  \quad\qquad\;  \tau' = - s h^*,  \quad\qquad h^{\prime 2} = h^2.       \label{eq:type3+sol4}
\end{align}
\end{theorem}

\begin{proof}
One routinely checks that each of \eqref{eq:type3+sol1}--\eqref{eq:type3+sol4} gives
a solution to \eqref{eq:type3+cond1}--\eqref{eq:type3+cond3}.
Now assume that \eqref{eq:type3+cond1}--\eqref{eq:type3+cond3} hold.
Note by \eqref{eq:type3+paramcond1b} that $h^* \neq 0$.
By \eqref{eq:type3+cond2}, 
\begin{align*}
\tau + s h^* &= \tau' + s' h^*  \qquad \text{ or } \qquad
\tau + s h^* = - \tau' - s' h^*.
\end{align*}
By \eqref{eq:type3+cond3}, 
\begin{align*}
\tau - s h^* &= \tau' - s' h^*  \qquad \text{ or } \qquad
\tau - s h^* = - \tau' + s' h^*.
\end{align*}
By these comments and \eqref{eq:type3+cond1}, 
we get at least one of \eqref{eq:type3+sol1}--\eqref{eq:type3+sol4}.
\end{proof}

We have some comments about \eqref{eq:type3+sol1}, \eqref{eq:type3+sol2}.
By Proposition \ref{prop:A=B}, \eqref{eq:type3+sol1} holds
if and only if there exists $\zeta \in \F$ such that $B = A + \zeta I$.
By Proposition \ref{prop:ABvee}, \eqref{eq:type3+sol2} holds
if and only if there exists $\zeta \in \F$ such that $B = A^\vee + \zeta I$.
The solutions \eqref{eq:type3+sol1}--\eqref{eq:type3+sol4} are not mutually exclusive.

\section{Describing the companions for a Leonard pair of type III$^+$}
\label{sec:companiontype3+}
\ifDRAFT {\rm sec:companiontype3+}. \fi

In this section we describe the companions for a Leonard pair of type III$^+$.
Throughout this section, Notation \ref{notation:all} is in effect.
Assume that $\Phi$ has type III$^+$.
Note that $\Phi'$ has type III$^+$.
Let
\begin{align}
& (\delta, s, h, \delta^*, s^*, h^*, \tau),  
&& 
 (\delta', s', h', \delta^*,  s^*, h^*, \tau')                 \label{eq:type3+pseqs}
\end{align}
denote the basic sequence of $\Phi$ and $\Phi'$, respectively.
Assume that $A$, $B$ are compatible, and consider the companion $K=A-B$.
We will give the entries of $K$.
To avoid complicated formulas, 
we assume that each of $\Phi$, $\Phi'$ is reduced,
so that $\delta=0$, $\delta'=0$, $\delta^* = 0$.

For the moment assume that  \eqref{eq:type3+sol1} holds.
By the comments below Theorem \ref{thm:type3+sol},
there exists $\zeta \in \F$ such that $B = A + \zeta I$.
By this and $\delta=\delta'=0$ we get $B=A$.
So $K=0$.
Next assume that \eqref{eq:type3+sol2} holds.
By the comments below Theorem \ref{thm:type3+sol},
there exists $\zeta \in \F$ such that $B = A^\vee + \zeta I$.
By this and $\delta=\delta'=0$ we get $B=A^\vee$.
By this and Lemma \ref{lem:bondcompOmega}, 
$K_{i,i} = 2 a_i$ for $0 \leq i \leq d$.
We now give the $K$ that corresponds to solutions \eqref{eq:type3+sol3}, \eqref{eq:type3+sol4}.

\begin{theorem}      \label{thm:type3+ex2}     \samepage
\ifDRAFT {\rm thm:type3+ex2}. \fi
The following hold.
\begin{itemize}
\item[\rm (i)]
Assume \eqref{eq:type3+sol3} holds.
Then
\begin{align*}
K_{0,0}  &= s - \tau/h^*,
\\
K_{i,i} &=
 \begin{cases} \displaystyle
   \frac{ ( s  - \tau / h^*) \big( 2 s^* - (d+1) h^* \big) }
          { 2 s^* - (d-2i+1) h^* }
         &  \text{\rm if $i$ is even},
  \\   \displaystyle
  -  \frac{ ( s  - \tau / h^*) \big( 2 s^* - (d+1) h^* \big) }
          { 2 s^* - (d-2i-1) h^* }   
         &  \text{\rm if $i$ is odd}
  \end{cases}             &&    (1 \leq i \leq d).
\end{align*}
\item[\rm (ii)]
Assume that \eqref{eq:type3+sol4} holds.
Then
\begin{align*}
K_{i,i} &=
 \begin{cases} \displaystyle
   \frac{ ( s  + \tau / h^*) \big( 2 s^* + (d+1) h^* \big) }
          { 2 s^* - (d-2i-1) h^* }
         &  \text{\rm if $i$ is even},
  \\   \displaystyle
  -  \frac{ ( s  + \tau / h^*) \big( 2 s^* + (d+1) h^* \big) }
          { 2 s^* - (d-2i+1) h^* }   
         &  \text{\rm if $i$ is odd}
  \end{cases}             &&    (0 \leq i \leq d-1),
\\
K_{d,d} &= s + \tau/h^*.
\end{align*}
\end{itemize}
\end{theorem}

\begin{proof}
Use \eqref{eq:Kaa} and Lemmas \ref{lem:aiparam}, \ref{lem:type3+param}
with $\delta = \delta' = 0$.
\end{proof}

\section{The parameter arrays of type III$^-$}
\label{sec:parraytype3-}
\ifDRAFT {\rm sec:parraytype3-}. \fi

In this section we describe the parameter arrays of type III$^-$.
We then prove Proposition \ref{prop:RST} for type III$^-$.
Throughout this section, assume that $d \geq 3$.

\begin{lemma}    \label{lem:type3-param}    \samepage
\ifDRAFT {\rm lem:type3-param}. \fi
Assume that $d$ is odd and $\text{\rm Char}(\F) \neq 2$.
For a sequence
\begin{equation}
   (\delta, s, h, \delta^*,  s^*, h^*, \tau)                \label{eq:type3-pseq}
\end{equation}
of scalars in $\F$,  define
\begin{align}
\theta_i &=
  \begin{cases}
     \delta+s+h(i-d/2)  & \text{\rm if $i$ is even}, \\
     \delta-s -h(i-d/2) & \text{\rm if $i$ is odd}
  \end{cases}   
        &&   (0 \leq i \leq d),                                       \label{eq:type3-th}
\\
\theta^*_i &=
   \begin{cases}
     \delta^* +s^* +h^*(i-d/2)  &   \text{\rm if $i$ is even}, \\
     \delta^* - s^* -h^*(i-d/2) &  \text{\rm if $i$ is odd}
   \end{cases}                                                 
        &&   ( 0 \leq i \leq d),                               \label{eq:type3-ths}
\end{align}
and for $1 \leq i \leq d$,
\begin{align}    
\varphi_i &=
   \begin{cases}
      h h^* i(d-i+1) &   \text{\rm if $i$ is even}, \\
      \tau- 2 s s^*+i(d-i+1) h h^* - (s h^* + s^* h)(2i- d-1)  & \text{\rm if $i$ is odd},
   \end{cases}                                                  \label{eq:type3-vphi}
\\
\phi_i &=
   \begin{cases}
      h h^* i(d-i+1) &   \text{\rm if $i$ is even}, \\
      \tau + 2s s^* + i(d-i+1) h h^* + (s h^* - s^* h)( 2i-d-1)  & \text{\rm if $i$ is odd}.
   \end{cases}                                                  \label{eq:type3-phi}
\end{align}
Then the sequence
\begin{equation}
   (\{\th_i\}_{i=0}^d; \{\th^*_i\}_{i=0}^d; \{\vphi_i\}_{i=1}^d; \{\phi_i\}_{i=1}^d)        \label{eq:type3-parray}
\end{equation}
is a parameter array over $\F$ that has type III$^-$,
provided that the inequalities in Lemma \ref{lem:classify}{\rm (i),(ii)} hold.
Conversely, assume that the sequence \eqref{eq:type3-parray} is a parameter 
array over $\F$ that has type III$^-$.
Then there exists a unique sequence \eqref{eq:type3-pseq} of scalars in $\F$ that satisfies
\eqref{eq:type3-th}--\eqref{eq:type3-phi}.
\end{lemma}

\begin{proof}
Assume that the inequalities in Lemma \ref{lem:classify}(i),(ii) hold.
Using \eqref{eq:type3-th}--\eqref{eq:type3-phi} we routinely verify
the conditions Lemma \ref{lem:classify}(iii)--(v).
Thus the sequence \eqref{eq:type3-parray} is a parameter array over $\F$.
Evaluating the expression on the left in \eqref{eq:indep} using \eqref{eq:type3-th} we find that
the parameter array \eqref{eq:type3-parray} has fundamental constant $\beta = - 2$. 
So the parameter array \eqref{eq:type3-parray} has type III$^-$.
The last assertion comes from \cite[Theorem 9.1]{NT:balanced}.
\end{proof}

\begin{defi}    \label{def:type3-pseq}    \samepage
\ifDRAFT {\rm def:type3-pseq}. \fi
Referring to Lemma \ref{lem:type3-param},
assume that the sequence \eqref{eq:type3-parray} is a parameter array over $\F$.
We call the scalars $\delta, s, h, \delta^*,  s^*, h^*, \tau$ 
the {\em basic variables of \eqref{eq:type3-parray}}.
We call the sequence $(\delta, s, h, \delta^*,  s^*, h^*, \tau)$
the {\em basic sequence of \eqref{eq:type3-parray}}.
\end{defi}

\begin{lemma}    \label{lem:type3-condpre}    \samepage
\ifDRAFT {\rm lem:type3-condpre}. \fi
Referring to Lemma \ref{lem:type3-param},
the following hold for $0 \leq i,j \leq d$:
\begin{align*}
\th_i - \th_j &=
 \begin{cases}
   h (i-j)  &  \text{\rm if $i$ is even, $j$ is even},
\\
  2 s + h (i+j-d) &  \text{\rm if $i$ is even, $j$ is odd},
\\
  h (j-i) & \text{\rm if $i$ is odd, $j$ is odd},
 \end{cases}
\\
\th^*_i - \th^*_j &=
 \begin{cases}
   h^* (i-j)  &  \text{\rm if $i$ is even, $j$ is even},
\\
  2 s^* + h^* (i+j-d) &  \text{\rm if $i$ is even, $j$ is odd},
\\
  h^* (j-i) & \text{\rm if $i$ is odd, $j$ is odd}.
 \end{cases}
\end{align*}
\end{lemma}

\begin{proof}
Routine verification using \eqref{eq:type3-th} and \eqref{eq:type3-ths}.
\end{proof}

\begin{lemma}    \label{lem:type3-cond}   \samepage
\ifDRAFT {\rm lem::type3-cond}. \fi
Referring to Lemma \ref{lem:type3-param}, 
the inequalities in Lemma \ref{lem:classify}{\rm (i),(ii)} hold
if and only if
{\small
\begin{align}
 & \text{$\text{\rm Char}(\F)$ is equal to $0$ or greater than $(d-1)/2$},          \label{eq:type3-paramcond1}
\\
 & h \neq 0,   \qquad h^* \neq 0,                                      \label{eq:type3-paramcond1b}
\\
 & 2 s \neq i h  \qquad\;\;\, \text{\rm if $i$ is even} \qquad  (1-d \leq i \leq d-1),          \label{eq:type3-paramcond2}
\\
 & 2 s^* \neq i h^*    \qquad \text{\rm if $i$ is even} \qquad  (1-d \leq i \leq d-1),     \label{eq:type3-paramcond3}
\\
 & \tau \neq 2 s s^* - i (d-i+1) h h^* + (s h^* + s^* h) ( 2i - d-1)
          \quad\;\; \text{\rm if $i$ is odd}  \quad (1 \leq i \leq d),       \label{eq:type3-paramcond4}
\\
 &\tau \neq - 2 s s^* - i (d-i+1) h h^* - (s h^* - s^* h) ( 2i - d-1)
          \quad \text{\rm if $i$ is odd}  \quad (1 \leq i \leq d).       \label{eq:type3-paramcond5}
\end{align}
}
\end{lemma}

\begin{proof}
Routine verification using \eqref{eq:type3-vphi}, \eqref{eq:type3-phi},
and Lemma \ref{lem:type3-condpre}.
\end{proof}

For the rest of this section, let $\Phi$ denote a Leonard system over $\F$
that has type III$^-$ and parameter array
\begin{equation}
   (\{\th_i\}_{i=0}^d; \{\th^*_i\}_{i=0}^d; \{\vphi_i\}_{i=1}^d; \{\phi_i\}_{i=1}^d).        \label{eq:type3-parray2}
\end{equation}

\begin{defi}     \label{def:type3-pseqLS}    \samepage
\ifDRAFT {\rm def:type3-pseqLS}. \fi
By the {\em basic variables} (resp.\ {\em basic sequence}) {\em of $\Phi$} we mean 
the basic variables (resp.\ basic sequence) of the parameter array \eqref{eq:type3-parray2}.
\end{defi}

For the rest of this section, let
$(\delta, s, h, \delta^*, s^*, h^*, \tau)$
denote the basic sequence of $\Phi$.

\begin{lemma}    \label{lem:type3-paramDown}    \samepage
\ifDRAFT {\rm lem:type3-paramDown}. \fi
In the table below, for each Leonard system in the first column we give
the basic sequence:
\[
\begin{array}{cc}
\text{\rm Leonard system}  &  \text{\rm basic sequence}
\\ \hline
\Phi^\downarrow & (\delta, s, h, \delta^*, s^*, - h^*, \tau)   \rule{0mm}{3ex}
\\
\Phi^\Downarrow &  (\delta, s, - h, \delta^*, s^*, h^*, \tau)    \rule{0mm}{2.7ex}
\\
\Phi^\vee &  (- \delta, -s, -h, \delta^*, s^*, h^*, - \tau)    \rule{0mm}{2.7ex}
\end{array}
\]
\end{lemma}

\begin{proof}
Concerning $\Phi^\downarrow$ and $\Phi^\Downarrow$, use 
Lemma \ref{lem:parrayrelative}.
Concerning $\Phi^\vee$, use Lemma \ref{lem:parrayPhivee}.
\end{proof}

\begin{lemma}    \label{lem:type3-affine}    \samepage
\ifDRAFT {\rm lem:type3-affine}. \fi
For scalars $\xi$, $\zeta$, $\xi^*$, $\zeta^*$ in $\F$ with $\xi \xi^* \neq 0$,
consider the Leonard system
\[
   (\xi A + \zeta I; \, \{E_i\}_{i=0}^d; \;  \xi^* A^* + \zeta^* I; \, \, \{E^*_i\}_{i=0}^d).
\]
For this Leonard system the basic sequence is equal to
\[
   ( \xi \delta + \zeta; \;  \xi s, \; \xi h,  \; 
     \xi^* \delta^* + \zeta^*;  \; \xi^* s^*, \; \xi^* h^*, \; \xi \xi^* \tau).
\]
\end{lemma}

\begin{proof}
Use Lemma \ref{lem:affineparam}.
\end{proof}

\begin{corollary}    \label{cor:type3-affine}    \samepage
\ifDRAFT {\rm cor:type3-affine}. \fi
The Leonard system 
\[
 (A- \delta I; \, \{E_i\}_{i=0}^d; \, A^* - \delta^* I; \,  \{E^*_i\}_{i=0}^d)
\]
has basic sequence
$(0, s, h, 0, s^*, h^*, \tau)$.
\end{corollary}

\begin{defi}    \label{def:type3-reduced}    \samepage
\ifDRAFT {\rm def:type3-reduced}. \fi
We say that $\Phi$ is {\em reduced} whenever
$\delta=0$ and $\delta^* = 0$.
\end{defi}

\begin{lemma}    \label{lem:type3-kappa}    \samepage
\ifDRAFT {\rm lem:type3-kappa}. \fi
The invariant value $\kappa$ for $\Phi$ satisfies
$\kappa = 4 h^2$.
\end{lemma}

\begin{proof}
By Lemma \ref{lem:kappa} and \eqref{eq:type3-th}.
\end{proof}

\begin{proofof}{Proposition \ref{prop:RST}, type III$^-$}
One routinely verifies \eqref{eq:RST} using \eqref{eq:Rtype1}--\eqref{eq:Ttype1}
and Lemmas \ref{lem:type3-param}, \ref{lem:type3-kappa}.
\end{proofof}

Note that Theorem \ref{thm:kappamain} holds for type III$^-$.

\section{A characterization of compatibility in terms of the basic sequence, type III$^-$}
\label{sec:characterizetype3-}
\ifDRAFT {\rm sec:characterizetype3-}. \fi

In this section we characterize the compatibility relation for Leonard pairs
of type III$^-$ in terms of the basic sequence.
Throughout this section, 
Notation \ref{notation:all} is in effect.
Assume that 
$\Phi$ has type III$^-$.
Note that $\Phi'$ has type III$^-$.
Let
\begin{align*}
& (\delta,  s, h,  \delta^*, s^*, h^*, \tau),  
&& 
 (\delta', s', h', \delta^*,  s^*, h^*, \tau')
\end{align*}
denote the basic sequence of $\Phi$ and $\Phi'$, respectively.

\begin{theorem}    \label{thm:type3-main}    \samepage
\ifDRAFT {\rm thm:type3-main}. \fi
The matrices $A$, $B$ are compatible if and only if
the following \eqref{eq:type3-cond1}--\eqref{eq:type3-cond3} hold:
\begin{align}
h^2 &= h^{\prime 2},                         \label{eq:type3-cond1}
\\
h \tau + 2 h^* s^2 &= h' \tau'  + 2 h^* s^{\prime 2},   \label{eq:type3-cond2}
\\
2 h^* \tau^2 + \left( 4 s^{* 2} + (d+1)^2 h^{* 2} \right) h \tau
&= 2 h^* \tau^{\prime 2} + \left(  4 s^{* 2} + (d+1)^2 h^{* 2} \right) h' \tau'.
                                                \label{eq:type3-cond3}
\end{align}
\end{theorem}

\begin{proof}
We will invoke Theorem \ref{thm:kappamain}.
To do this we investigate the conditions in \eqref{eq:kappakappad}.
By Lemma \ref{lem:type3-kappa} and $\text{\rm Char}(\F) \neq 2$,
$\kappa = \kappa'$ if and only if \eqref{eq:type3-cond1} holds.
Using \eqref{eq:type3-vphi}, \eqref{eq:type3-phi} we find that under 
the assumption \eqref{eq:type3-cond1}
the expression
$\vphi_1 \phi_1 - \vphi'_1 \phi'_1 - \vphi_d \phi_d + \vphi'_d \phi'_d$ is equal to
$4 (d-1) s^*$
times
\[
    h \tau + 2 h^* s^2 - h' \tau' - 2 h^* s^{\prime 2}.
\]
Using \eqref{eq:type3-vphi}, \eqref{eq:type3-phi} we find that under the assumptions
\eqref{eq:type3-cond1} and \eqref{eq:type3-cond2},
the expression $\vphi_1 \phi_1 - \vphi'_1 \phi'_1$ is equal to
\[
\frac{  2 h^* \tau^2 + \left( 4 s^{* 2} + (d+1)^2 h^{* 2} \right) h \tau
- 2 h^* \tau^{\prime 2} - \left( 4 s^{* 2} + (d+1)^2 h^{* 2} \right) h' \tau'}
{2 h^*}.
\]
By these comments and \eqref{eq:type3-paramcond1}, \eqref{eq:type3-paramcond3},
we find that under the assumption \eqref{eq:type3-cond1},
both $\vphi_1 \phi_1 = \vphi'_1 \phi'_1$ and $\vphi_d \phi_d = \vphi'_d \phi'_d$
hold if and only if both \eqref{eq:type3-cond2} and \eqref{eq:type3-cond3} hold.
Now the result follows from Theorem \ref{thm:kappamain}.
\end{proof}

Our next goal is to solve the equations \eqref{eq:type3-cond1}--\eqref{eq:type3-cond3}
for $s'$, $h'$, $\tau'$.

\begin{theorem}   \label{thm:type3-sol}      \samepage
\ifDRAFT {\rm thm:type3-sol}. \fi
The equations \eqref{eq:type3-cond1}--\eqref{eq:type3-cond3} hold if and only if at least one of the
following \eqref{eq:type3-sol1}--\eqref{eq:type3-sol4} holds:
\begin{align}
  & \;   \rule{0mm}{2ex}
   h' = h,   \qquad\qquad  \tau' = \tau, \qquad\qquad s^{\prime 2} = s^2;      \label{eq:type3-sol1}
\\
 & \;    \rule{0mm}{3ex}
h' = -h, \qquad\quad\; \tau' = - \tau,  \qquad\quad\, s^{\prime2} = s^2;   \label{eq:type3-sol2}
\\
&   \rule{0mm}{7ex}
\begin{array}{ll}
h' = h,  \qquad
\tau' = - \tau - 2 h h^* \left( (s^*/h^*)^2 + \big( (d+1)/2 \big)^2 \right),
\\
s^{\prime 2} = s^2 + (h/h^*)\tau 
  + h^2 \left( (s^*/h^*)^2 + \big( (d+1)/2 \big)^2 \right);
\end{array}                                           \label{eq:type3-sol3}
\\  
&   \rule{0mm}{7ex}
\begin{array}{ll}
h' = - h,  \qquad
\tau' = \tau + 2 h h^* \left( (s^*/h^*)^2 + \big( (d+1)/2 \big)^2 \right),
\\
s^{\prime 2} = s^2 + (h/h^*)\tau 
  + h^2 \left( (s^*/h^*)^2 + \big( (d+1)/2 \big)^2 \right).
\end{array}                                            \label{eq:type3-sol4}
\end{align}
\end{theorem}

\begin{proof}
One routinely checks that each of \eqref{eq:type3-sol1}--\eqref{eq:type3-sol4} gives
a solution to \eqref{eq:type3-cond1}--\eqref{eq:type3-cond3}.
Now assume that \eqref{eq:type3-cond1}--\eqref{eq:type3-cond3} hold.
We show that at least one of \eqref{eq:type3-sol1}--\eqref{eq:type3-sol4} holds.

By \eqref{eq:type3-cond1} we have either
$h' = h$ or $h' = - h$.
First assume that $h' = h$.
We may assume that $\tau \neq \tau'$;
otherwise $s^2 = s^{\prime 2}$ by \eqref{eq:type3-cond2}, and so \eqref{eq:type3-sol1} holds.
By \eqref{eq:type3-cond3},
\[
 (\tau - \tau')
 \left(   2 h^* (\tau + \tau') + h  \big( 4 s^{* 2} + (d+1)^2 h^{* 2} \big)  \rule{0mm}{2.3ex} \right) = 0.
\]
By this and $\tau - \tau' \neq 0$,
\begin{equation}
 2 h^* (\tau + \tau') + h  \left( 4 s^{* 2} + (d+1)^2 h^{* 2} \right) = 0.      \label{eq:type3-aux}
\end{equation}
By \eqref{eq:type3-cond2},
\[
  s^{\prime 2} = s^2 + \frac{h (\tau -\tau') }{ 2 h^* }.
\]
By this and \eqref{eq:type3-aux} we get \eqref{eq:type3-sol3}.
Next assume that $h' = - h$.
We may assume that $\tau' \neq - \tau$;
otherwise $s^2 = s^{\prime 2}$ by \eqref{eq:type3-cond2}, and so \eqref{eq:type3-sol2} holds.
By \eqref{eq:type3-cond3},
\[
 (\tau + \tau')
 \left( 2 h^* (\tau - \tau') + h  \big( 4 s^{* 2} + (d+1)^2 h^{* 2} \big)  \rule{0mm}{2.3ex}  \right) = 0.
\]
By this and $\tau + \tau' \neq 0$,
\begin{equation}
 2 h^* (\tau - \tau') + h  \left( 4 s^{* 2} + (d+1)^2 h^{* 2} \right) = 0.      \label{eq:type3-auxb}
\end{equation}
By \eqref{eq:type3-cond2},
\[
  s^{\prime 2} = s^2 + \frac{h (\tau + \tau') }{ 2 h^* }.
\]
By this and \eqref{eq:type3-auxb} we get \eqref{eq:type3-sol4}.
\end{proof}

We have some comments about \eqref{eq:type3-sol1}, \eqref{eq:type3-sol2}.
By Proposition \ref{prop:A=B}, \eqref{eq:type3-sol1} holds
if and only if there exists $\zeta \in \F$ such that $B = A + \zeta I$.
By Proposition \ref{prop:ABvee}, \eqref{eq:type3-sol2} holds
if and only if there exists $\zeta \in \F$ such that $B = A^\vee + \zeta I$.
The solutions \eqref{eq:type3-sol1}--\eqref{eq:type3-sol4} are not mutually exclusive.

\section{Describing the companions for a Leonard pair of type III$^-$}
\label{sec:companiontype3-}
\ifDRAFT {\rm sec:companiontype3-}. \fi

In this section we describe the companions for a Leonard pair of type III$^-$.
Throughout this section, Notation \ref{notation:all} is in effect.
Assume that $\Phi$ has type III$^-$.
Note that $\Phi'$ has type III$^-$.
Let
\begin{align}
& (\delta, s, h, \delta^*, s^*, h^*, \tau),  
&& 
 (\delta', s', h', \delta^*,  s^*, h^*, \tau')                 \label{eq:type3-pseqs}
\end{align}
denote the basic sequence of $\Phi$ and $\Phi'$, respectively.
Assume that $A$, $B$ are compatible, and consider the companion $K=A-B$.
We will give the entries of $K$.
To avoid complicated formulas, 
we assume that each of $\Phi$, $\Phi'$ is reduced,
so that $\delta=0$, $\delta'=0$, $\delta^* = 0$.

For the moment assume that  \eqref{eq:type3-sol1} holds.
By the comments below Theorem \ref{thm:type3-sol},
there exists $\zeta \in \F$ such that $B = A + \zeta I$.
By this and $\delta=\delta'=0$ we get $B=A$.
So $K=0$.
Next assume that \eqref{eq:type3-sol2} holds.
By the comments below Theorem \ref{thm:type3-sol},
there exists $\zeta \in \F$ such that $B = A^\vee + \zeta I$.
By this and $\delta=\delta'=0$ we get $B=A^\vee$.
By this and Lemma \ref{lem:bondcompOmega}, 
$K_{i,i} = 2 a_i$ for $0 \leq i \leq d$.
We now give the $K$ that corresponds to solutions \eqref{eq:type3-sol3}, \eqref{eq:type3-sol4}.

\begin{theorem}      \label{thm:type3-ex2}     \samepage
\ifDRAFT {\rm thm:type3-ex2}. \fi
The following hold.
\begin{itemize}
\item[\rm (i)]
Assume that \eqref{eq:type3-sol3}  holds.
Then
\begin{align*}
K_{i,i} &=
 \begin{cases} \displaystyle
    \frac{ \tau + h h^* \left( (s^*/h^*)^2 + \big( (d+1)/2 \big)^2 \right) }
          { s^* - \big( (d-1)/2 - i \big) h^*    }
         &  \text{\rm if $i$ is even},
  \\   \displaystyle
  -  \frac{ \tau + h h^* \left( (s^*/h^*)^2 + \big( (d+1)/2 \big)^2 \right) }
          { s^* - \big( (d+1)/2 - i \big) h^*    }
         &  \text{\rm if $i$ is odd}
  \end{cases}             &&    (0 \leq i \leq d).
\end{align*}
\item[\rm (ii)]
Assume that \eqref{eq:type3-sol4} holds.
Then
\begin{align*}
K_{0,0} &= - h \left( s^*/h^* + (d+1)/2 \right),
\\ 
K_{i,i} &=
 \begin{cases} \displaystyle
     - \frac{ h h^* \left( (s^*/h^*)^2 - \big( (d+1)/2 \big)^2 \right) }
              { s^* - \left( (d+1)/2 - i \right) h^* }
         &  \text{\rm if $i$ even},
  \\   \displaystyle
      \frac{ h h^* \left( (s^*/h^*)^2 - \big( (d+1)/2 \big)^2 \right) }
              { s^* - \left( (d-1)/2 - i \right) h^* }
         &  \text{\rm if $i$ odd}
  \end{cases}             &&    (1 \leq i \leq d-1),
\\
K_{d,d} &=  h \left( s^*/h^* -  (d+1)/2 \right).
\end{align*}
\end{itemize}
\end{theorem}

\begin{proof}
Use \eqref{eq:Kaa} and Lemmas \ref{lem:aiparam}, \ref{lem:type3-param}
with $\delta = \delta' = 0$.
\end{proof}

\section{The parameter arrays of type IV}
\label{sec:parraytype4}
\ifDRAFT {\rm sec:parraytype4}. \fi

In this section we describe the parameter arrays of type IV.
We then prove Proposition \ref{prop:RST} for type IV.
Note by Lemma \ref{lem:q} that $d=3$ for type IV.

\begin{lemma}    \label{lem:type4param}    \samepage
\ifDRAFT {\rm lem:type4param}. \fi
Assume that $d=3$ and $\text{\rm Char}(\F) = 2$.
For a sequence
\begin{equation}
   (\delta, h, s, \delta^*,  h^*, s^*, r)                \label{eq:type4pseq}
\end{equation}
of scalars in $\F$,  define
\begin{align}
 & \th_0 = \delta,
  \;\;\;\;\;\;\;\, \theta_1 = \delta + h(s+1),
  \quad\qquad\, \theta_2 = \delta + h,
  \qquad \;\, \theta_3 = \delta + h s,                     \label{eq:type4th}
\\
 & \th^*_0 = \delta^*,
  \quad \;\;  \theta^*_1 = \delta^* + h^*(s^*+1), 
  \qquad \theta^*_2 = \delta^* + h^*,
  \quad \;\; \theta^*_3 = \delta^* + h^*s^*,          \label{eq:type4ths}
\\
 & \quad  \varphi_1 = h h^* r,
  \qquad\qquad\qquad\;\;\;  \varphi_2 =  h h^*, 
  \qquad \varphi_3 = hh^*(r+s+s^*),                     \label{eq:type4vphi}
\\
  &  \quad \phi_1 = h h^*(r+s+ s s^*),
  \qquad \phi_2 = h h^*,
  \qquad \phi_3 = h h^*(r+s^* + s s^*).                 \label{eq:type4phi}
\end{align}
Then the sequence
\begin{equation}
   (\{\th_i\}_{i=0}^d; \{\th^*_i\}_{i=0}^d; \{\vphi_i\}_{i=1}^d; \{\phi_i\}_{i=1}^d)        \label{eq:type4parray}
\end{equation}
is a parameter array over $\F$ that has type IV,
provided that the inequalities in Lemma \ref{lem:classify}{\rm (i),(ii)} hold.
Conversely, assume that the sequence \eqref{eq:type4parray} is a parameter 
array over $\F$ that has type IV.
Then there exists a unique sequence \eqref{eq:type4pseq}
of scalars in $\F$ that satisfies
\eqref{eq:type4th}--\eqref{eq:type4phi}.
\end{lemma}

\begin{proof}
Assume that the inequalities in Lemma \ref{lem:classify}(i),(ii) hold.
Using \eqref{eq:type4th}--\eqref{eq:type4phi} we routinely verify
the conditions Lemma \ref{lem:classify}(iii)--(v).
Thus the sequence \eqref{eq:type4parray} is a parameter array over $\F$.
Evaluating the expression on the left in \eqref{eq:indep} using \eqref{eq:type4th} we find that
the parameter array \eqref{eq:type4parray} has fundamental constant $\beta = 2$. 
So the parameter array \eqref{eq:type4parray} has type IV.
The last assertion comes from \cite[Theorem 10.1]{NT:balanced}.
\end{proof}

\begin{defi}    \label{def:type4pseq}    \samepage
\ifDRAFT {\rm def:type4pseq}. \fi
Referring to Lemma \ref{lem:type4param},
assume that the sequence \eqref{eq:type4parray} is a parameter array over $\F$.
We call the scalars $\delta, h, s, \delta^*, h^*, s^*, r$ 
the {\em basic variables of \eqref{eq:type4parray}}.
We call the sequence $(\delta, h, s, \delta^*,  h^*, s^*, r)$ 
the {\em basic sequence of \eqref{eq:type4parray}}.
\end{defi}

\begin{lemma}    \label{lem:type4cond}   \samepage
\ifDRAFT {\rm lem::type4cond}. \fi
Referring to Lemma \ref{lem:type4param}, 
the inequalities in Lemma \ref{lem:classify}{\rm (i),(ii)} hold
if and only if
\begin{align}
& h \neq 0,  \qquad
 s \neq 0,    \qquad 
 s+1 \neq 0,  \qquad 
 h^* \neq 0, \qquad
 s^* \neq 0,  \qquad
 s^* + 1 \neq 0,                             \label{eq:type4paramcond1}
\\
& r \neq 0,  \qquad\;\;
 r+s+s^* \neq 0,  \qquad\;\;
 r + s  + s s^* \neq 0,   \qquad\;\;\,
 r+s^* + s s^* \neq 0.                      \label{eq:type4paramcond2}
\end{align}
\end{lemma}

\begin{proof}
Routine verification using \eqref{eq:type4th}--\eqref{eq:type4phi}.
\end{proof}

For the rest of this section, let $\Phi$ denote a Leonard system over $\F$
that has type IV and parameter array
\begin{equation}
   (\{\th_i\}_{i=0}^d; \{\th^*_i\}_{i=0}^d; \{\vphi_i\}_{i=1}^d; \{\phi_i\}_{i=1}^d).        \label{eq:type4parray2}
\end{equation}

\begin{defi}     \label{def:type4pseqLS}    \samepage
\ifDRAFT {\rm def:type4pseqLS}. \fi
By the {\em basic variables} (resp.\ {\em basic sequence}) {\em of $\Phi$} we mean 
the basic variables (resp.\ basic sequence) of the parameter array \eqref{eq:type4parray2}.
\end{defi}

For the rest of this section, let
$(\delta, h, s, \delta^*, h^*, s^*, r)$
denote the basic sequence of $\Phi$.

\begin{lemma}    \label{lem:type4paramDown}    \samepage
\ifDRAFT {\rm lem:type4paramDown}. \fi
In the table below,
for each Leonard system in the first column we give the basic sequence:
\[
\begin{array}{cc}
\text{\rm Leonard system}  &  \text{\rm basic sequence}
\\ \hline
\Phi^\downarrow & (\delta, s, h, \delta^* + h^* s^* , s^*, h^*, r + s^* + s s^*)   \rule{0mm}{3ex}
\\
\Phi^\Downarrow &  (\delta + h s, s, h, \delta^*, s^*, h^*, r + s + s s^*)    \rule{0mm}{2.7ex}
\\
\Phi^\vee &  (\delta, s, h, \delta^*, s^*, h^*, r)    \rule{0mm}{2.7ex}
\end{array}
\]
\end{lemma}

\begin{proof}
Concerning $\Phi^\downarrow$ and $\Phi^\Downarrow$, use 
Lemma \ref{lem:parrayrelative}.
Concerning $\Phi^\vee$, use Lemma \ref{lem:parrayPhivee}.
\end{proof}

\begin{lemma}   \label{lem:type4affine}    \samepage
\ifDRAFT {\rm lem:type4affine}. \fi
For scalars $\xi$, $\zeta$, $\xi^*$, $\zeta^*$ in $\F$
with $\xi \xi^* \neq 0$,
consider the Leonard system
\[
  ( \xi A + \zeta I; \{E_i\}_{i=0}^d; \xi^* A^* + \zeta^* I; \{E^*_i\}_{i=0}^d).
\]
For this Leonard system the basic sequence is equal to
\[
  (\xi \delta + \zeta, \xi h, s, \xi^* \delta^* + \zeta^*, \xi^* h^*, s^*, r).
\]
\end{lemma}

\begin{proof}
Use Lemma \ref{lem:affineparam}.
\end{proof}

\begin{lemma}    \label{lem:type4kappa}    \samepage
\ifDRAFT {\rm lem:type4kappa}. \fi
The invariant value $\kappa$ for $\Phi$ satisfies
$\kappa = h^2$.
\end{lemma}

\begin{proof}
By Lemma \ref{lem:kappa} and \eqref{eq:type4th}.
\end{proof}

\begin{proofof}{Proposition \ref{prop:RST}, type IV}
One routinely verifies \eqref{eq:RST} using \eqref{eq:Rtype1}--\eqref{eq:Ttype1}
and Lemmas \ref{lem:type4param}, \ref{lem:type4kappa}.
\end{proofof}

Note that Theorem \ref{thm:kappamain} holds for type IV.

\section{A characterization of compatibility in terms of the basic sequence, type IV}
\label{sec:characterizetype4}
\ifDRAFT {\rm sec:characterizetype4}. \fi

In this section we characterize the compatibility relation for Leonard pairs
of type IV in terms of the basic sequence.
Throughout this section, 
Notation \ref{notation:all} is in effect.
Assume that 
$\Phi$ has type IV.
Note that $\Phi'$ has type IV.
Let
\begin{align*}
& (\delta, h, s, \delta^*, h^*, s^*, r)
&& 
 (\delta', h', s', \delta^*, h^*, s^*, r')
\end{align*}
denote the basic sequence of $\Phi$ and $\Phi'$, respectively.

\begin{theorem}    \label{thm:type4main}    \samepage
\ifDRAFT {\rm thm:type4main}. \fi
The matrices $A$ and $B$ are compatible if and only if
the following \eqref{eq:type4cond1}--\eqref{eq:type4cond3} hold:
\begin{align}
 h &= h',                                        \label{eq:type4cond1}
\\
 s (1 + s + s^*) &= s' (1 + s' + s^*),               \label{eq:type4cond2}
\\
 r (r + s + s s^*) &= r' (r' + s' + s' s^*).     \label{eq:type4cond3}
\end{align}
\end{theorem}

\begin{proof}
We will invoke Theorem \ref{thm:kappamain}.
To do this we investigate the conditions in \eqref{eq:kappakappad}.
By Lemma \ref{lem:type4kappa} and since $\text{\rm Char}(\F)=2$,
$\kappa = \kappa'$ if and only if \eqref{eq:type4cond1} holds.
Using \eqref{eq:type4vphi}, \eqref{eq:type4phi} we find that under 
the assumption \eqref{eq:type4cond1}
the expression
$\vphi_1 \phi_1 - \vphi'_1 \phi'_1$ is equal to
$h^2 h^{* 2}$
times
\[
   r ( r + s + s s^*)  - r' (r' + s' + s' s^*),
\]
and the expression $\vphi_3 \phi_3 - \vphi'_3 \phi'_3$ is equal to
$h^2 h^{* 2}$
times
\[
 (r + s + s^*) (r + s^* + s s^*) - (r' + s' + s^*)(r' + s^* + s' s^*).
\]
By these comments and \eqref{eq:type4paramcond1},
we find that under the assumption \eqref{eq:type4cond1},
$\vphi_1 \phi_1 = \vphi'_1 \phi'_1$ holds if and only if \eqref{eq:type4cond2} holds,
and 
$\vphi_d \phi_d = \vphi'_d \phi'_d$  holds if and only if \eqref{eq:type4cond3} holds.
Now the result follows from Theorem \ref{thm:kappamain}.
\end{proof}

Our next goal is solve the equations \eqref{eq:type4cond1}--\eqref{eq:type4cond3}
for $h'$, $s'$, $r'$.

\begin{theorem}     \label{thm:type4sol}    \samepage
\ifDRAFT {\rm thm:type4sol}. \fi
The equations \eqref{eq:type4cond1}--\eqref{eq:type4cond3} hold if and only if
at least one of the following \eqref{eq:type4sol1}, \eqref{eq:type4sol3} holds:
\begin{align}
& h' = h, \qquad s' = s, \qquad \qquad \qquad
  r' = r \quad \text{or} \quad  r' = r + s + s s^*,    \label{eq:type4sol1}
\\
&h' = h,  \qquad 
  s' =  1+s+s^*,      \qquad
  \frac{r+r'}{1+s^*} + \frac{r' (1+s^*) } { r + r'} = s,
 \qquad r' + r \neq 0.       \label{eq:type4sol3}
\end{align}
\end{theorem}

\begin{proof}
Recall that $\text{\rm Char}(\F)=2$.
One routinely checks that each of \eqref{eq:type4sol1}, \eqref{eq:type4sol3} gives
a solution to \eqref{eq:type4cond1}--\eqref{eq:type4cond3}.
Now assume that \eqref{eq:type4cond1}--\eqref{eq:type4cond3} hold.
We show that \eqref{eq:type4sol1} or \eqref{eq:type4sol3} holds.

By \eqref{eq:type4cond1}, $h' = h$.
Using \eqref{eq:type4cond2} we find that
\[
 (s' + s)(s' + 1 + s + s^*) = 0.
\]
So either $s' =s$ or $s' = 1+ s + s^*$.
First assume that $s' = s$.
By \eqref{eq:type4cond3} with $s'=s$,
\[
  (r + r' ) ( r + r' + s + s s^*) = 0.
\]
So either $r' = r$ or $r' = r + s + s s^*$.
Thus \eqref{eq:type4sol1} holds.
Next assume that $s' = 1 + s + s^*$.
In \eqref{eq:type4cond3}, eliminate $s'$ to get
\[
 r (r + s + s s^*) + r' \big( r' + (1+s+s^*)(1+s^*) \big) = 0.
\]
In this equation, rearrange terms to get
\[
 (r+r')^2 + s (r+r')(1+s^*) + r' (1+s^*)^2 = 0.
\]
We have $r+r' \neq 0$; otherwise $r' (1+s^*)^2 = 0$, contradicting Lemma \ref{lem:type4cond}.
In the above equation, multiply each side by $(r+r')^{-1} (1+s^*)^{-1}$ to get
the  third equation in \eqref{eq:type4sol3}.
By these comments \eqref{eq:type4sol3} holds.
\end{proof}

We have a comment about \eqref{eq:type4sol1}.
By Proposition \ref{prop:A=B}, solution \eqref{eq:type4sol1} holds if and only if
there exists $\zeta \in \F$ such that $B = A + \zeta I$.
In this case $\zeta = \delta' - \delta$ in view of Lemma \ref{lem:type4affine}.

\section{Describing the companions for a Leonard pair of type IV}
\label{sec:companiontype4}
\ifDRAFT {\rm sec:companiontype4}. \fi

In this section we describe the companions for a Leonard pair of type IV.
Throughout this section, Notation \ref{notation:all} is in effect.
Assume that $\Phi$ has type IV.
Note that $\Phi'$ has type IV.
Let
\begin{align}
& (\delta, h, s, \delta^*, h^*, s^*, r)
&& 
 (\delta', h', s', \delta^*, h^*, s^*, r')                  \label{eq:type4pseqs}
\end{align}
denote the basic sequence of $\Phi$ and $\Phi'$, respectively.
Assume that $A$, $B$ are compatible, and consider the companion $K=A-B$.
We will give the entries of $K$.

For the moment assume that \eqref{eq:type4sol1} holds.
By the comment below  Theorem \ref{thm:type4sol},
we have $B = A + (\delta' - \delta) I$,
so $K = (\delta - \delta') I$.
We now give the $K$ that corresponds to solution \eqref{eq:type4sol3}.

\begin{theorem}     \label{thm:type4K}    \samepage
\ifDRAFT {\rm thm:type4K}. \fi
Assume that \eqref{eq:type4sol3} holds.
Then
\begin{align*}
K_{0,0} &= \delta - \delta' + \frac{ h (r+r') } { s^* + 1},
&
K_{1,1} &= \delta - \delta' + h \left( 1 + s^* + \frac{ r+r'} { s^* + 1} \right),
\\
K_{2,2}  &= \delta - \delta' + h \left( 1 + \frac{ r+r'} { s^* + 1} \right),
&
K_{3,3} &= \delta - \delta' + h \left( s^* + \frac{ r+r'} { s^* + 1} \right).
\end{align*}
\end{theorem}

\begin{proof}
Use \eqref{eq:Kaa} and Lemmas \ref{lem:aiparam}, \ref{lem:type4param}.
\end{proof}

{
\small

\bigskip\bigskip\noindent
Kazumasa Nomura\\
Tokyo Medical and Dental University\\
Kohnodai, Ichikawa, 272-0827 Japan\\
email: knomura@pop11.odn.ne.jp

\bigskip\noindent
Paul Terwilliger\\
Department of Mathematics\\
University of Wisconsin\\
480 Lincoln Drive\\ 
Madison, Wisconsin, 53706 USA\\
email: terwilli@math.wisc.edu

\bigskip\noindent
{\bf Keywords.}
Leonard system,
parameter array,
compatibility,
companion,
bond relation.
\\
\noindent
{\bf 2020 Mathematics Subject Classification.} 
05E30,  15A21, 15B10.

\end{document}